\documentclass[12pt]{amsart}
\pdfoutput=1 
\usepackage{hyperref}
\usepackage{amsmath,amsfonts,amssymb,verbatim}
\usepackage{graphicx}
\usepackage[english,italian]{babel}
\usepackage[justification=centering]{caption}

\newtheorem{theorem}{Theorem}[section]
\newtheorem{corollary}[theorem]{Corollary}
\newtheorem{proposition}[theorem]{Proposition}
\newtheorem{lemma}[theorem]{Lemma}
\newtheorem{definition}[theorem]{Definition}
\newtheorem{remark}[theorem]{Remark}

\numberwithin{equation}{section}

\def \bC {\mathbb C}
\def \bD {\mathbb D}

\def \bH {\mathbb H}

\def \bN {\mathbb N}
\def \bO {\mathbb O}

\def \bR {\mathbb R}
\def \bS {\mathbb S}
\def \bR {\mathbb R}
\def \bR {\mathbb R}

\def \bZ {\mathbb Z}

\def \cD {\mathcal D}

\def \cF {\mathcal F}
\def \cG {\mathcal G}
\def \cH {\mathcal H}

\def \cL {\mathcal L}

\def \cP {\mathcal P}
\def \cQ {\mathcal Q}
\def \cR {\mathcal R}
\def \cS {\mathcal S}
\def \cR {\mathcal R}
\def \cR {\mathcal R}

\def \fa {\mathfrak a}

\def \fh {\mathfrak h}

\def \fk {\mathfrak k}
\def \fl {\mathfrak l}

\def \fn {\mathfrak n}
\def \fo {\mathfrak o}
\def \fp {\mathfrak p}

\def \fr {\mathfrak r}
\def \fs {\mathfrak s}
\def \ft {\mathfrak t}
\def \fu {\mathfrak u}
\def \fv {\mathfrak v}

\def \fz {\mathfrak z}

\def \fS {\mathfrak S}

\def \fU {\mathfrak U}

\def \fX {\mathfrak X}

\def \su {\mathfrak {su}}
\def \RE {\text{\rm Re}\,}
\def \IM {\text{\rm Im}\,}
\def \al {\alpha}
\def \la {\lambda}
\def \ph {\varphi}
\def \del {\delta}

\def \lan {\langle}
\def \ran {\rangle}
\def \de {\partial}
\def \trans{\,{}^t\!}
\def \half{\frac12}
\def \inv{^{-1}}

\def \supp {\text{\rm supp\,}}

\def \dim {\text{\rm dim\,}}
\def \span {\text{\rm span\,}}

\def \tr {\text{\rm tr\,}}

\newcommand{\gt}{\mathfrak}

\begin{document}

\selectlanguage{english}

\title[]
{Nilpotent Gelfand pairs\\ and spherical transforms of Schwartz functions\\ III. Isomorphisms between Schwartz spaces\\ under Vinberg's condition}

\author[]
{V\'eronique Fischer, Fulvio Ricci, Oksana Yakimova}

\address {Universit\`a degli Studi di Padova\\DMMMSA\\ via Trieste 63\\ 35121 Padova\\ Italy}
\email{fischer@dmsa.unipd.it}

\address {Scuola Normale Superiore\\ Piazza dei Cavalieri 7\\ 56126 Pisa\\ Italy}
\email{fricci@sns.it}

\address {Mathematisches Institut \\
Friedrich-Schiller-Universit\"at Jena \\ 07737 Jena \\ Germany}
\email{oksana.yakimova@uni-jena.de}



\begin{abstract}
 Let$(N,K)$ be a nilpotent Gelfand pair, i.e., $N$ is a nilpotent Lie group, $K$ a compact group of automorphisms of $N$, and the algebra $\bD(N)^K$ of left-invariant and $K$-invariant differential operators on $N$ is commutative. In these hypotheses, $N$ is necessarily of step at most two.
We say that $(N,K)$ satisfies Vinberg's condition if $K$ acts irreducibly on $\fn/[\fn,\fn]$, where $\fn={\rm Lie}(N)$.

Fixing a  system $\cD$ of $d$ formally self-adjoint generators of $\bD(N)^K$, the Gelfand spectrum of the commutative convolution algebra $L^1(N)^K$ can be canonically identified with a closed subset $\Sigma_\cD$ of $\bR^d$.
We prove that, on a nilpotent Gelfand pair satisfying Vinberg's condition, the spherical transform $\cG: L^1(N)^K\longmapsto C_0(\Sigma_\cD)$ 
 establishes an isomorphism from the space $\cS(N)^K$ of $K$-invariant Schwartz functions on $N$ and the space $\cS(\Sigma_\cD)$ of restrictions to $\Sigma_\cD$ of functions in $\cS(\bR^d)$.
\end{abstract}

\maketitle

\makeatletter
\renewcommand\l@subsection{\@tocline{2}{0pt}{3pc}{5pc}{}}
\makeatother

\tableofcontents

\section{Introduction}

\bigskip

Let $N$ be a connected, simply connected nilpotent Lie group and $K$  a compact group of automorphisms of $N$.
We say that $(N,K)$ is a {\it nilpotent Gelfand pair} (n.G.p. in short)\footnote{When dealing with specific pairs, we will find it convenient to identify $N$ with its Lie algebra $\fn$ and write $(\fn,K)$ instead of $(N,K)$.} if either of the following  equivalent condition is satisfied:
\begin{enumerate}
\item[(i)] the convolution algebra $L^1(N)^K$ of $K$-invariant integrable functions on $N$ is commutative;
\item[(ii)] the algebra $\bD(N)^K$ of left-invariant and $K$-invariant differential operators on $N$ is commutative.
\end{enumerate}

 According to the common terminology, this is the same as saying that $(K\ltimes N,K)$ is a Gelfand pair. The expression ``commutative nilmanifold'' is used for $(K\ltimes N)/K$ in \cite{W}.

The relevance of  nilpotent Gelfand pairs in the class of general  Gelfand pairs is emphasized by
Vinberg's structure theorem \cite[Th. 5]{V1}. 

 \medskip

According to the Gelfand theory of commutative Banach algebras, harmonic analysis on Gelfand pairs is based on the notions of spherical function and spherical transform \cite{Fa}, \cite[Ch. IV]{He2}. For nilpotent pairs, spherical functions can be defined as the joint $K$-invariant eigenfunctions $\ph$ of all operators in $\bD(N)^K$ which take value 1 at the identity. The spherical transform of a function $F\in L^1(N)^K$ is
\begin{equation}\label{transform}
\cG F(\ph)=\int_N F(x)\ph(x\inv)\,dx\ ,
\end{equation}
defined on the {\it Gelfand spectrum} of the pair, $\Sigma=\Sigma(N,K)$, i.e., the space of bounded spherical functions endowed with the compact-open topology.  Then
$$
\cG:L^1(N)^K\longrightarrow C_0(\Sigma)
$$
and is continuous. 

The Gelfand spectrum $\Sigma$ admits natural embedding in Euclidean spaces. Let
$$
\cD=(D_1,\dots,D_d)\ ,
$$ 
be a $d$-tuple of essentially self-adjoint operators which generate $\bD(N)^K$ as an algebra. Every bounded spherical function $\ph$ is identified by the $d$-tuple $\xi=\xi(\ph)=\big(\xi_1(\ph),\dots,\xi_d(\ph)\big)$ of eigenvalues of $\ph$ relative to $D_1,\dots,D_d$ respectively.
The $d$-tuples $\xi(\ph)$ form a closed subset $\Sigma_\cD$ of $\bR^d$ which is homeomorphic to $\Sigma$ \cite{FeRu}. Hence the spherical transform $\cG F$ in \eqref{transform} can be viewed as a function on $\Sigma_\cD$.

In the case where $N=\bR^n$ and $K$ is trivial, $\bD(\bR^n)^K$ is the algebra of all constant coefficient differential operators, and the bounded spherical functions are the unitary characters $\ph_\la(x)=e^{i\la\cdot x}$, for $\la\in\bR^n$. Taking
$$
\cD=\big(i\inv\de_{x_1},\dots,i\inv\de_{x_n}\big)\ ,
$$
we have $\xi(\ph_\la)=\la$, so that  $\Sigma_\cD=\bR^n$, and $\cG F=\hat F$ is the ordinary Fourier transform.

It has been conjectured in~\cite{FR} and \cite{FRY1}, that the invariance under Fourier transform of the Schwartz space $\cS(\bR^n)$, a fundamental fact in Fourier analysis, has an analogue on nilpotent Gelfand pairs, in the sense that the spherical transform gives a bijective correspondence between $K$-invariant Schwartz functions on $N$ and restrictions to $\Sigma_\cD$ of Schwartz functions on $\bR^d$.

 To make the statement precise, denote by $\cS(N)^K$ the space of $K$-invariant Schwartz function on $N$ and by
$$
\cS(\Sigma_\cD)\overset{\rm def}=\cS(\bR^d)/\{f:f_{|{\Sigma_\cD}}=0\}
$$
the space of restrictions to $\Sigma_\cD$ of Schwartz functions on $\bR^d$, with the quotient topology. 
The conjectured property, for any n.G.p. $(N,K)$, is as follows:
\begin{equation*}
\text{\it The spherical transform $\cG$ maps the space $\cS(N)^K$ isomorphically onto $\cS(\Sigma_\cD)$.}
\tag{S}
\end{equation*}

The problem is well posed because the answer does not depend on the choice of~$\cD$ \cite{ADR2} and~\cite{FR}.
\smallskip

Property (S) has been proved to hold in several cases.
For ``abelian pairs'', i.e., with $N=\bR^n$ and $K\subset{\rm GL}_n(\bR)$ compact, it has been shown in \cite{ADR2} that Property (S) follows from G.~Schwarz's extension  \cite{Schw}  of Whitney's theorem \cite{Wh} to general linear actions of compact groups on $\bR^n$.

For nonabelian $N$, Property (S) has been proved in the following cases:
\begin{enumerate}
\item[(i)]  pairs in which $N$ is a Heisenberg group or a complexified Heisenberg group \cite{ADR1, ADR2};
\item[(ii)]   the pair $(\bH^n\oplus\IM\bH,{\rm Sp}_n)$ \cite{ADR1};
\item[(iii)]  ``rank-one'' pairs, where $[\fn,\fn]=\fz$, the centre of $\fn$, and the $K$-orbits in $\fz$ are full spheres \cite{FR, FRY1}.
\end{enumerate}
\smallskip

Part of the statement is a matter of functional calculus on Rockland operators on graded groups (i.e., with a graded Lie algebra). It was proved in \cite{H} that if $L$ is a Rockland operator and $g$ is a Schwartz function on the line, then the operator $g(L)$ is given by convolution with a Schwartz kernel. This statement has been later extended to commuting families of $d$ Rockland operators and  $g\in\cS(\bR^d)$, in \cite{Ven} in a special case, and in \cite[Th. 5.2]{ADR2}  in general.  

Since on any n.G.p. we always have a system $\cD$ consisting of Rockland operators \cite{ADR2} and 
$$
g(D_1,\dots,D_d)f=f*K\ \Longleftrightarrow\ \cG K=g_{|_{\Sigma_\cD}}\ ,
$$
this has the following consequence.

\begin{theorem}[\cite{ADR2, FR}]\label{hulanicki}
Let $(N,K)$ be a nilpotent Gelfand pair, and $\cD$, $\Sigma_\cD\subset\bR^d$  as  above. Given any Schwartz function $g$ on $\bR^d$, there is a $K$-invariant Schwartz function $F$ on $N$, depending continuously on $m$, such that $\cG F=g_{|_{\Sigma_\cD}}$.
\end{theorem}

In other words, we always have the continuous inclusion
$$
\cG\big(\cS(N)^K\big)\supseteq \cS(\Sigma_\cD)\ .
$$

So the proof of Property (S)  reduces to proving the opposite inclusion, i.e., that {\it the spherical transform of any function in $\cS(N)^K$ admits a Schwartz extension to $\bR^d$}.
\medskip

For all pairs with nonabelian $N$ studied in \cite{ADR1, ADR2, FR, FRY1, FRY2}, the proof contains a bootstrapping argument (see below), which makes the validity of Property (S) at a given stage a necessary requirement for proving it in more complex or general situations. In a more systematic way, the results of this paper are proved via an inductive procedure, which involves a similar bootstrapping, and also relies on the validity of Property (S) in the above mentioned cases (i)-(iii).

It turns out very useful in this process to have at hand a classification of all nilpotent Gelfand pairs. Their knowledge makes it possible to develop a general strategy of proof  and simplify the most technical parts.

A first classification of nilpotent Gelfand pairs was obtained by E.\,Vinberg in \cite{V1,V2}, under the following assumption, which we call {\it Vinberg's condition}: 
\begin{equation*}
\text{\it $K$ acts irreducibly on $\fn/[\fn,\fn]$}.
\tag{V}
\end{equation*}

The list of all pairs satisfying Vinberg's condition is in \cite[Table 3]{V1} (with an inaccuracy corrected in \cite{Y2}). The  classification of all nilpotent Gelfand pairs was completed by O.~Yakimova \cite{Y1, Y2}, see also \cite[Ch. 13, 15]{W}.

Vinberg's list can also be found in the Appendix of \cite{FRY1}, where families of fundamental invariants are obtained for each case.

We can state now our main theorem.

\begin{theorem}\label{main}
Property (S) holds for all nilpotent Gelfand pairs satisfying condition (V).
\end{theorem}

In the course of the proof, it will be sufficient to limit our analysis to pairs in a much more restricted list. In fact, we can disregard the pairs in (i)-(iii) above, as well as many others according to the following principles: 
\begin{enumerate} 
\item[(a)] if Property (S) holds for $(N,K)$, and $K$ is normal in a larger compact group $K^\#$ of automorphisms of $N$, then it  also holds for  $(N,K^\#)$ ({\it normal extension});
\item[(b)] if Property (S) holds for $(N,K)$, and the centre $\fz$ of $\fn$ has a nontrivial proper $K$-invariant subspace $\fs$,  then it  also holds for $(N/\exp\fs,K)$ ({\it central reduction}).
\end{enumerate}

The resulting reduced list of pairs to work on is given in Table \ref{vinberg}. 

We also mention here a third principle which will be used in the course of the proof:
\begin{enumerate} 
\item[(c)] if Property (S) holds for two n.G.p., $(N_1,K_1)$ and $(N_2,K_2)$, it also holds for the {\it product pair} $(N_1\times N_2,K_1\times K_2)$.
\end{enumerate}

The proof of (a) is contained in \cite{FRY1}, and (b), (c) will be proved in Section \ref{section-reductions}.
\medskip

 In order to present the main ideas in this paper, we first sketch the scheme used in the previous papers \cite{ADR1, ADR2, FR, FRY1, FRY2} on this subject. 

Inside $\Sigma_\cD$ one can identify an open ``regular'' set and a complementary ``singular'' set. 
The regular set has the property that any point $\xi$ in it has a neighbourhood $U_\xi$ which is homeomorphic to an open neighbourhood $V_\eta$ of some point in the Gelfand spectrum of another n.G.p., which is  ``simpler'' in the sense that it is already known to satisfy Property (S). More precisely, there is a natural homeomorphism which induces, by composition, a correspondence between $C^\infty$-functions on $U_\xi$ and $C^\infty$-functions on $V_\eta$.
Using a partition of unity, this argument is sufficient to imply the existence of Schwartz extensions to $\bR^d$ when the spherical transform $\cG F$ of a function $F\in\cS(N)^K$ vanishes of infinite order on the singular set. This is the ``bootstrapping'' part of the argument.

 Such reduction to simpler pairs is not possible on the singular set.  However, the singular set  itself is identified with the Gelfand spectrum of a n.G.p. for which Property (S) is known to hold. At this point, a Hadamard-type formula for $K$-invariant smooth functions on $N$ produces, for any given $F\in\cS(N)^K$, a Whitney jet of infinite order on the singular set, i.e., it determines the derivatives that any smooth extension of $\cG F$ must have on the singular set.  By Whitney extension theorem \cite{Wh2}, there is a Schwartz function $g$ on $\bR^d$ with the prescribed derivatives on the singular set. Applying Theorem \ref{hulanicki}, we find a function $G\in\cS(N)^K$ having $g$ as its spherical transform. Since the difference $\cG (F-G)$ vanishes of infinite order on the singular set,  the bootstrapping argument allows to conclude the proof.
\vskip.2cm

One novelty of this paper is that the notions of regular and singular set must be refined by taking into account the ``higher-rank'' nature of the action of $K$ on the centre $\fz$ of $\fn$, which produces different levels of singularity of points in $\Sigma_\cD$. 

Taking into account that  in every n.G.p. the group $N$ has step at most two \cite{BJR90}, and that $[\fn,\fn]=\fz$ for all pairs in Vinberg's list, we can associate to each bounded spherical function, i.e., to each point in the Gelfand spectrum, a conjugacy class, modulo the action of $K$, of irreducible unitary representations of $N$, and hence a $K$-orbit in $\fz^*\cong\fz$ (cf. Section~\ref{section-spherical}). 

Let $\ph$ be a bounded spherical function and $t$ any point in the corresponding orbit $O_\ph\subset\fz$. Denoting by $N_t$ the quotient group of $N$ with Lie algebra $\fn/T_tO_\ph$ and by $K_t$  the stabilizer of $t$ in $K$,  $\ph$ projects to $N_t$ as a spherical function $\ph^t$ for the n.G.p. $(N_t,K_t)$.

The {\it quotient pair} $(N_t,K_t)$ of $(N,K)$ measures the level of singularity of $\ph$, or of $\xi(\ph)$ as a point of $\Sigma_\cD$. The highest level of singularity occurs when $t$ is fixed by $K$. In this case $(N_t,K_t)=(N,K)$. In all other cases, $(N_t,K_t)$ is a {\it proper} quotient pair.

For instance, consider  the pairs at line 2 of Table \ref{vinberg}, where
\begin{itemize}
\item $\fn=\bC^n\oplus\fu_n$ with Lie bracket $\big[(v,z),(v',z')\big]=\big(0,v{v'}^*-v'v^*\big)$,
\item $K= {\rm U}_n$ acts on $\fn$ by $k\cdot(v,z)=(kv,kzk^*)$.
\end{itemize}

Then the action of $K$ on $\fz=\fu_n$ is the adjoint action. In this case, our notion of singularity of a spherical function matches with the notion of singularity of a point in a Cartan subalgebra of $\fu_n$, with the level of singularity measured by the set of positive roots annihilating it.
In this example, if the orbit associated to a given spherical function contains the element $t={\rm diag}(t_1I_{p_1},\dots,t_kI_{p_k})$ with different $t_j$'s, the quotient pair $(N_t,K_t)$ is a product of  pairs $(\bC^{p_j}\oplus\fu_{p_j}, {\rm U}_{p_j})$ (cf. Section \ref{line2}). This is a proper quotient pair unless $t$ is a scalar multiple of the identity matrix.

In Section~\ref{quotient}, we prove that the map $\ph\longmapsto \ph^t$ gives a local homeomorphism of a neighbourhood of 
$\ph$ in $\Sigma$ onto a neighbourhood of $\ph^t$ in the Gelfand spectrum $\Sigma^t$ of  $(N_t,K_t)$. The proof is based on the existence of slices transversal to  $K$-orbits in $\fz$, which allows local extensions near $t$ of smooth $K_t$-invariant functions on $N_t$ to smooth $K$-invariant functions on $N$ ({\it radialisation}). In Section~\ref{relations-spectra}, we prove that,  given two realisations $\Sigma_\cD$, $\Sigma^t_{\cD_t}$ of the two spectra in  Euclidean spaces, this homeomorphism induces a local identification, near the points corresponding to $\ph$ and $\ph^t$ respectively, of the two spaces $\cS(\Sigma_\cD)$ and 
$\cS(\Sigma^t_{\cD_t})$. 

This result allows to prove the following version of the bootstrapping argument, cf. Section~\ref{sec_towards}.
Denote by $\check \Sigma_\cD$ the set of ``most singular points'' in $\Sigma_\cD$, i.e., those for which $(N_t,K_t)=(N,K)$, and {\it assume that Property {\rm (S)} holds for all proper quotient pairs of $(N,K)$}. If   the spherical transform $\cG F$ of a function $F\in\cS(N)^K$ vanishes of infinite order on $\check\Sigma_\cD$, then it can be extended to a Schwartz function on $\bR^d$.

Let us assume for a moment that all proper quotient pairs of $(N,K)$ satisfy Property~(S). 
Following the general pattern, we show in Section \ref{section_checkN} that $\cG F$ determines a Whitney jet on $\check\Sigma_\cD$. This is a consequence of two facts.

The first fact is that $\check\Sigma_\cD$ is naturally identified with the Gelfand spectrum of $(\check N, K)$, where the Lie algebra of $\check N$ is $\fn/\fz_0$, 
 where $\fz_0$ is the component of $\fz$ on which $K$ acts nontrivially, cf. \eqref{z_0}. 
More precisely, for an appropriate choice of the system $\cD$ of generators of $\bD(N)^K$, we construct a system $\check\cD$ of $d'<d$ generators of $\bD(\check N)^K$ such that 
$$
\check\Sigma_\cD=\Sigma_\cD\cap\big(\bR^{d'}\times\{0\}\big)=\Sigma_{\check\cD}\times\{0\}\ .
$$

The second fact is the Hadamard-type formula of Proposition \ref{hadamard}. Once read on the other side of the spherical transform, it has the form of the inductive step for a Taylor development of $\cG F$,  for $F\in\cS(N)^K$, centred on  $\bR^{d'}\times\{0\}$, in the remaining $d-d'$ variables. This provides the desired Whitney jet, once we observe, by case by case inspection,  that  $(\check N, K)$ is  one of the pairs for which Property (S) is already known to hold (e.g., in the above example, $\check N$ is a Heisenberg group). 

The proof of Proposition \ref{hadamard} relies on the preliminary Proposition \ref{prop_FRY2}, which has already appeared in~\cite{FRY2}. In this paper we only show how Proposition \ref{prop_FRY2} implies Proposition \ref{hadamard}. We remark that the proof of Proposition \ref{prop_FRY2} given in \cite{FRY2} is based on the fact that certain tensor products of irreducible representations of $K$ decompose without multiplicities, and this requires a case by case analysis.
\smallskip

It remains to answer the question if Property~(S) is satisfied by all proper quotient pairs $(N_t,K_t)$ generated by a pair $(N,K)$ in Table \ref{vinberg}. The list of such quotient pairs is given in Section \ref{appendix}.
One can notice that the pairs in the first two blocks of Table \ref{vinberg} form a self-contained family, in the sense that the  quotient pairs that they generate  are products of pairs in the same family with lower dimensional groups. This allows an inductive argument using principle (c) above.

The situation is different for the pairs in the third block, since the quotient pairs they generate do not satisfy condition (V). For them we must provide an {\it ad-hoc} adaptation of the previous argument, with the inductive procedure replaced by  analysis of three consecutive generations of quotient pairs. This  is done in Section \ref{third-block}. We are confident that this extra work on isolated cases will turn out to be useful in view of a future proof of Property (S) for  general pairs not satisfying Vinberg's condition.

\vskip0.5ex

\noindent
{\bf Acknowledgments.}  This work started when the third author was a long-term visitor at Centro di Ricerca Matematica Ennio de Giorgi in Pisa. Part of it was later carried out 
at the Max-Planck-Institute f\"ur Mathematik (Bonn), where all the three authors have been short-term guests. 
We would like to thank these institutes for their warm hospitality and for  providing an excellent stimulating environment.  

The first author acknowledges the support of Scuola Normale Superiore, Pisa, and of the London Mathematical Society Grace Chisholm fellowship held at King's College, London.
The third author also wishes to thank Scuola Normale Superiore for regular invitations 
and D.\,Timashev for bringing the book of Bredon to her attention.

\vskip1cm
\section{Generalities on nilpotent Gelfand pairs}
\bigskip

Let $N$ be a nilpotent, connected and simply connected Lie group, and let $K$ be a compact group of automorphisms of $N$.

\begin{definition}\label{nGp}
 $(N,K)$ is a {\it nilpotent Gelfand pair} (n.G.p. in short) if either of the following equivalent conditions is satisfied:
\begin{enumerate}
\item[\rm(i)]  the convolution algebra $L^1(N)^K$ of integrable $K$-invariant functions on $N$ is commutative;
\item[\rm(ii)] the algebra $\bD(N)^K$ of left-invariant and $K$-invariant differential operators on $N$ is commutative;
\item[\rm(iii)] if $\pi$ is an irreducible unitary representation of $N$ and $K_\pi$ is the stabilizer in $K$ of the equivalence class of $\pi$, then the representation space $\cH_\pi$ decomposes under $K_\pi$ without multiplicities;
\item[\rm(iv)] same as {\rm (iii)}, for $\pi$ generic.
\end{enumerate}
\end{definition}

This is the same as  saying that $(K\ltimes N,K)$ is a Gelfand pair. With $\fn$ denoting the Lie algebra of $N$, we often write 
$(\fn,K)$ instead of $(N,K)$.

In any nilpotent Gelfand pair, $N$ has step at most 2 \cite{BJR90}. We can then split $\fn$ as the direct sum $\fv\oplus[\fn,\fn]$, where $[\fn,\fn]$ is the derived algebra and $\fv$ a $K$-invariant complement of it. For  pairs satisfying Vinberg's condition, the derived algebra coincides with the centre $\fz$, and it will henceforth be denoted by this symbol. 

We regard the Lie bracket on $\fn$ as a skew-symmetric bilinear map from $\fv\times\fv$ to $\fz$. We split $\check\fz$ as
\begin{equation}\label{z_0}
\fz=\fz_0\oplus\check\fz\ ,
\end{equation}
where $\check\fz$ denotes the subspace of $K$-fixed elements of $\fz$, and  $\fz_0$ its (unique) $K$-invariant complement in $\fz$.

We constantly use exponential coordinates on $N$ to identify elements of $N$ with elements of $\fn$. In particular, the product on $N$ is expressed as an operation on $\fv\oplus\fz$, via the Baker-Campbell-Hausdorff formula
$$
(v,z)\cdot(v',z')=\Big(v+v',z+z'+\half[v,v']\Big)\ .
$$

Let $\la=\la_N$ be the standard symmetrisation operator from the symmetric algebra $\fS(\fn)$ onto the universal enveloping algebra $\fU(\fn)$, which is linear and satisfies the identity $\la(X^n)=X^n$ for every $X\in\fn$ and $n\in\bN$. As usual, we regard $\fS(\fn)$ as the space $\cP(\fn^*)$ of polynomials on the dual space $\fn^*$.  When  the elements of $\fU(\fn)$ are regarded as left-invariant differential operators on $N$, we use the notation $\bD(N)$.

Following \cite[Sect. 2.2]{FRY1}, we will use a modified symmetrisation $\la'_N:\cP(\fn^*)\longrightarrow\bD(N)$, which maps the polynomial $p\in\cP(\fn)$ to the differential operator
\begin{equation}\label{modsym}
\la'(p)F=p(i\inv\nabla_{v'},i\inv\nabla_{z'})_{|_{v'=z'=0}} F\big((v,z)\cdot(v',z')\big)\ .
\end{equation}
i.e., $\la'(p)=\la\big(p(i\inv\cdot)\big)$, in terms of the standard symmetrisation $\la$. When it is necessary to specify the group $N$, we write $\la'_N$ instead of $\la'$. The advantage of this modification is that  polynomials with real coefficients are transformed by $\la'$ into formally self-adjoint differential operators\footnote{The first two authors take this opportunity to correct an error in the formulation of Proposition 3.1 in \cite{FR}: it applies to operators $D=\la'(p)$ with $p$ real.}.
Clearly, $\la'(p)$ is $K$-invariant if and only if $p$ is $K$-invariant.

 Introducing a  $K$-invariant  inner product $\lan\ ,\ \ran$  on $\fv\oplus\fz$ under which $\fv\perp\fz$, we identify $\fn^*$ with $\fn$ throughout the paper.

\bigskip
\subsection{Spherical functions, representations of $N$ and $K$-orbits in $\fz$}\label{section-spherical}\quad
\medskip

Let $(N,K)$ be a n.G.p. The spherical functions are the $K$-invariant joint eigenfunctions $\ph$ of all operators in $\bD(N)^K$ normalised by the condition $\ph(0)=1$.  Given $D\in\bD(N)^K$ and a spherical function~$\ph$, we denote by $\xi(D,\ph)$ the corresponding eigenvalue.

We are interested in the bounded spherical functions,  which play the main r\^ole in the Fourier-Godement analysis of Gelfand pairs.

It has been proved in \cite{BJR90} that all bounded spherical functions of  $(N,K)$ are of positive type, hence they are in one-to-one correspondence with (equivalence classes of) irreducible unitary representations of $K\ltimes N$  admitting non-trivial $K$-invariant vectors. 
However, we prefer to avoid representations of the semidirect product and express bounded spherical functions  as partial traces of irreducible unitary representations of $N$.

For $\zeta\in\fz$, denote by $\fr_\zeta\subseteq\fv$ the radical of the bilinear form $B_\zeta(v,v')=\lan\zeta,[v,v']\ran$ and set $\fv_\zeta=\fr_\zeta^\perp$.
The following statement is a direct consequence of the Stone-von Neumann theorem and we omit its proof.

\begin{lemma}\label{representations}
For each $\zeta\in\fz$ there is a unique, up to equivalence, irreducible unitary representation $\pi_\zeta$ of $N$ such that $d\pi_\zeta(0,z)=i\lan\zeta,z\ran I$ for all $z\in \fz$ and $d\pi_\zeta(v,0)=0$ for all $v\in \fr_\zeta$. 

For each $\zeta\in\fz$ and $\omega\in\fr_\zeta$ there is a unique
 irreducible unitary representation $\pi_{\zeta,\omega}$ such that
 \begin{equation}\label{zeta-omega}
d\pi_{\zeta,\omega}(v,z)=d\pi_\zeta(v,z)+i\lan v,\omega\ran I\ .
\end{equation}
Every irreducible unitary representation of $N$ is equivalent to one, and only one, $\pi_{\zeta,\omega}$.
\end{lemma}

We denote by $\cH_\zeta$ the representation space of the representations $\pi_{\zeta,\omega}$. 
The stabilizer $K_{\zeta,\omega}\subset K$ of the point $\omega+\zeta\in\fn$ also stabilizes the equivalence class of $\pi_{\zeta,\omega}$, inducing a unitary representation\footnote{In general, this operation leads to a projective representation of the stabilizer in $K$. In our case we obtain true representations, since restriction of the metaplectic representation of ${\rm Sp}(\fr_\zeta^\perp,B_\zeta)$ to a compact subgroup can be linearized \cite{F}.}  $\sigma$ of $K_{\zeta,\omega}$ on $\cH_\zeta$. By \cite{C, V1}, the fact that $(N,K)$ is a n.G.p. is equivalent to saying that, for each $\zeta,\omega$, $\cH_\zeta$ decomposes without multiplicities into irreducible components under the action of $K_{\zeta,\omega}$, namely,
\begin{equation}\label{multiplicity-free}
\cH_\zeta=\sum_{\mu\in \fX_{\zeta,\omega}}V(\mu)\ ,
\end{equation}
with $\fX_{\zeta,\omega}\subseteq \widehat{K_{\zeta,\omega}}$. To each $\mu\in\fX_{\zeta,\omega}$ we can associate the spherical function
\begin{equation}\label{trace}
\ph_{\zeta,\omega,\mu}(v,z)=\frac1{\dim V(\mu)}\int_K \tr\big(\pi_{\zeta,\omega}(kv,kz)_{|_{V(\mu)}}\big)\,dk\ .
\end{equation}

For given $k\in K$, we have $\fX_{k\zeta,k\omega}=\fX_{\zeta,\omega}$, under the natural identification of the dual object $\widehat{K_{\zeta,\omega}}$ of $K_{\zeta,\omega}$ with the dual object of $K_{\pi_{\zeta,\omega}^k}=k\inv K_{\pi_{\zeta,\omega}} k$, and
$$
\ph_{k\zeta,k\omega,\mu}=\ph_{\zeta,\omega,\mu}\ .
$$

\bigskip

\subsection{Spectra and their immersions in $\bR^d$}\quad
\medskip

Given a n.G.p. $(N,K)$, we denote by $\Sigma$, or $\Sigma(N,K)$, the Gelfand spectrum of $L^1(N)^K$, i.e. the set of bounded $K$-spherical functions on $N$ with the compact-open topology.

By a {\it  homogeneous Hilbert basis}, or a {\it fundamental system of invariants}, we mean a $d$-tuple ${\boldsymbol\rho}=(\rho_1,\dots, \rho_d)$ of real, $K$-invariant polynomials on $\fn$ which generate the $K$-invariant polynomial algebra $\cP(\fn)^K$ over $\fn$ and with each $\rho_j$  homogeneous in the $\fv$-variables and in the $\fz$-variables separately, i.e., belonging to $\cP^{r_j}(\fv)\otimes\cP^{s_j}(\fz)$ for some  $r_j,s_j$.

We set $D_j=\la'(\rho_j)$ and $\cD=(D_1,\dots,D_d)$.  Then $\cD$ generates $\bD(N)^K$. We call it a {\it homogeneous basis} of $\bD(N)^K$.

By Proposition 3.1 of \cite{FR}, the $D_j$ are essentially self-adjoint on $\cS(N)$ and their closures admit a joint spectral resolution. 

Given $\ph\in\Sigma$, denote by $\xi_j(\ph)\in\bR$ the eigenvalue $\xi(D_j,\ph)$ of $\ph$ under $D_j$.
Then
\begin{equation}\label{Sigma_D}
\Sigma_\cD=\big\{\xi(\ph)=\big(\xi_1(\ph),\dots,\xi_d(\ph)\big):\ph\in\Sigma\big\}\subset\bR^d
\end{equation}
is closed and homeomorphic to $\Sigma$, cf. \cite{FeRu}. Moreover, $\Sigma_\cD$ is the joint $L^2$-spectrum of  $\cD$ as a family of strongly commuting self-adjoint operators \cite{FR}.

If $\tilde{\boldsymbol\rho}=(\tilde\rho_1,\dots, \tilde\rho_{\tilde d})$ is another real, homogeneous Hilbert basis of $\cP(\fn)^K$ and $\tilde\cD$ is the corresponding homogeneous basis of $\bD(N)^K$,  there are polynomials $P_k$, $\tilde P_j$ such that
\begin{equation}\label{changeD}
\tilde D_k=P_k(D_1,\dots,D_d)\ ,\qquad D_j=\tilde P_j(\tilde D_1,\dots,\tilde D_{\tilde d})\ ,
\end{equation}
for $1\le j\le d$, $1\le k\le \tilde d$. Hence $\Sigma_{\tilde \cD}=P(\Sigma_\cD)$ and  $\Sigma_\cD=\tilde P(\Sigma_{\tilde \cD})$.

The following statement is proved in \cite{ADR2} (cf. Lemma 3.1 and Corollary 3.2 therein).

\begin{proposition}\label{indipendence}
The validity of Property (S) is independent of the choice of ${\boldsymbol\rho}$ (i.e. of $\cD$).
\end{proposition}

On $N$, as well as on its Lie algebra,  we consider the automorphic dilations $\del\cdot(v,z)=(\del^\half v,\del z)$, defined for $\del>0$. If $p\in \cP(\fn)$ is 
homogeneous of degree $\nu'$ in the $\fv$-variables and of degree $\nu''$ in the $\fz$-variables, i.e., $p\in\cP^{\nu'}(\fv)\otimes\cP^{\nu''}(\fz)$, then the operator $D=\la'(p)$ is homogeneous of degree $\nu=\frac{\nu'}2+\nu''$ with respect to the automorphic dilations, i.e., it satisfies the identity
\begin{equation}\label{homogeneity}
D\big(F(\del\cdot)\big)(v,z)=\del^\nu(DF)\big(\del\cdot(v,z)\big)\ .
\end{equation}

 If, in particular, $\ph\in\Sigma$, then also $\ph^\del(v,t)=\ph\big(\del(v,z)\big)$ is in $\Sigma$. If $\cD=\la'(\boldsymbol\rho)$  and $\nu_j$ is the homogeneity degree of $D_j\in\cD$,
 $$
 \xi(D_j,\ph^\del)=\del^{\nu_j}\xi(D_j\ph)\ .
 $$
 
Hence $\Sigma_\cD$ is invariant under the dilations on $\bR^d$
\begin{equation}\label{dilations}
\xi=(\xi_1,\dots,\xi_d)\longmapsto \big(\del^{\nu_1}\xi_1,\dots,\del^{\nu_d}\xi_d\big)=D(\del)\xi\ ,\qquad (\del>0)\ .
\end{equation}

 On $\bR^d$ we introduce the homogeneous norm, compatible with the dilations $D(\del)$,
\begin{equation}\label{norm}
\|\xi\|=\sum_{j=1}^d|\xi_j|^{\frac1{\nu_j}}\ .
\end{equation}

\bigskip

\subsection{Special Hilbert bases}\label{special}\quad
\medskip

We will privilege homogeneous Hilbert bases ${\boldsymbol\rho}$  which split as ${\boldsymbol\rho}=({\boldsymbol\rho}_{\fz_0},{\boldsymbol\rho}_{\check\fz},{\boldsymbol\rho}_\fv,{\boldsymbol\rho}_{\fv,\fz_0})$, where, keeping in mind \eqref{z_0},
\begin{enumerate}
\item[(i)] ${\boldsymbol\rho}_{\fz_0}$ is a homogeneous Hilbert basis  of $\cP(\fz_0)^K$; 
\item [(ii)] ${\boldsymbol\rho}_{\check\fz}$ is a system  of coordinate functions on $\check\fz$; 
\item[(iii)] ${\boldsymbol\rho}_\fv$ is a homogeneous Hilbert basis  of $\cP(\fv)^K$;
\item[(iv)] ${\boldsymbol\rho}_{\fv,\fz_0}$ contains polynomials in $\big(\cP^{\nu'}(\fv)\otimes\cP^{\nu''}(\fz_0)\big)^K$ with $\nu',\nu''>0$.
\end{enumerate}

Denoting by $d_{\fz_0}$, $d_{\check\fz}$, $d_\fv$, $d_{\fv,\fz_0}$  the number of elements in each subfamily, we split $\bR^d$ as
$$
\bR^d=\bR^{d_{\fz_0}}\times\bR^{d_{\check\fz}}\times\bR^{d_{\fv}}\times\bR^{d_{\fv,\fz_0}}\ ,
$$
and set
$$
\xi(\ph)=\big(\xi_{\fz_0}(\ph),\xi_{\check\fz}(\ph),\xi_\fv(\ph),\xi_{\fv,\fz_0}(\ph)\big)\ .
$$

Whenever a unified notation for all invariants on $\fz$ is preferable, we set 
$$
{\boldsymbol\rho}_\fz=({\boldsymbol\rho}_{\fz_0},{\boldsymbol\rho}_{\check\fz})=\big(\rho_1(z),\dots,\rho_{d_\fz}(z)\big)\ ,\qquad \xi_\fz(\ph)=\big(\xi_{\fz_0}(\ph),\xi_{\check\fz}(\ph)\big)\ .
$$

 Since $\la'(q)=q(i\inv\nabla_z)$ for every polynomial $q$ on $\fz$, we have 
\begin{equation}\label{rho_z}
\xi_\fz(\ph)={\boldsymbol\rho}_\fz(\zeta)\ ,
\end{equation}
for every bounded spherical function $\ph=\ph_{\zeta,\omega,\mu}$.
This gives the following.

\begin{lemma}\label{Pi}
The canonical projection $\Pi$ from $\bR^d$ onto $\bR^{d_\fz}$ restricts to a surjective map
$$
\Pi_{|_{\Sigma_\cD}}:\Sigma_\cD\longrightarrow {\boldsymbol\rho}_\fz(\fz)\ ,
$$
and, for $\xi_\fz\in{\boldsymbol\rho}_\fz(\fz)$, 
$$
\Pi_{|_{\Sigma_\cD}}\inv(\xi_\fz)=\big\{\xi(\ph_{\zeta,\omega,\mu}):{\boldsymbol\rho}_\fz(\zeta)=\xi_\fz\big\}\ .
$$

In particular, the map which assigns to a spherical function $\ph_{\zeta,\omega,\mu}\in\Sigma(N,K)$ the value ${\boldsymbol\rho}_\fz(\zeta)\in \bR^{d_\fz}$ is continuous.
\end{lemma}

 We observe that ${\boldsymbol\rho}_\fz(\fz)$ is homeomorphic to the orbit space $\fz/K$.

\bigskip

\subsection{Dominant coordinates in $\Sigma_\cD$}\quad
\medskip

\begin{lemma}\label{dominant}
Let $\xi=(\xi_{\fz_0},\xi_{\check\fz},\xi_\fv,\xi_{\fv,\fz_0})$ be a point in $\Sigma_\cD$. 
We have the following inequalities:
\begin{enumerate}
\item[\rm(i)] if $\rho_j\in\cP^{\nu'_j}(\fv)\otimes\cP^{\nu''_j}(\fz_0)$, with $\nu'_j,\nu''_j>0$, then $|\xi_j|\le C\|\xi_\fv\|^{\frac{\nu'_j}2}\|\xi_{\fz_0}\|^{\nu''_j}$;
\item[\rm(ii)] $\|\xi\|\le C\|\xi_\fv\|$.
\end{enumerate}

In particular, if $\xi_{\fz_0}=0$, then also $\xi_{\fv,\fz_0}=0$.
\end{lemma}

\begin{proof}
Let $p(v)=|v|^2$, where $|\ |$ denotes the norm induced by a $K$-invariant inner product on $\fv$. 
Denote by $X_v$ the left-invariant vector field equal to $\de_v$ at the identity of $N$. Then $\la'(p)$ is the sublaplacian $L=-\sum X_{e_j}^2$, where $\{e_j\}$ is an orthonormal basis of $\fv$. Then $L$ is hypoelliptic and, for every $v_1,\dots,v_m\in\fv$ and $F\in\cS(N)$,
$$
\|X_{v_1}\cdots X_{v_m}F\|_2\le C_{v_1,\dots,v_m}\| L^{\frac m2}F\|_2\ ,
$$
cf. \cite{FS}.
For $z_1,\dots,z_m\in\fz_0$ and $F\in\cS(N)$, we also have, by classical Fourier analysis,
$$
\|\de_{z_1}\cdots \de_{z_m}F\|_2\le C_{z_1,\dots,z_m}\| \Delta_{\fz_0}^{\frac m2}F\|_2\ .
$$

Therefore,
$$
\begin{aligned}
\|\de_{z_1}\cdots \de_{z_{m'}}X_{v_1}\cdots X_{v_m}F\|_2^2&=\big|\lan\de_{z_1}^2\cdots \de_{z_{m'}}^2F,X_{v_1}\cdots X_{v_m}X_{v_m}\cdots X_{v_1}F\ran\big|\\
&\le C_{v_1,\dots,v_m,z_1,\dots,z_{m'}}\|L^mF\|_2\|\Delta_{\fz_0}^{m'}F\|_2\ .
\end{aligned}
$$

If $\rho_j\in\cP^{\nu'_j}(\fv)\otimes\cP^{\nu''_j}(\fz_0)$, then $D_j=\la'(\rho_j)$ is a linear combination of terms of this kind, with $m=\nu'_j$ and $m'=\nu''_j$. Therefore, for $F\in\cS(N)$,
\begin{equation}\label{control}
\|D_jF\|_2\le C\|L^{\nu'_j}F\|_2^\half\|\Delta_{\fz_0}^{\nu''_j}F\|_2^\half\ .
\end{equation}

 We may assume that ${\boldsymbol\rho}_\fv$ contains the squares $p_k=|v_k|^2$ of the norm restricted to mutually orthogonal irreducible components of $\fv$. Then $p$ is the sum of such $p_k$. The same can be said about $q(z_0)=|z_0|^2$ on $\fz_0$. 

Denote by $\xi_k$  the component of $\xi_\fv$ corresponding to $\la'(p_k)$, $\xi_\ell$  the component of $\xi_{\fz_0}$ corresponding to $\la'(q_\ell)$,  and $\xi_j$ the component of $\xi_{\fv,\fz_0}$ corresponding to $D_j$. Then, since $\la'(p_k)$ and $\la'(q_\ell)$ are positive operators, we have $\xi_k,\xi_\ell\ge0$ in $\Sigma_\cD$. Hence \eqref{control} implies that, on~$\Sigma_\cD$,
$$
|\xi_j|\le C\Big(\sum_k\xi_k\Big)^\frac{\nu'_j}2\Big(\sum_\ell\xi_\ell\Big)^\frac{\nu''_j}2\le C\|\xi_\fv\|^\frac{\nu'_j}2\|\xi_{\fz_0}\|^{\nu''_j}\ .
$$

This proves (i). To prove (ii) it suffices to prove that $\|\xi_\fz\|\le C\|\xi_\fv\|$. For this, it suffices to observe that every derivative $\de_z^\al$ in the $\fz$-variables can be expressed as a combination of products $X_{v_1}\cdots X_{v_m}$ with $m=2|\al|$. Then, for every $F\in\cS(N)$,
$$
\|\de_z^\al F\|_2\le C\|L^{|\al|}F\|_2\ .\qedhere
$$
\end{proof}

\bigskip

\subsection{Quotients in $\fz$ and Radon transforms of functions and differential operators}\label{modulo-s}\quad
\medskip

Given a subspace $\fs$  of $\fz$, denote by $\fn'$ the quotient algebra $\fn/\fs$ and by $N'$ the corresponding quotient group. 

Given a $K$-invariant inner product on $\fz$, set $\fz'=\fs^\perp$ and let ${\rm proj}$ be the orthogonal projection of $\fz$ onto $\fz'$. Then $\fn'$ can be regarded as $\fv\oplus\fz'$ with Lie bracket
$$
[v,w]_{\fn'}={\rm proj}[v,w]\ .
$$

 We denote by $K'$  the stabilizer of $\fs$ in $K$ and fix Lebesgue measures $dv$, $dz'$,  $ds$ on $\fv$, $\fz'$,  $\fs$, respectively.
 We define the {\it Radon transform}\footnote{The term ``Radon transform'' is abused here. It comes from the special case where $\fz=\bR^n$, $K={\rm SO}_n$ and $\fs=\bR^{n-1}$. 
Assuming $K$-invariance of $F$, the integrals in \eqref{radon*} give the values of integrals over all subspaces of $\fz$ obtained from $\fs$ by translations and rotations by elements of $K$. However, $\cR$  is not injective in general. It is injective if $\fz'$ intersects  almost all $K$-orbits in $\fz$.} $\cR F$ of a function $F\in \cS(N)^K$, as the function on $N'$
\begin{equation}\label{radon*}
\cR F(v,z')=\int_\fs F(v,z'+s)\,ds\ .
\end{equation}

Then $\cR F\in \cS(N')^{K'}$, and for every bounded function $G$ on $N'$,
$$
\int_{N'}\cR F(v,z')G(v,z')\,dv\,dz'=\int_N F(v,z)(G\circ{\rm proj})(v,z)\,dv\,dz\ .
$$

Accordingly, given $D\in\bD(N)^K$, we define $\cR D\in\bD(N')^{K'}$ as the operator such that
\begin{equation}\label{opradon*}
\big((\cR D)G\big)\circ{\rm proj}=D(G\circ{\rm proj})\ .
\end{equation}

Notice, however, that $(N',K')$ is not a Gelfand pair in general.

We remark here some basic properties of Radon transforms of differential operators, cf. \cite{FR}, Section 4 and \cite{FRY1}, Lemma 4.2. Further properties will be stated when needed.

\begin{proposition}\label{propradon*}
\quad

\begin{enumerate}
\item[\rm(i)] If $D=\la'_N(p)\in\bD(N)^K$, then $\cR D=\la'_{N'}(p_{|_{\fn'}})$.
\item[\rm(ii)] If $(N',K')$ is a Gelfand pair and $\ph'$ is a bounded spherical function on $N'$,  then 
\begin{equation}\label{lift-phi}
\Lambda\ph'(v,z)=\int_K(\ph'\circ{\rm proj})(kv,kz)\,dk\ ,
\end{equation} 
is a bounded spherical function on~$N$. For $D\in\bD(N)^K$,
\begin{equation}\label{radon-eigenvalues}
\xi(D,\Lambda\ph')=\xi(\cR D,\ph')\ .
\end{equation}
\end{enumerate}
\end{proposition}
 
\begin{proof}
(i) follows from \eqref{modsym}. 

For $D\in\bD(N)^K$, set $\xi=\xi(\cR D,\ph')$. Then
$$
\begin{aligned}
D(\Lambda\ph')(v,z)&=\int_KD(\ph'\circ{\rm proj})(kv,kz)\,dk\\
&=\int_K\big((\cR D)\ph'\big)\circ{\rm proj}(kv,kz)\,dk\\
&=\xi \,\Lambda\ph'(v,z)\ .
\end{aligned}
$$

This proves that $\Lambda\ph'$ is spherical and that \eqref{radon-eigenvalues} holds.
\end{proof}

Since the operation in \eqref{lift-phi} is continuous in the compact-open topology, we have the following.

\begin{corollary}\label{Sigma'->Sigma}
Assume that $(N',K')$ is a Gelfand pair.
The map $\Lambda$ is continuous from $\Sigma(N',K')$ to $\Sigma(N,K)$.
\end{corollary}

\vskip1cm
\section{Reductions}\label{section-reductions}
\bigskip

In this section we collect three preliminary results which will be used repeatedly in the course of this paper. Moreover, Propositions \ref{normal} and \ref{central} are the ingredients that allow to reduce the full Vinberg list to Table \ref{vinberg}. 

We refer to \cite{FRY1} (which extends an argument of \cite{ADR2}), for the proof of the first result,   concerning {\it normal extensions} of $K$. 

\begin{proposition}\label{normal}
Let $K^\#$ be a compact group of automorphisms of $N$, and $K$ a normal subgroup of $K^\#$. If $(N,K)$ is a n.G.p. satisfying Property (S), then also $(N,K^\#)$ satisfies Property (S).
\end{proposition}

\smallskip

 The second result,  concerning {\it central reductions} of $N$, has been announced in \cite{FRY2} without proof, and we give its proof below.

Assume  that $\fz$ contains  a nontrivial proper $K$-invariant subspace $\fs$. If $\fn'=\fn/\fs$, as in Subsection \ref{modulo-s}, then $K'=K$.  It is proved in \cite{V2} that $(N',K)$ is a n.G.p.

\begin{proposition}\label{central}
Suppose that Property (S) holds for the n.G.p. $(N,K)$. Then Property (S) also holds for  $(N',K)$.
\end{proposition}

\begin{proof}
We fix a $K$-invariant  complementary subspace $\fz'$ of $\fs$ in $\fz$ and adopt the notation of Subsection~\ref{modulo-s}.

 Given a $K$-invariant function $f$ on $N'$, the function $\tilde f=f\circ{\rm proj}$ is $K$-invariant on $N$. Hence, if $\ph$ is spherical on $N'$,  the map $\Lambda$ in Corollary \ref{Sigma'->Sigma}, reduces to composition with ${\rm proj}$, and therefore it is injective.
With suitable choices of generators for $\bD(N')^K$ and $\bD(N)^K$,
this injection can be viewed in the following way.

We fix a fundamental system $(\rho_1,\dots,\rho_{d'})$ of $K$-invariants on $\fv\oplus\fz'$, and complete it into a fundamental system $(\rho_1,\dots,\rho_{d'},\rho_{d'+1},\dots,\rho_d)$ of $K$-invariants on $\fn$, where each of the added polynomials $\rho_{d'+1}(v,z'+s),\dots,\rho_d(v,z'+s)$ vanishes for $s=0$.

Applying the symmetrisation  \eqref{modsym} on $N$ and $N'$ respectively, we obtain the generating systems of differential operators 
$$
\cD=(D_1,\dots,D_d)\ ,\qquad \cD'=(D'_1,\dots,D'_{d'})\ ,
$$
on $N$ and $N'$ respectively. 
In particular,
$$
\cR D_j=
\begin{cases}
D'_j&\text{ if } j=1,\dots,d'\ ,\\
0&\text{ if }  j=d'\!\!+\!1,\dots,d\ .
\end{cases}
$$

It follows from Proposition \ref{propradon*} that, if $\ph$ is a bounded spherical function on $N'$ represented by the point $\xi'\in\Sigma_{\cD'}\subset\bR^{d'}$, then $\tilde\ph$ is represented by the point $(\xi',0)\in\Sigma_\cD\subset\bR^d$.

Hence the injective map of Corollary \ref{Sigma'->Sigma} corresponds to the map $\xi'\longmapsto (\xi',0)$
 from $\Sigma_{\cD'}$ into $\Sigma_\cD$. The corresponding restriction operator is a continuous linear mapping
from $\cS(\Sigma_\cD)$ onto $\cS(\Sigma_{\cD'})$.

Let $\cG$ and $\cG'$ be the spherical transforms corresponding to $\Sigma_\cD$ and $\Sigma_{\cD'}$.
We fix a smooth and compactly supported function $\psi$ on $\fs$ with integral 1.

Given $F\in \cS(N')^K$,
for any bounded spherical function $\varphi$ of $(N',K)$ we have
$$
\int_{\fn'} F(v,z') \varphi(v,z') \,dv\, dz'
=
\int_{\fn} F(v,z') \psi(s) \tilde \varphi(v,z'+s) \,dv\, dz'\,ds
\ ,
$$
where $\tilde\ph=\ph\circ{\rm proj}$.
Thus, for any  $\xi' \in \Sigma_{\cD'}$,
$$
\cG'F (\xi')
=
\cG (F\otimes \psi) (\xi',0)
\ ,
$$
where  $(F\otimes\psi)(v,z'+s)= F(v,z')\psi(s)$.
We have obtained the identity
\begin{equation}
\label{Gelfand_transforms}
\cG' F=\big(\cG (F\otimes \psi)\big)_{|_{\Sigma_{\cD'}}} \ .
\end{equation}

Obviously, $F\otimes\psi \in \cS(N)^K$ and the map $F\mapsto F\otimes\psi$ is continuous  and linear from $\cS(N')^K$ to $\cS(N)^K$. 

As we are assuming that $(N,K)$ satisfies Property (S), 
$\cG:\cS(N)^K\rightarrow \cS(\Sigma_\cD)$ 
is an isomorphism of Fr\'echet spaces. By composition, the map
$$
F\longmapsto \big(\cG (F\otimes \psi)\big)_{|_{\Sigma_{\cD'}}}=\cG' F
$$
is continuous from $\cS(N')^K$ to $\cS(\Sigma_{\cD'})$. Since we already know that the inverse of $\cG'$ maps $\cS(\Sigma_{\cD'})$ into $\cS(N')^K$ continuously \cite{ADR2}, we have the conclusion.
\end{proof}

\smallskip

 Our third result concerns {\it product pairs}.

\begin{proposition}\label{product}
Let $(N_1,K_1)$, $(N_2,K_2)$ be two n.G.p. satisfying Property (S). Then also the product pair $(N_1\times N_2,K_1\times K_2)$ satisfies Property (S).
\end{proposition}

In order to  prove Proposition \ref{product},
we need the following property of decomposition of Schwartz functions on the product of two Euclidean spaces.
We use the following Schwartz norms on $\cS(\bR^n)$:
$$
\|f\|_{\cS(\bR^n), M}
=
\|f\|_{M}
=
\sup_{|\alpha|\leq M, x\in \bR^n} (1+|x|)^M |\partial^\alpha f(x)|
\quad , \quad M\in \bN
\ .
$$

\begin{lemma}\label{lemma_product}
Let $n_1$ and $n_2$ be two positive integers. Set $n=n_1+n_2$.

Let also $\psi_\nu$, $\nu=1,2$, be smooth function on $\bR^{n_\nu}$,
supported in $[-1, 1]^{n_\nu}$ and satisfying:
$$
0\leq \psi_\nu\leq 1
\qquad\mbox{and}\qquad
\psi_\nu=1 \;\text{\rm on }\; \big[-\frac 34, \frac34\big]^{n_\nu}
\ .
$$

For $\nu=1,2$, $l_\nu,m_\nu\in\bZ^{n_\nu}$, set
$$
H^{(\nu)}_{l_\nu,m_\nu}(x_\nu)=e^{-i(x_\nu+l_\nu)\cdot m_\nu}
\psi_\nu(x_\nu+l_\nu)
\ .
$$

Then the following properties hold.
\begin{enumerate}
\item[\rm(a)] Let $\nu=1$ or $2$.
Given $M\in \bN$ there is a constant $C_M>0$ such that,
\begin{equation}
  \label{control_Hlmi}
\forall l_\nu,m_\nu\in \bZ^{n_\nu}\ ,\qquad
\|H_{l_\nu,m_\nu}^{(\nu)}\|_{\cS(\bR^{n_\nu}), M}
\leq C_{M} (1+|l_\nu|+|m_\nu|)^{2M}
\ .  
\end{equation}
\item[\rm(b)] For any function $F\in \cS(\bR^n)$ 
there exist coefficients $c_{l,m}\in \bC$, 
$l=(l_1,l_2)$, $m=(m_1,m_2) \in \bZ^n=\bZ^{n_1}\times\bZ^{n_2}$,
such that
\begin{equation}
  \label{dec_F_sum_lm}
F=\sum_{l,m\in \bZ^n} c_{l,m} H_{l_1,m_1}^{(1)}\otimes H_{l_2,m_2}^{(2)}
\ .  
\end{equation}
\item[\rm(c)] For any $M\in\bN$, there is a constant $C_M$, independent of $F$,
such that the coefficients $c_{l,m}$ satisfy:
\begin{equation}
\label{control_clm}
\forall l,m\in \bZ^n\qquad
|c_{l,m}|\leq C_M \|F\|_{\cS(\bR^n),M} (1+ |l|+|m|)^{-M}
\ .
\end{equation}

\end{enumerate}
\end{lemma}

\begin{proof}
For $\nu=1,2$,
we consider a partition of $\bR^{n_1}$ with the cubes $l_\nu+[-\frac12,\frac12]^{n_\nu}$, $l_\nu\in \bZ_{d_\nu}$, and the partition of unity obtained from a smooth function $\phi_\nu$ on $\bR^{n_\nu}$
supported in $[-\frac 34, \frac34]^{n_\nu}$ such that
$$
0\leq \phi_\nu\leq 1
\qquad,\qquad
\phi_\nu=1 \;\mbox{on}\; [-\frac 14, \frac14]^{n_\nu}
\qquad\mbox{and}\qquad
\forall x_\nu\in \bR^{n_\nu}\quad \sum_{l_\nu\in \bZ^{n_\nu}}\phi_\nu(x_\nu+l_\nu)=1
\ .
$$
For each $l=(l_1,l_2)\in \bZ^{n_1}\times\bZ^{n_2}$,
the function $x=(x_1,x_2)\mapsto F(x_1-l_1,x_2-l_2) \phi_1(x_1)\phi_2(x_2)$
is smooth and supported in $[-\frac 34, \frac34]^n$. 
 We keep the same notation for its $2\pi$-periodic extension in each variable; then 
its decomposition in Fourier series gives:
$$
F(x-l) \phi_1(x_1)\phi_2(x_2)
=
\sum_{m\in \bZ^n} c_{l,m} e^{-i x.m}
\quad ,\quad x\in [-\pi,\pi]^n
\ .
$$
It is easy to see that the coefficients $c_{l,m}$ satisfy \eqref{control_clm} for any $M\in \bN$.
We can write:
\begin{eqnarray*}
F(x)
&=&
\sum_{l=(l_1,l_2)\in \bZ^{n_1}\times \bZ^{n_2}}
F(x)\phi_1(x_1+l_1)\phi_2(x_2+l_2) 
\psi_1(x_1+l_1)\psi_2(x_2+l_2) 
\\
&=&
\sum_{l,m\in \bZ^n}
c_{l,m} e^{-i(x+l).m}
\psi_1(x_1+l_1)\psi_2(x_2+l_2) 
\ .
\end{eqnarray*}
 satisfy \eqref{control_Hlmi} and \eqref{dec_F_sum_lm}.
\end{proof}

\begin{corollary}\label{Kinvar}
For $\nu=1,2$, let $K_\nu$ be a compact subgroup of ${\rm GL}_{n_\nu}(\bR)$. For $\nu=1,2$, $l_\nu,m_\nu\in\bZ^{n_\nu}$, there exist $K_\nu$-invariant smooth functions $\tilde H^{(\nu)}_{l_\nu,m_\nu}$ on $\bR^{n_\nu}$ such that the conclusions of Lemma \ref{lemma_product} hold, with $\tilde H^{(\nu)}_{l_\nu,m_\nu}$ in place of $H^{(\nu)}_{l_\nu,m_\nu}$, for $F$ in $\cS(\bR^n)$ and $K_1\times K_2$-invariant.
\end{corollary}

\begin{proof} Just take as $\tilde H^{(\nu)}_{l_\nu,m_\nu}$ the $K_\nu$-average of $H^{(\nu)}_{l_\nu,m_\nu}$. The conclusion is quite obvious.
\end{proof}

\begin{proof}[Proof of Proposition \ref{product}]
We fix two families, $\cD^{(1)}=(D_1^{(1)},\ldots, D_{d_1}^{(1)})$ 
and $\cD^{(2)}=(D_1^{(2)},\ldots,D_{d_2}^{(2)})$, of generators
of $\bD(N_1)^{K_1}$ and $\bD(N_2)^{K_2}$ respectively.
We keep the same notation when the operators $D_j^{(\nu)}$ 
are applied to functions on $N=N_1\times N_2$ by differentiating in the $N_\nu$-variables. 

Then 
$\cD=(D_1^{(1)},\ldots, D_{d_1}^{(1)},D_1^{(2)},\ldots,D_{d_2}^{(2)})$ is a family
of generators of $\bD(N)^K$, $K=K_1\times K_2$,
$$
\Sigma_\cD=\Sigma_{\cD_1}\times\Sigma_{\cD_2}\subset \bR^{d_1}\times\bR^{d_2}\ ,
$$
and, if $\cG_1$, $\cG_2$, $\cG$ are the corresponding Gelfand transforms, then 
$$
\cG(F_1\otimes F_2)=(\cG_1 F_1)\otimes(\cG_2F_2)\ .
$$

Identifying $N_1$ and $N_2$ with their Lie algebras,
we consider the $K_i$-invariant functions $\tilde H_{l_\nu,m_\nu}^{(\nu)}$, $\nu=1,2$, $l_\nu,m_\nu\in \bZ^{n_\nu}$, 
satisfying the properties of Corollary \ref{Kinvar}.

Given $F\in \cS(N)^K$, we decompose it as
$$
F=\sum_{l,m\in \bZ^n} c_{l,m} \tilde H_{l_1,m_1}^{(1)}\otimes \tilde H_{l_2,m_2}^{(2)}
\ ,
$$
with coefficients $c_{l,m}$ satisfying \eqref{control_clm}.
Then
$$
\cG F=\sum_{l,m\in \bZ^n} c_{l,m}\cG_1 \tilde H_{l_1,m_1}^{(1)}\otimes \cG_2\tilde H_{l_2,m_2}^{(2)}
\ .
$$

Since we are assuming that each $(N_\nu,K_\nu)$ satisfies Property (S), given any  $M\in \bN$, there are functions $h_{l_\nu,m_\nu}^{(\nu,M)}\in \cS(\bR^{d_\nu})$, for  $\nu=1,2$ and $l_\nu,m_\nu\in \bZ^{n_\nu}$, and an integer $A_M$ such that
\begin{enumerate}
\item[(i)] $h_{l_\nu,m_\nu}^{(\nu,M)}$ coincides with $\cG_\nu \tilde H_{l_\nu,m_\nu}^{(\nu)}$ on $\Sigma_{\cD_\nu}$,
\item[(ii)] $\|h_{l_\nu,m_\nu}^{(\nu,M)}\|_{\cS(\bR^{d_\nu}),M}\le C_M\|\tilde H_{l_\nu,m_\nu}^{(\nu)}\|_{\cS(N_\nu),A_M}\le C_M(1+|l_\nu|+|m_\nu|)^{2A_M}$.
\end{enumerate}

If we set
$$
h_{l,m}^{(M)}=h_{l_1,m_1}^{(1,M)}\otimes h_{l_2,m_2}^{(2,M)}
\in \bS(\bR^d)\ ,
$$
where $d=d_1+d_2$, we have, for any $l,m\in \bZ^n$, 
\begin{equation}
\label{control_hlmM}
\|h_{l,m}^{(M)}\|_{\cS(\bR^{d}),M}\leq C_M (1+|l|+|m|)^{2A_M}
\ .
\end{equation}

Combining the rapid decay of the coefficients with the polynomial growth \eqref{control_hlmM} of the $M$-th Schwartz norm of the $h_{l,m}^{(M)}$, we obtain that
\begin{equation}\label{B_M-estimate}
\sum_{l,m\in \bZ^n}|c_{l,m}|\|h_{l,m}^{(M)}\|_{\cS(\bR^{d}),M}\le C_M\|F\|_{\cS(N),B_M} \ ,
\end{equation}
with $B_M=2A_M+n+1$.

Given $M_0\in\bN$, we want to construct a Schwartz extension $f^{(M_0)}$ of $\cG F$ whose $M_0$-th Schwartz norm is controlled by a constant, independent of $F$, times $\|F\|_{\cS(N),B_{M_0}}$.

By \eqref{B_M-estimate}, for every $M> M_0$, there exists $a_M=a_{M,M_0,F}\in \bN$ such that
\begin{equation}
\label{series_2-M}
\sum_{\tiny\begin{array}{c}
l,m\in \bZ^n
\\
|l|+|m|\geq a_M
\end{array}
 }|c_{l,m}|\|h_{l,m}^{(M)}\|_{\cS(\bR^{d}),M}
\leq 2^{-M} \| F\|_{\cS(N), B_{M_0}}\ .
\end{equation}

The $a_M$ can be inductively chosen to be non-decreasing.
Then we define $f^{(M_0)}$ as
$$
f^{(M_0)}=
\sum_{\tiny\begin{array}{c}
l,m\in \bZ^n
\\
|l|+|m|< a_{M_0+1}
\end{array}}
 c_{l,m} h_{l,m}^{(M_0)}
 +
 \sum_{M=M_0+1}^\infty \!\!\!\!\!\!\!\!\!
 \sum_{\tiny\begin{array}{c}
l,m\in \bZ^n
\\
a_M\leq |l|+|m|< a_{M+1}
\end{array}}\!\!\!\!\!\!\!\!\!
c_{l,m} h_{l,m}^{(M)}
\ .
$$

Clearly, the series converges on $\Sigma_\cD$ to $\cG F$.
To show that it defines a Schwartz function on all of $\bR^d$, notice that, for every $M_1\in \bN$ with $M_1> M_0$ and every $M\ge M_1$, we have by \eqref{series_2-M} that
$$
\sum_{\tiny\begin{array}{c}
l,m\in \bZ^n
\\
a_M\leq |l|+|m|< a_{M+1}
\end{array} }\!\!\!\!\!\!\!\!\!
|c_{l,m}|\|h_{l,m}^{(M)}\|_{\cS(\bR^{d}),M_1}
\leq \!\!\!\!\!\!\!\!\!
\sum_{\tiny\begin{array}{c}
l,m\in \bZ^n
\\
a_M\leq |l|+|m|
\end{array} }\!\!\!\!\!\!\!\!\!
|c_{l,m}|\|h_{l,m}^{(M)}\|_{\cS(\bR^{d}),M}
\leq 2^{-M} \| F\|_{\cS(N), B_{M_0}}
\ .
$$

This implies that $\|f\|_{\cS(\bR^{d}), M_1}$ is finite. Moreover, combining \eqref{B_M-estimate} and \eqref{series_2-M} together, we have that
$$
\|f^{(M)}\|_{\cS(\bR^d),M_0}\le (C_{M_0}+2^{-M_0})\|F\|_{\cS(N),B_{M_0}}\ ,
$$
as required.
\end{proof}

\vskip1cm
\section{Pairs satisfying Vinberg's condition}\label{spectra}
\bigskip

Given the results of \cite{ADR1, ADR2, FR, FRY1} and the reduction arguments exhibited in Section \ref{section-reductions}, the pairs for which we need to prove  Property (S) are those listed in Table \ref{vinberg}. We recall the splitting $\fz=\fz_0\oplus\check \fz$, where $\check\fz$ is the subspace of fixed points under $K$.

\begin{table}[htdp]
\begin{center}
\begin{tabular}{|r||l|l|l|l|l|}
\hline
&$K$&$\fv$&$\fz$&notes&$\fz_0$ (if $\ne\fz$)\\
\hline
1&${\rm SO}_n$&$\bR^n$&$\fs\fo_n$&$n\ge4$&\\
2&${\rm U}_n$&$\bC^n$&$\fu_n$&$n\ge2$&$\su_n$\\
3&$\text{Sp}_n$&$\bH^n$&$HS^2_0\bH^n\oplus\IM\bH$&$n\ge2$&$HS^2_0\bH^n$ \\
\hline
4&${\rm SU}_{2n+1}$&$\bC^{2n+1}$&$\Lambda^2\bC^{2n+1}$&$n\ge2$& \\
5&${\rm U}_{2n+1}$&$\bC^{2n+1}$&$\Lambda^2\bC^{2n+1}\oplus\bR$&$n\ge1$&$\Lambda^2\bC^{2n+1}$ \\
6& ${\rm SU}_{2n}$& $\bC^{2n}$& $\Lambda^2\bC^{2n}\oplus\bR$&$n\ge2$& $\Lambda^2\bC^{2n}$\\
\hline
7&${\rm U}_2\times {\rm SU}_n$&$\bC^2\otimes\bC^n$&$ \fu_2$& $n\ge2$&$\su_2$\\
8&${\rm U}_2\times\text{Sp}_n$&$\bC^2\otimes\bC^{2n}$&$ \fu_2$ &$n\ge2$&$\su_2$\\
9&${\rm U}_1\times\text{Spin}_7$&$\bR^2\otimes\bO$&$\IM\bO\oplus\bR$&&$\IM\bO$\\
10&$\text{Sp}_2\times\text{Sp}_n$&$\bH^2\otimes\bH^n$&$\fs\fp_2$&$n\ge1$&\\
\hline
\end{tabular}
\end{center}
\bigskip
\caption{The reduced list of pairs satisfying Vinberg's condition}
\label{vinberg}
\end{table}

All pairs in Table~\ref{vinberg} admit Hilbert bases ${\boldsymbol\rho}$ which are free  of relations and satisfy the conditions of Subsection \ref{special}. They are given in  \cite[Section 7]{FRY1} and are reproduced in Table~\ref{invariants} at the end of the paper. For unexplained notation, we refer to \cite{FRY1}.
 
It is apparent in Table \ref{vinberg} that
$K=K_0\times K'$ (with $K'$ possibly trivial), where the action of $K_0$ on $\fz$ is faithful (up to a finite subgroup at most) and that of $K'$ is trivial.

We say that an element of $\fz$ is {\it regular} if its $K$-orbit has maximal dimension and {\it singular} otherwise.  Clearly, regularity of an element only depends  on its $\fz_0$-component.

\vskip1cm
\section{Quotient pairs}\label{quotient}
\bigskip

Given $t\in\fz$, we set $\ft_t=\fk\cdot t$, i.e., the tangent space in $t$ to the $K$-orbit  $K\cdot t$, and $\fz_t=(\fk \cdot t)^\perp$.  We consider the quotient algebra $\fn_t=\fn/\ft_t$, and denote the canonical projection by ${\rm proj}_t$. As in Subsection~\ref{modulo-s}, we regard $\fn_t$ as $\fv\oplus\fz_t$, with Lie bracket $[v,v']_{\fn_t}={\rm proj}_t[v,v']$. By $N_t$ we denote the quotient group $N/\exp\ft_t$.

Since $\ft_t$ and $\fz_t$ are invariant under the action of $K_t$, the stabilizer of $t$ in $K$,  passing to the quotient we obtain an action of $K_t$ on $N_t$. We call $(N_t,K_t)$ a {\it quotient pair} of $(N,K)$. By \cite{C, V1}, $(N_t,K_t)$ is a n.G.p.

The quotient pairs generated by the pairs in Table \ref{vinberg} are listed in the Appendix at the end of the paper. Here we only point out the following facts.

\begin{remark}\label{QP-sec5}
{\rm For the pairs at lines 1-6, the quotient pairs are products of pairs in the same list, possibly up to a central reduction. 
In detail, pairs at line 1 bring in pairs at line 2 plus pairs at line 1 of lower rank. Pairs at lines 2 and 3 form self-contained families. Pairs at lines 4,\,5, and 6 lead to pairs at line 3 and pairs  of lower rank in the same family. 

For the pairs at lines 7-10, we obtain instead new pairs, which are indecomposable, but do not satisfy Vinberg's condition. }
\end{remark}

\bigskip
\subsection{Slices and radialisation}\label{sec-slices}\quad
\medskip

Suppose that a compact real Lie group $K$ acts 
on a linear space $W$. Take any $x\in W$ and let 
$K_x$ be the stabiliser of $x$ in $K$. Let also 
$N_x:=(\fk\cdot x)^\perp$ be the normal space to the orbit $Kx$ at $x$.
A construction of a {\it slice} for a compact group action goes back to 
Gleason \cite{Glea}. We will need the ``linear" version of the slice theorem.

\begin{theorem}\label{Slice-Thm}
There is an open and $K_x$-invariant (Euclidean) neighbourhood $S_x$ of $0$ in $N_x$
such that the $K$-equivariant map
$$
\sigma:K\times_{K_x} S_x \longrightarrow W,
$$
given by $\sigma(k,y)= k(x+y)$, 
is
a diffeomorphism of $K\times_{K_x}S_x$ onto the open neighbourhood $K(x+S_x)$ of $Kx$.
\end{theorem} 

We call $S_x$ a {\it slice at} $x$. The notation $K\times_{K_x} S_x$ stands for the quotient of $K\times S_x$ modulo  the action of $K_x$, i.e., $(kk',x)$ is equivalent to $(k,k'x)$ for $k'\in K_x$.


The proof of Theorem~\ref{Slice-Thm} can be found, for instance, in \cite[Ch.~2,\,Section~5]{bre}, 
in particular, see Corollary~5.2 therein. 
The theorem has the following almost immediate consequence (for part (i), see e.g. \cite[Ch.~2,\,Sections~4,~5]{bre}).
  
\begin{corollary}\label{Slice-cor}\quad
\begin{enumerate}
\item[\rm(i)] For every $y\in S_x$ we have the inclusion $K_y\subset K_x$, more explicitly  
$K_y=(K_x)_y$.
\item[\rm(ii)]  Two points in $S_x$ are conjugate under $K$ if and only if they are conjugate under~$K_x$.
\item[\rm(iii)] Suppose that $f$ is a $K_x$-invariant smooth function 
on $S_x$.  Then $f$
extends in a unique way to a smooth $K$-invariant function $f^{\rm rad}$ 
on $K(x+S_x)$.
\end{enumerate}
\end{corollary}

We call $f^{\rm rad}$ the {\it radialisation of} $f$. 
\smallskip

\begin{remark}
{\rm This notion of radialisation extends the one used in \cite{FR} to  general pairs, including the rank-one pairs considered in \cite{FRY1}. We notice that  the  arguments in the rest of this paper will not rely on the results of \cite{FRY1}, which in fact are being given a different proof.}
\end{remark}

We will apply Theorem~\ref{Slice-Thm} and Corollary~\ref{Slice-cor} to the action of $K$ on $\fz$.  
If 
$t\in\fz$ and $S_t$ is a slice at $t$ in $\fz$,  then 
$\fv\times S_t$ 
is a slice at $t$ for the $K$-action on $\mathfrak v{\oplus}\mathfrak z$. 
(The injectivity on the ``$\mathfrak z$"-side and zero at the ``$\mathfrak v$"-side of $t$ allow us to include $\mathfrak v$ into 
the slice.)

\bigskip
\subsection{Relations between invariants for $(N,K)$ and invariants on its quotient pairs}\label{relations-invariants}$\!\!{}$
\medskip

For $t\in\fz$, we consider the quotient pair $(\fn_t,K_t)$ defined at the beginning of this section.
\smallskip

 Let ${\boldsymbol\rho}$, resp. ${\boldsymbol\rho}^t$, be a minimal Hilbert basis for $(\fn,K)$, resp. $(\fn_t,K_t)$, whose elements are homogeneous in both the $\fv$- and the $\fz$-variables. Denote by ${\boldsymbol\rho}_{(k)}$, resp. ${\boldsymbol\rho}^t_{(k)}$ the invariants in ${\boldsymbol\rho}$, resp. ${\boldsymbol\rho}^t$, which have degree $k$ in the $\fv$-variables. Then ${\boldsymbol\rho}_{(0)}={\boldsymbol\rho}_\fz$ and ${\boldsymbol\rho}^t_{(0)}={\boldsymbol\rho}^t_{\fz_t}$.

\begin{remark}
\label{rem_cardinality}
{\rm We point out some identities concerning cardinalities of Hilbert bases for a given pair and its quotient pairs. Strictly speaking, none of these facts is  needed in the sequel. We will nevertheless assume them, both for notational convenience, and to avoid unnecessary inductions on $k$ in some of the proofs in this section. 

\begin{enumerate}
\item[(i)]  The cardinality of each of the sets ${\boldsymbol\rho}_{(k)}$,  ${\boldsymbol\rho}^t_{(k)}$ does not depend on the choice of the minimal Hilbert basis;
\item[(ii)] for every $k$, ${\boldsymbol\rho}_{(k)}$ and ${\boldsymbol\rho}^t_{(k)}$ have the same number of elements;
\item[(iii)] the different values of $k$ that appear are at most three, either $0,2$ or $0,1,2$ or $0,2,4$.
\end{enumerate}

To have a unified notation, we will denote the values of $k$ in (iii) as $0,k_0$ and, when present,~$2k_0$.

The first statement is perfectly clear, 
for example, the cardinality of ${\boldsymbol\rho}_{(k)}$ is equal to the dimension (over $\mathbb R$) of 
${\mathcal A}_k/ ({\mathcal A}_k\cap {\mathcal A}^2)$, 
where  ${\mathcal A}=(\fS^{>0}(\gt n))^K$ and   ${\mathcal A}_k=\mathcal A\cap \fS^k(\gt v){\otimes}\fS(\gt z)$.  

Since the $K$-action on $\gt n$ admits a free Hilbert basis, see \cite[Section 7]{FRY1}, this property is inherited by 
$(\fn_t,K_t)$, see e.g. \cite[Theorem~8.2]{PV}.
In particular, a minimal Hilbert basis ${\boldsymbol\rho}^t$ is also free of relations. 
Recall next that $K$-orbits in $\gt n$ (as well  as $K_t$-orbits in $\gt n_t$) of maximal dimension form a Zariski open subset, the so-called subset of {\it generic} points. 
Being a non-empty open subset, $\gt v{\times}S_t$ contains a generic point, say $\zeta$, of $\gt n_t$.
By Theorem~\ref{Slice-Thm}, $\zeta$ is also generic in $K(\fv{\times}(t+S_t))$ 
and hence in $\gt n$.  We see also that $\gt z_\zeta=(\gt z_t)_\zeta$ and  $K_\zeta=(K_t)_\zeta$. 
Since orbits of a compact group are closed and closed orbits are always separated  
by polynomial invariants, 
the number of elements in 
${\boldsymbol\rho}_{(0)}$ (resp. ${\boldsymbol\rho}^t_{(0)}$) is equal to 
the codimension of a generic $K$-orbit in $\gt z$ (resp. $K_t$-orbit in  $\gt z_t$), i.e., $\dim\gt z_\zeta$ ($\dim(\gt z_t)_\zeta$, respectively), see e.g. \cite[Sections~2.3,~4]{PV}. 
This takes care of the case $k=0$. 
For $k>0$, the statement of part~(ii) follows now from Lemma~\ref{k-numbers}  applied to both $(K,\gt n)$ and 
$(K_t,\gt n_t)$.

Statement (iii) is evident from Table~\ref{invariants} and Table~\ref{quotient-invariants} in Section~\ref{third-block}.
}\end{remark}

\begin{lemma}\label{k-numbers}
Suppose that we have a linear  action of a compact group $K$ on 
$\gt v{\oplus}\gt z$ and the algebra of invariants $\mathbb R[\gt v{\oplus}\gt z]^K$ 
has a free Hilbert basis. Then 
the number of generating bi-homogeneous  invariants of degree $k>0$ in $\gt v$ is the same as the number of generators of degree $k$ in a minimal Hilbert basis in 
$\mathbb R[\gt v]^{K_\zeta}$, where 
$\zeta\in\gt z$ and
the orbit $K\zeta\subset\gt z$ has the maximal possible dimension.
\end{lemma}
\begin{proof}
In view of Theorem~\ref{Slice-Thm}, the maximality of $\dim K\zeta$ implies that the action of $K_\zeta$ on $\gt z_\zeta$ is trivial.
Consider now the composition 
$$
\mathbb R[\gt n]^K\to \mathbb R[\gt v{\oplus}\gt z_\zeta]^{K_\zeta} \to \mathbb R[\gt v\times\{\zeta\}]^{K_\zeta}=\mathbb R[\gt v]^{K_\zeta}
$$
of two restriction morphisms. It is known to be surjective, see 
e.g. \cite[proof of Theorem~1.3]{Y2}. The elements of $\mathbb R[\gt z]$ 
are constant on $\gt v\times\{\zeta\}$. Since 
$K$-orbits are closed, ${\rm tr.deg}\,\mathbb R[\gt z]^K=\dim\gt z-\dim K\zeta$
(the codimension of a generic orbit, see e.g. \cite[Sections~2.3,~4]{PV}) and 
since the whole algebra of $K$-invariants on $\gt  n$ is free, 
there are exactly $\dim\gt z_\zeta$ generators in $\mathbb R[\gt z]^K$.  
The remaining generators of $\mathbb R[\gt n]^K$ restrict to the generators of $\mathbb R[\gt v]^{K_\zeta}$, because 
${\rm tr.deg}\,\mathbb R[\gt n]^K-{\rm tr.deg}\,\mathbb R[\gt v]^{K_\zeta}=\dim\gt z_\zeta$ 
by the same minimal codimension reasoning.
%
\end{proof}

We label the elements of ${\boldsymbol\rho}_{(k)}$ as $(\rho_{k,1},\dots,\rho_{k,d_k})$ and those of 
${\boldsymbol\rho}^t_{(k)}$ as $(\rho^t_{k,1},\dots,\rho^t_{k,d_k})$.

When we restrict an element of ${\boldsymbol\rho}$ to $\fn_t$, it can be expressed as a polynomial in the elements of 
${\boldsymbol\rho}^t$. By degree considerations, it is quite clear that
\begin{equation}\label{Q}
{\rho_{k,j}}_{|_{\fn_t}}(v,z)=
\begin{cases} 
Q_{0,j}\big({\boldsymbol\rho}^t_{(0)}(z)\big)&\text{ if }k=0\\
\displaystyle{\sum_{\ell=1}^{d_{k_0}}}Q_{k_0,j,\ell}\big({\boldsymbol\rho}^t_{(0)}(z)\big)\rho^t_{k_0,\ell}(v,z)&\text{ if }k=k_0\\
\begin{aligned}
\displaystyle{\sum_{\ell=1}^{d_{2k_0}}}&Q_{2k_0,j,\ell}\big({\boldsymbol\rho}^t_{(0)}(z)\big)\rho^t_{2k_0,\ell}(v,z)\\&+\displaystyle{\sum_{\ell,\ell'=1}^{d_{k_0}}}R_{2k_0,j,\ell,\ell'}\big({\boldsymbol\rho}^t_{(0)}(z)\big)\rho^t_{k_0,\ell}(v,z)\rho^t_{k_0,\ell'}(v,z)
\end{aligned}&\text{ if }k=2k_0\ ,
\end{cases}
\end{equation}
where the $Q$'s and the $R$'s are polynomials.
Then the map $\cQ$ from $\bR^d=\bR^{d_0+d_{k_0}+d_{2k_0}}$ to itself, whose components $Q_{k,j}$ are
\begin{equation}\label{Qcomponents}
\begin{aligned}
Q_{0,j}&=Q_{0,j}(\xi_{(0)})\ ,\\  Q_{k_0,j}&=\sum_{\ell=1}^{d_{k_0}}Q_{k_0,j,\ell}(\xi_{(0)})\xi_{k_0,\ell}\ ,\\ Q_{2k_0,j}&=\sum_{\ell=1}^{d_{2k_0}}Q_{2k_0,j,\ell}(\xi_{(0)})\xi_{2k_0,\ell}+\sum_{\ell,\ell'=1}^{d_{k_0}}R_{2k_0,j,\ell,\ell'}(\xi_{(0)})\xi_{k_0,\ell}\xi_{k_0,\ell'}\ ,
\end{aligned}
\end{equation}
 maps ${\boldsymbol\rho}^t\big(\fv\times(t+ S_t)\big)$ into ${\boldsymbol\rho}\big(\fv\times K(t+S_t)\big)$. Since the elements of ${\boldsymbol\rho}^t\big(\fv\times(t+ S_t)\big)$, resp. ${\boldsymbol\rho}\big(\fv\times K(t+S_t)\big)$, are in 1-to-1 correspondence with $K_t$-orbits in $\fv\times (t+S_t)$, resp. $K$-orbits in $\fv\times K(t+S_t)$, it follows from Corollary \ref{Slice-cor} that $\cQ$ is a bijection from ${\boldsymbol\rho}^t\big(\fv\times(t+ S_t)\big)$ onto ${\boldsymbol\rho}\big(\fv\times K(t+S_t)\big)$.
In particular, $\cQ_0=(Q_{0,j})_{1\le j\le d_{(0)}}$ is a bijection from ${\boldsymbol\rho}^t_{\fz_t}(t+S_t)$ onto ${\boldsymbol\rho}_\fz\big(K(t+S_t)\big)$. 

We show now that  $\cQ\inv$  is the restriction of a smooth map on a neighbourhood  in $\bR^d$ of ${\boldsymbol\rho}_{(0)}(t)\times\bR^{d_{k_0}+d_{2k_0}}$.

\begin{proposition}\label{Qinv}
 For $t\in\fz$, let $S_t$ be as in Corollary \ref{Slice-cor}, and let $U$ be a Euclidean neighbourhood of~$t$,  $K_t$-invariant and relatively compact in $t+S_t$. Then there exists a smooth function $\Phi$ from $\bR^d$ to $\bR^d$ which inverts $\cQ$ on ${\boldsymbol\rho}(\fn)\cap\big({\boldsymbol\rho}_{(0)}(U)\times\bR^{d_{k_0}+d_{2k_0}}\big)$. The scalar components 
$\Phi_{k,j}$ of $\Phi$ have the following form:
\begin{enumerate}
\item[\rm(i)] each $\Phi_{0,j}$ only depends on $\xi_{(0)}$;
\item[\rm(ii)] $\Phi_{k_0,j}(\xi)=\sum_{\ell=1}^{d_{k_0}}\Phi_{k_0,j,\ell}(\xi_{(0)})\xi_{k_0,\ell}$;
\item[\rm(iii)] $\Phi_{2k_0,j}(\xi)=\sum_{\ell=1}^{d_{2k_0}}\Phi_{2k_0,j,\ell}(\xi_{(0)})\xi_{2k_0,\ell}+\sum_{\ell,\ell'=1}^{d_{k_0}}\Psi_{2k_0,j,\ell,\ell'}(\xi_{(0)})\xi_{k_0,\ell}\xi_{k_0,\ell'}$.
\end{enumerate}

In particular, $\Phi_0=(\Phi_{0,j})_{1\le j\le d_\fz}$ inverts $\cQ_0$ on ${\boldsymbol\rho}_\fz(U)$.
\end{proposition}

\begin{proof}
Let $\chi\in C^\infty_c(t+S_t)$ be $K_t$-invariant and equal to 1 on  $U$.   For $j=1,\dots,d_\fz$, set
$u_j(z)=(\chi \rho^t_j)^{\rm rad}(z)$ on $K(t+S_t)$ and extend it to  $\fz$ as 0. By Corollary \ref{Slice-cor} (iii), $u_j$ is smooth, so that, G. Schwarz's theorem \cite{Schw}, $u_j=\Phi_{0,j}(\rho_1,\dots,\rho_{d_\fz})$, with $\Phi_{0,j}$ smooth on $\bR^{d_0}=\bR^{d_\fz}$. 

Then
\begin{equation}\label{j-le-n}
\rho^t_j=\Phi_{0,j}(\rho_1,\dots,\rho_{d_\fz})\ ,\qquad j=1,\dots,d_\fz\ ,
\end{equation}
on $U$.
Combining this with the first line of \eqref{Q}, we conclude that the map 
$$
\Phi_0=(\Phi_{0,j})_{1\le j\le d_\fz}:{\boldsymbol\rho}_{\fz}(KU)\longrightarrow {\boldsymbol\rho}^t_{\fz_t}(U)
$$ 
is the inverse of the map $\cQ_0$. 

We apply now the same construction to the elements of ${\boldsymbol\rho}_{(k_0)}$. The radialisation of $\chi\rho^t_{k_0,\ell}$, defined on $\fv\times \fz$, remains a polynomial of degree $k_0$ in the $\fv$-variable. Applying Schwarz's theorem again and using Proposition 2.1 of \cite{FRY2} (with the r\^oles of $\fv$ and $\fz$ interchanged), we have that
$$
(\chi\rho^t_{k_0,\ell})^{\rm rad}=\sum_{j'=1}^{d_{(k_0)}}(\Phi_{k_0,\ell,j'}\circ{\boldsymbol\rho}_{(0)})\rho_{k_0,j'}\ ,
$$
with $\Phi_{k_0,\ell,j'}\in C^\infty(\bR^{d_\fz})$, and this gives (ii).

In the same way we obtain, for a polynomial in ${\boldsymbol\rho}_{(2k_0)}$, that
$$
\begin{aligned}
(\chi\rho^t_{2k_0,\ell})^{\rm rad}&=\sum_{j'=1}^{d_{(2k_0)}}(\Phi_{2k_0,\ell,j'}\circ{\boldsymbol\rho}_{(0)})\rho_{2k_0,j'}\\
&\qquad +\sum_{j',j''=1}^{d_{(k_0)}}(\Psi_{k_0,\ell,j',j''}\circ{\boldsymbol\rho}_{(0)})\rho_{k_0,j'}\rho_{k_0,j''}\ ,
\end{aligned}
$$
with $\Phi_{2k_0,\ell,j'},\Psi_{k_0,\ell,j',j''}\in C^\infty(\bR^{d_\fz})$, and this gives (iii).
\end{proof}

\vskip1cm
\section{Smooth maps between spectra}\label{sec_smooth}\label{relations-spectra}\quad
\bigskip

Let $(N,K)$ be a pair in Table~\ref{vinberg}.
Following the notation of \eqref{lift-phi},
to a bounded spherical function $\ph^t$ for the pair $(N_t,K_t)$, we associate the bounded spherical function for $(N,K)$
\begin{equation} \label{Lambda^gamma}
\Lambda^t\ph^t(v,z) =\int_K(\ph^t\circ{\rm proj}_t)(kv,kz)\,dk\ .
\end{equation}

\begin{proposition}\label{epi}
Then $\Lambda^t$ is a continuous surjection of $\Sigma^t$ onto $\Sigma$.
\end{proposition}

\begin{proof}
Given $\zeta\in\fa$ and $\omega\in\fr_\zeta$ , consider the two decompositions of $\cH_\zeta$ into irreducible components under the action of $K_{\zeta,\omega}$ and $(K_t)_{\zeta,\omega}$ respectively,
$$
\cH_\zeta=\sum_{\mu\in\fX_{\zeta,\omega}}V(\mu)\ ,\qquad \cH_\zeta=\sum_{\nu\in\fX^t_{\zeta,\omega}}V^t(\nu)\ .
$$

Since both decompositions are multiplicity-free and $(K_t)_{\zeta,\omega}\subseteq K_{\zeta,\omega}$, each $V(\mu)$ is a finite union of $V^t(\nu)$.

The spherical function $\ph^t_{\zeta,\omega,\nu}\in\Sigma^t$ is, by \eqref{trace},  
$$
\ph^t_{\zeta,\omega,\nu}(v,z)=\frac1{\dim V^t(\nu)}\int_{K_t}\tr\big(\pi_{\zeta,\omega}(hv,hz)_{|_{V^t(\nu)}}\big)\,dh\ .
$$

By Schur's lemma, for every $\mu\in\fX_{\zeta,\omega}$ and every unit element $e\in V(\mu)$,
$$
\int_K\lan\pi_{\zeta,\omega}(kv,kz)e,e\ran\,dk=\frac1{\dim V(\mu)}\int_K\tr\big(\pi_{\zeta,\omega}(kv,kz)_{|_{V(\mu)}}\big)\,dk\ .
$$

It follows easily that, if $V^t(\nu)\subseteq V(\mu)$, then $\Lambda^t\ph^t_{\zeta,\omega,\nu}=\ph_{\zeta,\omega,\mu}$. Continuity of $\Lambda^t$ with respect to the compact-open topologies on the two spectra is obvious from \eqref{Lambda^gamma}.
\end{proof}

We point out that the proof shows the following identity:
\begin{equation}\label{Lambda-inv}
(\Lambda^t)\inv (\ph_{\zeta,\omega,\mu})=\{\ph^t_{\zeta,\omega,\nu}:V^t(\nu)\subseteq V(\mu)\}\ .
\end{equation}

This has the following consequence. 

\begin{corollary}\label{bijection}
 For $t\in\fz$,  let $S_t$ be as in Corollary \ref{Slice-cor}. If
\begin{equation}\label{tildeOmega}
\tilde S_t=\{\ph_{\zeta,\omega,\mu}:\zeta\in t+S_t\}\subset\Sigma\ ,
\end{equation}
then $\Lambda^t$ is a bijection from $(\Lambda^t)\inv(\tilde S_t)$ to $\tilde S_t$.
\end{corollary}

\begin{proof} For $\zeta\in S_t$, $K_\zeta=(K_t)_\zeta$ by Corollary \ref{Slice-cor} (i). Hence, by \eqref{Lambda-inv} and \eqref{trace}, $\Lambda^t$ is injective.
\end{proof}

The same proof shows that the following more general statement holds.

\begin{corollary} \label{graph}
If $t,t'\in\fz$ and $K_t\subset K_{t'}$ (hence $\fz_t\subset\fz_{t'}$),  the map
\begin{equation} \label{Lambda^gamma_gamma'}
\Lambda^t_{t'}\ph^t(v,z) =\int_{K_{t'}}(\ph^t\circ{\rm proj}_t)(kv,kz)\,dk\ ,\qquad (z\in\fz_{t'})
\end{equation}
is continuous and surjective from $\Sigma^t$ to $\Sigma^{t'}$, and
\begin{equation}\label{inverse}
(\Lambda^t_{t'})\inv (\ph^{t'}_{\zeta,\omega,\nu'})=\{\ph^t_{\zeta,\omega,\nu}:V^t(\nu)\subseteq V^{t'}(\nu')\}\ .
\end{equation}

If $t,t',t''$ are such that $K_t\subset K_{t'}\subset K_{t''}$ then 
$$
\Lambda^t_{t''}=\Lambda^{t'}_{t''}\circ\Lambda^t_{t'}\ .
$$

 Moreover, $S_t$ satisfies the conditions of Corollary \ref{Slice-cor} for the action of $K_{t'}$ on $\fz_{t'}$ and $\Lambda^t_{t'}$ is invertible on $(\Lambda^t)\inv(\tilde S_t)$, with $\tilde S_t$ the set in \eqref{tildeOmega}.
\end{corollary}

The whole picture is represented by  the following ``dual'' commutative diagrams, representing the different ``levels of singularity'' of elements of $\fa$. We then refer to the elements of $\check\fz$ as the ``most singular'' elements of $\fa$. The sub- or super-script ``reg'' stands for regular elements.  

$$
\begin{array}[c]{ccccc} &&\fv\oplus\fa&&\\
&{\scriptstyle{\rm proj}_{\rm reg}}\nearrow&\cdots&\nwarrow\scriptstyle{{\rm proj}_{\rm reg}}&\\
&\fv\oplus\fz_t\phantom{A}&\cdots&\phantom{AA}\fv\oplus\fz_{t'}&\\
&\vdots\phantom{a}&&\phantom{a}\vdots&\\
&\fv\oplus\fz_{t''}\phantom{A}&&\phantom{AA}\fv\oplus\fz_{t'''}&\\
&{\scriptstyle{\rm proj}_{t''}}\nwarrow\phantom{a}&\cdots&\phantom{a}\nearrow\scriptstyle{\rm proj}_{t'''}&\\
&&\fv\oplus\fz&&\\
\end{array}
\qquad
\begin{array}[c]{ccccc} &&\Sigma^{\rm reg}&&\\
&{\scriptstyle \Lambda^{\rm reg}_t}\swarrow&\cdots&\searrow\scriptstyle\Lambda^{\rm reg}_{t'}&\\
&\Sigma^t\phantom{A}&\cdots&\phantom{AA}\Sigma^{t'}&\\
&\vdots\phantom{a}&&\phantom{a}\vdots&\\
&\Sigma^{t''}\phantom{A}&&\phantom{AA}\Sigma^{t'''}&\\
&{\scriptstyle\Lambda^{t''}}\searrow\phantom{A}&\cdots&\phantom{A}\swarrow\scriptstyle\Lambda^{t'''}&\\
&&\Sigma^{\check\fz}=\Sigma&&\\
\end{array}
$$


Assuming that two homogeneous bases $\cD$, resp. $\cD_t$, of  $\bD(N)^K$, resp $\bD(N_t)^{K_t}$, have been fixed, we regard  the map $\Lambda^t$ in \eqref{Lambda^gamma} as a map from $\Sigma^t_{\cD_t}$ to $\Sigma_\cD$.

We want to realize $\Lambda^t$, and its local inverse on $\tilde S_t$, as  restrictions of  smooth maps on  open subsets of the ambient space\footnote{We do not require these extensions to be diffeomorphisms of such open sets.}. 
Notice that the existence of such smooth extensions is independent of the choice of $\cD$ and $\cD_t$, because of \eqref{changeD}.

Denote by $\cR^t$ the Radon transform \eqref{opradon*} adapted to the projection ${\rm proj}_t$ of $\fn$ onto $\fn_t$.
By \eqref{radon-eigenvalues}, the eigenvalue  $\xi(\cR^t D,\ph^t)$ is the same as the eigenvalue of $\xi(D,\Lambda^t\ph^t)$.


\begin{proposition}\label{Lambda-smooth}
Let $\cD$, resp. $\cD_t$, be homogeneous bases of $\bD(N)^K$, resp. $\bD(N_t)^{K_t}$, consisting of $d$, resp.~$d'$, elements. Then
$\Lambda^t$, regarded as a map from $\Sigma^t_{\cD_t}$ to $\Sigma_\cD$, is the restriction of a smooth map from~$\bR^d$ to~$\bR^{d'}$.
\end{proposition}

\begin{proof}
Let $\cD=(D_1,\dots,D_d)$. Since the statement does not depend on the choice of $\cD_t$, we construct one starting from $(\cR^t D_1,\dots,\cR^t D_d)$ and completing it into a generating system 
$$
\cD_t=(\cR^t D_1,\dots,\cR^t D_d,D^t_{d+1},\dots,D^t_{d'})
$$ 
of $\bD(N_t)^{K_t}$. By Proposition \ref{propradon*}, $\Lambda^t$ is the restriction to $\Sigma_{\cD_t}$ of the canonical projection from $\bR^{d'}$ to its first $d$-dimensional coordinate subspace. 
\end{proof}

From now on, we denote by $\cD$, resp. $\cD_t$, the system of differential operators  obtained by symmetrisation, resp. on $N$ and $N_t$, from  the Hilbert bases ${\boldsymbol\rho}$, resp. ${\boldsymbol\rho}^t$, of Section \ref{relations-invariants}. 
Then $\Sigma_\cD$ and $\Sigma^t_{\cD_t}$ are subsets of the same space $\bR^d$.

By Lemma \ref{Pi}, for each point $\xi_\fz\in{\boldsymbol\rho}_\fz(\fz)$, there is an associated subset of $\Sigma_\cD$, denoted by $\Pi\inv_{|_{\Sigma_\cD}}(\xi)$, whose elements $(\xi_\fz,\xi_\fv,\xi_{\fv,\fz_0})$ correspond to the spherical $\ph_{\zeta,\omega,\mu}$ in \eqref{trace} with ${\boldsymbol\rho}_\fz(\zeta)=\xi_\fz$. 
The same applies to $\Sigma^t_{\cD_t}$ and a point $\xi_{\fz_t}\in{\boldsymbol\rho}^t_{\fz_t}(\fz_t)$.

\begin{proposition}\label{tildeQ}
Let $t\in\fz$, and let $U\subset S_t$ be as in Proposition~\ref{Qinv}. 
Then $\Lambda^t$ coincides with a map $\tilde \cQ$, obtained from $\cQ$ in \eqref{Qcomponents} by adding a linear term in $\xi_{(k_0)}$ to the components $Q_{2k_0,j}$. Then $\tilde \cQ$ is still invertible on $\Sigma_{\cD_t}\cap \Pi\inv\big({\boldsymbol\rho}^t_{\fz_t}(U)\big)$, and $(\Lambda^t)\inv=\tilde \cQ\inv$ is the restriction of a smooth map $\tilde\Phi$, obtained by adding a linear term in $\xi_{(k_0)}$ to the map $\Phi$ of Proposition~\ref{Qinv}.
\end{proposition}

\begin{proof} Take a point $\xi\in\Sigma_{\cD_t}$ with $(\xi_1,\dots,\xi_{d_\fz})\in{\boldsymbol\rho}_{\fz_t}(U)$, and let $\tilde\xi=\Lambda^t\xi\in\Sigma_\cD$.

If $\ph^t$ is the $K_t$-spherical function on $N_t$ associated with $\xi$, then $\ph=\Lambda^t\ph^t$, is the $K$-spherical function on $N$ associated with $\tilde\xi$.

By Proposition \ref{propradon*}, since $D_\ell\ph=\tilde\xi_\ell\ph$ for $D_\ell\in\cD$, we also have that
$(\cR^t D_\ell)\ph^t=\tilde\xi_\ell\ph_t$.
Moreover, 
$$
\cR^t D_\ell=\la'({\rho_\ell}_{|_{\fn_t}})=\la'\big(Q_\ell(\rho^t_1,\dots,\rho^t_d)\big)\ ,
$$
where $Q_\ell$ is the polynomial in \eqref{Q}. 

Notice that, by \eqref{Q}, $\la'\big(Q_\ell(\rho^t_1,\dots,\rho^t_d)\big)$ differs from $Q_\ell(D^t_1,\dots,D^t_d)$, i.e., $Q_\ell\big(\la'(\rho^t_1),\dots,\la'(\rho^t_d)\big)$, only for the components of degree $2k_0$, in the terms containing a product of two $\rho_{k_0,j}$.  It is sufficient for us to observe that 
$$
\la'(\rho^t_{k_0,\ell}\rho^t_{k_0,\ell'})=\la'(\rho^t_{k_0,\ell})\la'(\rho^t_{k_0,\ell'})+ \la'(\text {linear combination of elements of }{\boldsymbol\rho}_{(k_0)})\ .
$$

The conclusion follows easily.
\end{proof}

\vskip1cm
\section{ Extending Gelfand transforms: reduction to quotient pairs}
\label{sec_towards}
\bigskip

Denote by $\cS_0(N)$ the space of  function $F\in\cS(N)$ with vanishing moments of any order in the $\fz_0$-variables, i.e., such that
$$
\int_{\fz_0}z^\beta F(v,z,u)\,dz=0\ ,
$$
for every $\beta\in\bN^{d_{\fz_0}}$, $v\in\fv$, $u\in \check\fz$.

In this section we prove the following statement.

\begin{proposition}\label{S_0}
Let $(N,K)$ be one of the pairs in Table \ref{vinberg}, and assume that Property (S) holds for all its proper quotient pairs $(N_t,K_t)$.
If $F\in\cS_0(N)^K$, then its spherical transform, defined on $\Sigma_\cD$, can be extended to a function in $\cS(\bR^d)$ which vanishes with all its derivatives on $\{0\}\times\bR^{d_{\check\fz}}\times\bR^{d_\fv}\times\bR^{d_{\fv,\fz_0}}$. 
\end{proposition} 

The set $\{0\}\times\bR^{d_{\check\fz}}\times\bR^{d_\fv}\times\bR^{d_{\fv,\fz_0}}$ coincides with $\Pi\inv\big({\boldsymbol\rho}_\fz(\check\fz)\big)$. Its intersection with $\Sigma_\cD$ is the set $\Sigma_\cD^0$ of points with  ``highest level of singularity'', i.e., associated to trivial orbits in $\fz$, cf. \cite{FRY2}. Under the identification of $\Sigma$ with $\Sigma_\cD$, it corresponds to the set of spherical functions which are identically equal to 1 on $\fz_0$, or, equivalently, which are associated to representations of $N$ that are trivial on $\exp\fz_0$. 

To prove Proposition \ref{S_0}, we will work on the complement of $\Sigma^0_\cD$, i.e. at points which are ``regular'' or of ``intermediate singularity''. It is on  neighbourhoods of these points that we have  at our disposal the local identification of Proposition \ref{tildeQ} with spectra of quotient pairs.

\bigskip

\subsection{Dilations and partitions of unity}\label{dilations-partitions}\quad
\medskip

Coherently with the previous comments, in this section we restrict our attention to elements $t$ of $\fz$ which are not in $\check\fz$. Then $(N_t,K_t)$ will be a proper quotient pair of $(N,K)$.

The constructions in Sections \ref{relations-invariants} and \ref{relations-spectra} present a natural homogeneity with respect to the dilations on $\fz_0$, as well as a translation-invariance in $\check\fz$. Precisely, if $U\subset t+S_t$ is the neighbourhood of $t\ne0$ in $\fz_t$ for which Propositions \ref{Qinv} and \ref{tildeQ} hold, then the same hold for
\begin{enumerate}
\item[(i)]   the neighbourhood $\del U$ of $\del t$, for $\del>0$;
\item [(ii)]  the neighbourhood $ U+u$ of $t+u$, for $u\in\check\fz$.
\end{enumerate}

Also notice that, since $K$ acts trivially on $\check\fz$, the $\check\fz$-variables only appear in the components of ${\boldsymbol\rho}_{\check\fz}$ and ${\boldsymbol\rho}_{\check\fz_t}$. 

All this has the following implications on the maps of Propositions \ref{Qinv} and \ref{tildeQ}, which we denote now by $\cQ_t$, $\Phi_t$, $\tilde\cQ_t$, $\tilde\Phi_t$.
\begin{enumerate}
\item[(i)] The maps $\cQ_t$, $\Phi_t$, $\tilde\cQ_t$, $\tilde\Phi_t$  contain the identity function in the $\xi_{\check\fz}$-components, and all the other components do not involve the $\xi_{\check\fz}$-variables.
\item[(ii)]  In the statements of Propositions \ref{Qinv} and \ref{tildeQ} we may assume that $t\in\fz_0$ and choose $S_t$ and $U$ of the form $S_t=S_{0,t}+\check\fz$, with  $S_{0,t}=S_t\cap\fz_0$,   and $U=U_{0,t}+\check\fz$, for a neighbourhood $U_{0,t}$ of $t$ relatively compact in $t+S_{0,t}$.
\item[(iii)]  Let $D(\del)$, resp. $D^t(\del)$, be the dilations \eqref{dilations} on $\bR^d$, with exponents $\nu_j$, resp. $\nu^t_j$, equal to the degrees of homogeneity of the elements of ${\boldsymbol\rho}$, resp. ${\boldsymbol\rho}^t$. Then, for $\del>0$, $\cQ_{\del t}=\cQ_t$, $\tilde\cQ_{\del t}=\tilde\cQ_t$ and
$$
D(\del)\circ \cQ_t=\cQ_t\circ D^t(\del)\ ,\qquad D(\del)\circ \tilde\cQ_t= \tilde\cQ_t\circ D^t(\del)\ ,
$$
\item[(iv)] once $\Phi_t$, $\tilde\Phi_t$ have been chosen for $t$ with $|t|=1$, $\Phi_{\del t}$, $\tilde\Phi_{\del t}$ can be chosen, for $\del>0$,  as
\begin{equation}\label{delta-scaling}
\Phi_{\del t}=D^t(\del)\circ \Phi_t\circ D(\del\inv)\ ,\qquad \tilde\Phi_{\del t}=D^t(\del)\circ \tilde\Phi_t\circ D(\del\inv)\ .
\end{equation}
\end{enumerate}

Let  $T$ be a finite set of points on the unit sphere in $\fz_0$ such that $\{KU_{0,t}\}_{t\in T}$ covers the unit sphere. Then there is $r>1$  such that the annulus $\{z_0\in\fz_0:1\le|z_0|\le r\}$ is contained in $ \bigcup_{t\in T}KU_{0,t}$. Therefore $\{r^jKU_{0,t}\}_{t\in T\,,\,j\in\bZ}$ is a locally finite covering of $\fz_0\setminus\{0\}$.

For each $t\in T$ we choose $\chi_t\ge0$ in $C^\infty_c(U_{0,t})$ so that $\sum_t\chi_t>0$ on $\{z\in\fz_0:1\le|z|\le r\}$, and set 
$$
\chi^\#_{t,j}(z)=\chi_t^{\rm rad}(r^{-j}z)\ .
$$ 

Up to dividing each $\chi_{t,j}^\#$ by $\sum_{t\in T\,,\,j\in\bZ}\chi_{t,j}^\#$, we may assume that the
$\chi_{t,j}^\#$ form a partition of unity on $\fz_0\setminus\{0\}$  subordinated to the covering $\{r^jKU_{0,t}\}_{t\in T,j\in\bZ}$.

\begin{lemma}\label{xi-partition}
There exists a family  $\{\eta_{t,j}\}_{t\in T\,,\,j\in\bZ}$ of nonnegative functions on $\bR^{d_{\fz_0}}$ and a $D(\del)$-invariant neighbourhood $\Omega$ of ${\boldsymbol\rho}_{\fz_0}(\fz_0)\setminus\{0\}$ in $\bR^{d_{\fz_0}}$ such that
\begin{equation}\label{sum-eta}
\sum_{t\in T\,,\,j\in\bZ}\eta_{t,j}(\xi)=1 \ ,\qquad (\xi\in\Omega)\ , 
\end{equation}
and, for every $t,j$,
\begin{enumerate}
\item[\rm(i)]  $\eta_{t,j}\in C^\infty_c(\bR^{d_{\fz_0}}\setminus\{0\})$;
\item[\rm(ii)] $\eta_{t,j}(\xi)=\eta_{t,0}\big(D(r^{-j})\xi\big)$;
\item[\rm(iii)] $(\supp\eta_{t,j})\cap {\boldsymbol\rho}_{\fz_0}(\fz_0)\subset {\boldsymbol\rho}_{\fz_0}(r^jU_{0,t})$;
\item[\rm(iv)]  $\eta_{t,j}\big({\boldsymbol\rho}_{\fz_0}(z)\big)=\chi^\#_{t,j}(z)$.
\end{enumerate}
\end{lemma}

\begin{proof}
There exist smooth functions $u_t$, $t\in T$, on $\bR^{d_{\fz_0}}$ such that 
$$
u_t\big(\boldsymbol\rho_{\fz_0}(z)\big)=\chi_t^{\rm rad}(z)=\chi^\#_{t,0}(z)\ .
$$

Setting $u_{t,j}=u_t\circ D(r^{-j})$, we have
$\chi_{t,j}^\#(z)=u_{t,j}\big({\boldsymbol\rho}_{\fz_0}(z)\big)$, 
for every $t$ and $j$.

Since ${\boldsymbol\rho}_{\fz_0}(\supp\chi^\#_{t,0})$ does not contain the origin, we may assume that each $u_t$, $t\in T$, is supported on a fixed compact set $E$ of $\bR^{d_{\fz_0}}$ not containing the origin. Moreover, since ${\boldsymbol\rho}_{\fz_0}$ is a proper map, and ${\boldsymbol\rho}_{\fz_0}(\fz_0\setminus KU_{0,t})$ is closed in $\bR^{d_{\fz_0}}$ and disjoint from ${\boldsymbol\rho}_{\fz_0}(\supp\chi^\#_{t,0})$, we may also assume  that $(\supp u_t)\cap{\boldsymbol\rho}_{\fz_0}(\fz_0)\subset {\boldsymbol\rho}_{\fz_0}(KU_{0,t})={\boldsymbol\rho}_{\fz_0}(U_{0,t})$.

Clearly,
$\sum_{t,j}u_{t,j}=1$ on ${\boldsymbol\rho}_{\fz_0}(\fz_0)\setminus\{0\}$. Therefore, if we set
$$
\eta_{t,j}=\frac{u_{t,j}}{\sum_{t',j'}u_{t',j'}}\ ,
$$
$\eta_{t,j}=u_{t,j}$ on ${\boldsymbol\rho}_{\fz_0}(\fz_0)$, and the sum of the $\eta_{t,j}$ remains equal to 1 where some $\eta_{t,j}$ is positive.  Then \eqref{sum-eta} and properties (i)-(iv) follow easily.
\end{proof}

\bigskip

\subsection{Characterisations of functions in $\cS_0(N)^K$}\label{S_0-characterization}\quad
\medskip

If  $g$ is an  integrable function  on $\fz_0$ and $F$ and integrable function on $N$, we set
\begin{equation}\label{*z_0}
F*_{\fz_0}g(v,z,u)=\int_{\fz_0} F(v,z-z',u)g(z')\,dz'\ .
\end{equation}

This  can be regarded as the convolution  on $N$ of $F$ and the finite measure $\del_0\otimes g$, where $\del_0$ is the Dirac delta at the origin in $\fv\oplus\check\fz$.
We use the symbol $\widehat{\phantom a}$ to denote Fourier transform in the $\fz_0$-variables.  For a  function on $N$ we then set
\begin{equation}\label{fourier}
\widehat F(v,\zeta,u)=\int_{\fz_0} F(v,z,u)e^{-i\lan z,\zeta\ran}\,dz\ ,
\end{equation}
for $v\in\fv$, $\zeta\in\fz_0$, $u\in\check\fz$. For $F$ and $g$  as in \eqref{*z_0}, 
$\widehat{F*_{\fz_0}g}(v,\zeta,u)=\widehat F(v,\zeta,u)\widehat g(\zeta)$.

Finally, we denote by $\psi_{t,j}$ the inverse Fourier transform of $\chi^\#_{t,j}$.

\begin{lemma}\label{S_0-equivalence}
The following are equivalent for a function $F\in\cS(N)$:
\begin{enumerate}
\item[\rm(i)] $F\in\cS_0(N)$;
\item[\rm(ii)] $\widehat F(v,\zeta,u)$ vanishes with all its derivatives for $\zeta=0$;
\item [\rm(iii)] for every $k\in\bN$, $F(v,z)=\sum_{|\al|=k}\de_z^\al G_\al(v,z)$, with $G_\al\in\cS(N)$ for every $\al$;
\item[\rm(iv)] the series $\sum_{t\in T\,,\,j\in\bZ}F*_{\fz_0}\psi_{t,j}$ converges to $F$ in every Schwartz norm;
\item[\rm(v)] for every Schwartz norm $\|\ \|_{\cS(N),M}$ and every $q\in\bN$, $\|F*_{\fz_0}\psi_{t,j}\|_{\cS(N),M}=o(r^{-q|j|})$ as $j\to\pm\infty$.
\end{enumerate}
\end{lemma}

\begin{proof}
The equivalence of (i) and (ii) is a direct consequence of the definition of $\cS_0(N)$. The equivalence of (ii) and (iii) follows from Hadamard's lemma, cf. \cite{FRY2}. Finally, the equivalence among (ii), (iv) and (v) can be easily seen on the $\fz_0$-Fourier transform side.
\end{proof}

\bigskip

\subsection{Radon transforms of $K$-invariant functions}\label{section-Radon}\quad
\medskip

Given $t\in \fz_0$ and $F\in\cS(N)^K$, let $\cR^t F\in\cS(N_t)^{K_t}$ be its Radon transform \eqref{radon*} defined on $N_t$.
By \eqref{Lambda^gamma}, for $\ph^t\in\Sigma^t$, we have, cf. \cite{FR},
$$
\int_{N_t} \cR^t F(v,z') \ph^t(v,z')\,dv\,dz'=\int_N F(v,z)\Lambda^t\ph^t (v,z)\,dv\,dz\ ,
$$
i.e.,
\begin{equation}\label{G-G_t}
\cG_t(\cR^t F)=(\cG F)\circ\Lambda^t\ .
\end{equation}

We are now able to prove Proposition \ref{S_0}.

\begin{proof}[Proof of Proposition \ref{S_0}]
Decompose $F$ according to Lemma \ref{S_0-equivalence} (iv). Then
$$
\cG F=\sum_{t\in T\,,\,j\in\bZ}\cG(F*_{\fz_0}\psi_{t,j})\ .
$$

Denoting by $\mu_{t,j}$ the measure $\del_0\otimes\psi_{t,j}$, where $\del_0$ is the Dirac delta at the origin in $\fv\oplus\check\fz$, $\cG(F*_{\fz_0}\psi_{t,j})$ is the product of $\cG F$ and $\cG\mu_{t,j}$. By \eqref{trace}, if $\zeta=\zeta_0+\check\zeta\in\fz_0\oplus\check\fz$, then 
$$
\begin{aligned}
\cG\mu_{t,j}(\ph_{\zeta,\omega,\mu})&=\int_{\fz_0}\psi_{t,j}(z)\ph_{\zeta\omega,\mu}(0,-z,0)\,dz\\
&=\int_{\fz_0}\int_K\psi_{t,j}(z)e^{-i\lan\zeta_0,kz\ran}\,dk\,dz\\
&=\chi_{t,j}^\#(\zeta_0)\ .
\end{aligned}
$$

Then, for $\xi=(\xi_{\fz_0},\xi_{\check\fz},\xi_\fv,\xi_{\fv,\fz_0})\in\Sigma_\cD$,
\begin{equation}\label{G(f*psi)}
\cG(F*_{\fz_0}\psi_{t,j})(\xi)=\cG F(\xi)\eta_{t,j}(\xi_{\fz_0})\ .
\end{equation}

Set $F_{t,j}=F*_{\fz_0}\psi_{t,j}$, and consider its Radon transform $\cR^tF_{t,j}$. Since we are assuming that Property (S) holds for $(N_t,K_t)$, we have $\cG_t(\cR^tF_{t,j})\in \cS(\Sigma^t_{\cD_t})$. Moreover, for every $M,q\in\bN$, 
$$
\|\cG_t(\cR^tF_{t,j})\|_{M,\cS(\Sigma^t_{\cD_t})}=o(r^{-q|j|})\ .
$$

For fixed $M$, there exist functions $h^{(M)}_{t,j}\in\cS(\bR^d)$ such that 
$$
{h^{(M)}_{t,j}}_{\Sigma^t_{\cD_t}}=\cG_t(\cR^tF_{t,j})\ ,\qquad \|h^{(M)}_{t,j}\|_{M,\cS(\bR^d)}\le 2\|\cG_t(\cR^tF_{t,j})\|_{M,\cS(\Sigma^t_{\cD_t})}\ .
$$

Let $\Phi_{r^jt}$ be the map in \eqref{delta-scaling}.
Since $\cG F_{t,j}$ is supported in $\Sigma_\cD\cap \Pi\inv\big({\boldsymbol\rho}_{\fz_0}(r^jU_{0,t})\big)$,  \eqref{G-G_t} and Proposition \ref{tildeQ} imply that the composition $g^{(M)}_{t,j}=h^{(M)}_{t,j}\circ\tilde\Phi_{r^jt}$ coincides with $\cG F_{t,j}$ on $\Sigma_{\cD}$.

Therefore, for every choice of the integers $M_j$, we can say that the series
$$
\sum_{t,t'\in T\,,\,j,j'\in\bZ}\eta_{t',j'}(\xi_{\fz_0})g^{(M_j)}_{t,j}(\xi)
$$
converges pointwise to $\cG F$ on $\Sigma_\cD$. In fact, many of the terms will vanish identically on $\Sigma_\cD$, and this surely occurs when ${\boldsymbol\rho}_{\fz_0}(r^{j'}U_{0,t'})\cap {\boldsymbol\rho}_{\fz_0}(r^jU_{0,t})=\emptyset$.  Therefore,  if $E_{t,j}=\{(t',j'):{\boldsymbol\rho}_{\fz_0}(r^{j'}U_{0,t'})\cap {\boldsymbol\rho}_{\fz_0}(r^jU_{0,t})\ne\emptyset\}$,
\begin{equation}\label{sum}
\sum_{t\in T\,,\,j\in\bZ}\sum_{(t',j')\in E_{t,j}}\eta_{t',j'}(\xi_{\fz_0})g^{(M_j)}_{t,j}(\xi)=\cG F(\xi)
\end{equation}
on $\Sigma_\cD$.
Notice that, by construction,  there is $A>0$ such that $|j-j'|\le A$ for $(t',j')\in E_{t,j}$. In particular, the sets $E_{t,j}$ are finite and their cardinalities have a uniform upper bound. 

We claim that, choosing the $M_j$ appropriately, we can make the series \eqref{sum} converge to the required Schwartz extension of $\cG F$.

 In order to estimate the Schwartz norms of $\eta_{t',j'}g^{(M)}_{t,j}$, for $(t',j')\in E_{t,j}$, observe that, by Proposition \ref{Qinv}, the $\fz_0$-variables are bounded to a compact set, and the other variables appear in $\Phi_t$ as linear or quadratic factors. Taking this into account, together with the scaling properties of $\Phi_{r^jt}$ and $\eta_{t',j'}$, cf.  \eqref{delta-scaling} and Lemma \ref{xi-partition} (ii), it is not hard to see that there is $p_M$ depending only on $M$ such that, for $(t',j')\in E_{t,j}$,
$$
\|\eta_{t',j'}g^{(p_M)}_{t,j}\|_{M,\cS(\bR^d)}\le Cr^{|j|p_M}\|h^{(p_M)}_{t,j}\|_{p_M,\cS(\bR^d)}\ .
$$

Hence,
$$
\|\eta_{t',j'}g^{(p_M)}_{t,j}\|_{M,\cS(\bR^d)}=o(r^{-q|j|})\ ,
$$
for every $q$ and $(t',j')\in E_{t,j}$.

By induction, we can then select a strictly increasing sequence $\{j_q\}_{q\in\bN}$ of integers, such that $j_0=0$ and
$$
\|\eta_{t',j'}g^{(p_q)}_{t,j}\|_{q,\cS(\bR^d)}\le r^{-q|j|}\ ,
$$
for $j\ge j_q$ and $(t',j')\in E_{t,j}$. For $j_q\le j< j_{q+1}$, we select $M_j=p_q$ and consider the following special case of \eqref{sum}:
$$
g(\xi)=\sum_{q=0}^\infty\,\sum_{\substack{j_q\le j<j_{q+1}\\ t\in T}}\,\sum_{(t',j')\in E_{t,j}}\eta_{t',j'}(\xi_{\fz_0})g^{(p_q)}_{t,j}(\xi)\ .
$$

Then the series converges in every Schwartz norm. Since each term vanishes identically on a neighbourhood of $\{0\}\times\bR^{d_{\check\fz}}\times\bR^{d_\fv}\times\bR^{d_{\fv,\fz_0}}$, the sum must have all derivatives vanishing on this set.
\end{proof}

\vskip1cm
\section{The pair $(\check N,K)$ and Taylor developments}
\label{section_checkN}
\bigskip

In this section, we complete the proof of Theorem \ref{main}, under the assumption that Property~(S) holds for all proper quotient pairs of $(N,K)$. In order to do so, we want to remove the restriction on $F$ in Proposition \ref{S_0} and assume $F\in\cS(N)^K$.
This means analysing $\cG F$  in a neighbourhood of $\check\Sigma_\cD=\Sigma_\cD\cap (\{0\}\times\bR^{d_{\check\fz}}\times\bR^{d_\fv}\times\bR^{d_{\fv,\fz_0}})$.

\subsection{The group $\check N$}\label{}\quad
\medskip

By Lemma \ref{Pi}, the spherical functions $\ph_{\zeta,\omega,\mu}$ associated to points in this set are the ones for which $\zeta\in\check\fz$. This means that the representation $\pi_{\zeta,\omega}$ factors to a representation of $\check N=N/\exp\fz_0$ and, correspondingly, $\ph_{\zeta,\omega,\mu}$ is a spherical function of the pair $(\check N,K)$ composed with the canonical projection.

Notice that, as we have implicitly admitted a few lines above, $(\check N,K)$ is a n.G.p., and that Property (S) is known to hold for it (in fact, $\check N$ is either a Heisenberg group, or its quaternionic analogue, with $\check\fn=\bH^n\oplus\IM\bH$).

We denote by $\cD_{\check\fz}$, $\cD_{\fz_0}$, $\cD_\fv$, $\cD_{\fv,\fz_0}$  the families of operators corresponding to the corresponding subfamilies of polynomials in ${\boldsymbol\rho}$.

The following properties are easily verified:
\begin{enumerate}
\item[(i)] for $D\in\cD_{\fz_0,\fv}$, we have $D\ph_{\zeta,\omega,\mu}=0$, i.e., whenever $\xi\in\Sigma_\cD$ has $\xi_{\fz_0}=0$, then also $\xi_{\fv,\fz_0}=0$, cf. Lemma \ref{dominant} (i);
\item[(ii)] if $\check \cR$ denotes the Radon transform \eqref{opradon*} mapping differential operators on $N$ into differential operators on $\check N$, then $\check \cR D=0$ for $D\in\cD_{\fz_0}\cup \cD_{\fv,\fz_0}$, and $\check \cD=\{\check\cR D: D\in\cD_{\check\fz}\cup\cD_\fv\}$, is a free homogeneous basis of $\bD(\check N)^K$;
\item[(iii)] if $\Sigma_{\check\cD}$ is the Gelfand spectrum of $(\check N,K)$, then
\begin{equation}\label{check-sigma}
\check\Sigma_\cD=\big\{(0,\xi_{\check\fz},\xi_\fv,0):(\xi_{\check\fz},\xi_\fv)\in\Sigma_{\check\cD}\big\}\ ;
\end{equation}
\item[(iv)]  denoting by  $\check\cG$ the Gelfand transform of $(\check N,K)$ and by $\check\cR$ the Radon transform \eqref{radon*} with the integral taken over $\fz_0$, we have the identity
\begin{equation}\label{gelfand-radon}
\check\cG(\check\cR F)(\xi_{\check\fz},\xi_\fv)=\cG F(0,\xi_{\check\fz},\xi_\fv,0)\ .
\end{equation}
\end{enumerate}

\bigskip

\subsection{Taylor developments on $\check\Sigma_\cD$}\label{}\quad
\medskip

For a multi-index $\al=(\al',\al'')\in\bN^{d_{\fz_0}}\times\bN^{d_{\fv,\fz_0}}$, we denote by  $[\al]$ the degree of   ${\boldsymbol\rho}_{\fz_0}^{\al'}{\boldsymbol\rho}_{\fv,\fz_0}^{\al''}$ in the $\fz_0$-variables.

For $\rho_j\in{\boldsymbol\rho}_{\fz_0}\cup{\boldsymbol\rho}_{\fv,\fz_0}$, let $\tilde D_j\in\bD(\check N)\otimes\cP(\fz_0)$ denote $(\la'_{\check N}\otimes I)(\rho_j)$, where $\la'_{\check N}$ is the symmetrisation operator \eqref{modsym}, and $\rho_j$ is regarded as an element of $\cP(\check\fn)\otimes\cP(\fz_0)$. 
If $\al=(\al',\al'')\in\bN^{d_{\fz_0}}\times\bN^{d_{\fv,\fz_0}}$, then
$$
\tilde D^\al={\boldsymbol\rho}_{\fz_0}^{\al'}\tilde D^{\al''}
$$
has degree $[\al]$ in the $\fz_0$-variables.

In \cite{FRY2} the following result was proved\footnote{In fact, \cite{FRY2} also gives the control of the Schwartz norms of $H_\al$ in terms of Schwartz norms of the $G_\gamma$, but this is not strictly needed here.}.

\begin{proposition}\label{prop_FRY2}
 Let $G$ be a $K$-invariant function on $\check N\times\fz_0$ of the form
\begin{equation}\label{G}
 G(v,z,u)=\sum_{|\gamma|=k}z^\gamma G_\gamma(v,u)\ ,
\end{equation}
 with $G_\gamma\in\cS(\check N)$. 
 Then there exist functions $H_{\al}\in\cS(\check N)^K$, for $[\al]=k$, such that
\begin{equation}\label{Hal}
 G=\sum_{[\al]= k}\frac1{\al!}\tilde D^{\al} H_{\al}\ .
\end{equation}
\end{proposition}

We complete, giving all details, the argument sketched in \cite{FRY2}, which derives the following Hadamard-type formula from Proposition \ref{prop_FRY2}.

\begin{proposition}\label{hadamard}
Let $F\in\cS(N)^K$, and assume that 
$$
F(v,z,u)=\sum_{|\gamma|=k}\de_z^\gamma R_\gamma(v,z,u)\ ,
$$
with $R_\gamma\in\cS(N)$ for every $\gamma$. Then, for every $\al$ with $[\al]=k$, there exists a function $F_\al\in\cS(N)^K$ such that
$$
\cG F_\al(\xi)=h_\al(\xi_{\check\fz},\xi_\fv)\ ,\qquad (\xi\in\Sigma_\cD)\ ,
$$
with $h_\al\in\cS(\bR^{d_{\check \fz}+d_{\fv}})$, and
$$
F(v,z,u)=\sum_{[\al]=k}\frac1{\al!}D^\al F_\al(v,z,u)+\sum_{|\gamma'|=k+1}\de_z^\beta R'_{\gamma'}(v,z,u)\ ,
$$
with $R'_{\gamma'}\in\cS(N)$ for every $\gamma'$.
\end{proposition}

\begin{proof}
Consider the Fourier transform \eqref{fourier} of $F$ in the $\fz_0$-variables. Then
$$
\widehat F(v,\zeta,u)=i^k\sum_{|\gamma|=k}\zeta^\gamma\widehat R_\gamma(v,\zeta,u)\ .
$$

Setting $G_\gamma(v,u)=i^k\widehat R_\gamma(v,0,u)\in\cS(\check N)$, we have
$$
\widehat F(v,\zeta,u)=\sum_{|\gamma|=k}\zeta^\gamma G_\gamma(v,u)+\sum_{|\gamma'|=k+1}\zeta^{\gamma'} S_{\gamma'}(v,\zeta,u)\ ,
$$

Since $G=\sum_\gamma \zeta^\gamma G_\gamma$ is $K$-invariant, we can apply Proposition \ref{prop_FRY2} to obtain that there exist $H_\al\in\cS(\check N)^K$ such that
$$
\widehat F(v,\zeta,u)=\sum_{[\al]= k}\frac1{\al!}(\tilde D^{\al} H_{\al})(v,\zeta,u)+\sum_{|\gamma'|=k+1}\zeta^{\gamma'} S_{\gamma'}(v,\zeta,u)\ ,
$$

By Property (S)  for the pair $(\check N,K)$, for every $\al$ there is $h_\al\in\cS(\bR^{d_{\check\fz}+d_\fv})$ such that $\check\cG H_\al={h_\al}_{|_{\Sigma_{\check\cD}}}$. By Lemma \ref{dominant} (ii),  $|\xi_\fz|$ and $|\xi_{\fv,\fz_0}|$ are controlled by powers of $|\xi_\fv|$ on $\Sigma_\cD$. We then have that
$$
\tilde h_\al(\xi)=h_\al(\xi_{\check\fz},\xi_{\fv,\fz_0})\in\cS(\Sigma_\cD)\ .
$$

If $F_\al\in\cS(N)^K$ is the function such that $\cG H_\al$ equals $\tilde h_\al$ on $\Sigma_\cD$, it follows from \eqref{gelfand-radon} that $H_\al=\check\cR F_\al$ for every $\al$. This is equivalent to saying that
$\widehat F_\al(v,0,u)=H_\al(v,u)$. By Hadamard's lemma,
$$
\widehat F_\al(v,\zeta,u)-H_\al(v,u)=\sum_{j=1}^{d_{\fz_0}}\zeta_j K_j(v,\zeta,u)\ ,
$$
with $K_j$ smooth. Therefore,
$$
\begin{aligned}
\widehat F(v,\zeta,u)&=\sum_{[\al]= k}\frac1{\al!}(\tilde D^{\al} \widehat F_{\al})(v,\zeta,u)-\sum_{[\al]= k}\sum_{j=1}^{d_{\fz_0}}\frac1{\al!}\zeta_j\tilde D^{\al} K_j(v,\zeta,u) +\sum_{|\gamma'|=k+1}\zeta^{\gamma'} S_{\gamma'}(v,\zeta,u)\\
&=\sum_{[\al]= k}\frac1{\al!}(\tilde D^{\al} \widehat F_{\al})(v,\zeta,u)+\sum_{|\gamma'|=k+1}\zeta^{\gamma'} S'_{\gamma'}(v,\zeta,u)\ .
\end{aligned}
$$

Notice that $S'=\sum_{|\gamma'|=k+1}\zeta^{\gamma'} S'_{\gamma'}\in\cS(\check N\times\fz_0)$ because it is the difference of two Schwartz functions. Since the derivatives of order up to $k$ vanish for $\zeta=0$, it follows from Hadamard's lemma that the same sum $S'$ can be obtained by   replacing each  $S'_{\gamma'}$ by a function $S''_{\gamma'}\in\cS(\check N\times\fz_0)$.

Now we can undo the Fourier transform. In doing so, the monomials in $\zeta$ are turned into derivatives in the $\fz_0$-variables, and each $\tilde D_j$ is turned into the differential operator $\lambda'_{\tilde N}(\rho_j)$, where the symmetrisation is taken on the direct product $\bar N=\check N\times\fz_0$. We denote by $\bar D_j$ this operator.

We then have
$$
F=\sum_{[\al]= k}\frac1{\al!}\bar D^{\al} F_{\al}+\sum_{|\gamma'|=k+1}\de_z^{\gamma'} U_{\gamma'}\ ,
$$
where $F_\al$ satisfies the stated requirements and $U_{\gamma'}\in\cS(\bar N)$ for every $\gamma'$.

It only remains to replace each operator $\bar D^\al$ with the corresponding $D^\al$. In order to do this, it is sufficient to compare the left-invariant vector field $X_{v_0}$ on $N$ corresponding to an element $v_0\in\fv$ with the  left-invariant vector field $\bar X_{v_0}$ on $\bar N$ corresponding to the same $v_0$. Clearly, the difference $X_{v_0}-\bar X_{v_0}$ is a linear combination $\sum_{j=1}^{d_{\fz_0}}\ell_{j,v_0}(v)\de_{z_j}$ where the $\ell_{j,v_0}$ are linear functionals on $\fv$. Therefore, any difference $D^\al-\bar D^\al$ is a sum of terms, each of which contains at least $k+1$ derivatives in the $\fz_0$-variables. The corresponding term $(D^\al-\bar D^\al)F_\al$ is absorbed by the remainder term.
\end{proof}

\begin{proposition}\label{F-G}
Given $F\in\cS(N)^K$, there is $g\in\cS(\bR^d)$ such that, calling $G\in\cS(N)^K$ the function such that $\cG G=g_{|_{\Sigma_\cD}}$, then $F-G\in\cS_0(N)$.
\end{proposition}

\begin{proof}
Consider the restriction of $\cG F$ to the set $\check\Sigma_\cD$ in \eqref{check-sigma}, which equals $\check\cG(\check\cR F)$ by \eqref{gelfand-radon}. If $h_0\in\cS(\bR^{d_{\check\fz}+d_\fv})$  extends  $\check\cG(\check\cR F)$, let $F_0\in\cS(N)^K$ be the function such that $\cG F_0(\xi)=h_0(\xi_{\check\fz},\xi_\fv)$.

Then $\check\cR(F-F_0)=0$, which implies that $F-F_0=\sum_{j=1}^{d_{\fz_0}}\de_{z_j}G_j$, with $G_j\in\cS(N)$. By iterated application of Proposition \ref{hadamard}, we find a family $\{F_\al\}_{\al\in\bN^{d_{\fz_0}}\times\bN^{d_{\fv,\fz_0}}}$ such that, for every $k$,
\begin{equation}\label{expansion}
F=\sum_{[\al]\le k}\frac1{\al!}D^\al F_\al+\sum_{|\gamma|=k+1}\de_z^\gamma R_\gamma\ ,
\end{equation}
and, for every $\al$, $\cG F_\al(\xi)=h_\al(\xi_{\check\fz},\xi_\fv)$, with $h_\al\in\cS(\bR^{d_{\check\fz}+d_\fv})$.

By Whitney's extension theorem \cite{Mal} (see \cite{ADR2} for a proof in the Schwartz setting), there is a function $g\in\cS(\bR^d)$ such that, for every $\al=(\al',\al'')\in\bN^{d_{\fz_0}}\times\bN^{d_{\fv,\fz_0}}$,
$$
\de^{\al'}_{\xi_{\fz_0}}\de^{\al''}_{\xi_{\fv,\fz_0}}(0,\xi_{\check\fz},\xi_\fv,0)=h_\al(\xi_{\check\fz},\xi_\fv)\ .
$$

We take as $G$ the function in $\cS(N)^K$  such that $\cG G=g_{\Sigma_\cD}$. We must then prove that $G-F\in\cS_0(N)^K$.

Take a monomial $z^\beta$   on $\fz_0$. 
Given an integer $k\ge|\beta|$,   decompose $F$ as in \eqref{expansion} and observe that
$$
\int_{\fz_0}\big(G(v,z,u)-F(v,z,u)\big)z^\beta\,dz=\int_{\fz_0}\Big(G-\sum_{[\al]\le k}\frac1{\al!}D^\al F_\al\Big)(v,z,u)z^\beta\,dz\ ,
$$
since the remainder term gives integral 0 by integration by parts.

We set
$$
r_k(\xi)=g(\xi)-\sum_{[\al]\le k}\frac1{\al!}h_\al(\xi_{\check\fz},\xi_\fv)\xi_{\fz_0}^{\al'}\xi_{\fv,\fz_0}^{\al''}
$$
so that
$r_k=\cG\Big(G-\sum_{[\al]\le k}\frac1{\al!}D^\al F_\al\Big)$ on $\Sigma_\cD$.

Then Lemma \ref{xi-partition} gives the pointwise identity on $\Sigma_\cD$
$$
\cG\Big(G-\sum_{[\al]\le k}\frac1{\al!}D^\al F_\al\Big)(\xi)=
\sum_{t\in T\,,\,j\in\bZ}r_k(\xi)\eta_{t,j}(\xi_{\fz_0})\ .
$$

We claim that $k$ can be chosen large enough so that, undoing the Gelfand transform, the series
$$
\sum_{t\in T\,,\,j\in\bZ}\Big(G-\sum_{[\al]\le k}\frac1{\al!}D^\al F_\al\Big)*_{\fz_0}\psi_{t,j}
$$
 converges  in the $(\cS(N),|\beta|)$-norm  (in which case it converges to $G-\sum_{[\al]\le k}\frac1{\al!}D^\al F_\al$).

By the continuity of $\cG\inv$, there are $m\in\bN$ and $C>0$, depending on $|\beta|$, such that, for every $t,j,k$,
$$
\Big\|\Big(G-\sum_{[\al]\le k}\frac1{\al!}D^\al F_\al\Big)*_{\fz_0}\psi_{t,j}\Big\|_{\cS(N),|\beta|}\le C\|r_k\eta_{t,j}\|_{\cS(\bR^d),m}\  .
$$

Hence we look for $k$ such that $\|r_k\eta_{t,j}\|_{\cS(\bR^d),m}=O(r^{-|j|})$ for every $t\in T$. We first do so with $r_k$ replaced by the remainder $s_{k'}$ in Taylor's formula,
$$
s_{k'}(\xi)=g(\xi)-\sum_{|\al|\le k'}\frac1{\al!}h_\al(\xi_{\check\fz},\xi_\fv)\xi_{\fz_0}^{\al'}\xi_{\fv,\fz_0}^{\al''}\ .
$$

Then, for $\gamma\in\bN^d$ with $|\gamma|\le m$ and for every $M\in\bN$,
$$
\big|\de^\gamma (s_{k'}\eta_{t,j})(\xi)\big|\le C_{m,M}(1+2^{-\tilde mj})(1+|\xi_\fv|)^{-M}\big(|\xi_{\fz_0}|+|\xi_{\fv,\fz_0}|\big)^{k'+1-m}\ ,
$$
where $\tilde m$ only depends on $m$.

From Lemma~\ref{dominant} we obtain that there are positive exponents $a,b,c$, $a',b',c'$ such that, on  $\Sigma_\cD$,
\begin{equation}\label{dominant'}
\text{for $|\xi|$ small: }\begin{cases}|\xi_{\fz_0}|\lesssim |\xi_\fv|^a\\ |\xi_{\fv,\fz_0}|\lesssim |\xi_\fv|^b|\xi_{\fz_0}|^c\ ,\end{cases} \qquad \text{for $|\xi|$ large: }\begin{cases}|\xi_{\fz_0}|\lesssim |\xi_\fv|^{a'}\\ |\xi_{\fv,\fz_0}|\lesssim |\xi_\fv|^{b'}|\xi_{\fz_0}|^{c'}\ .\end{cases} 
\end{equation}

Since the inequalities of Lemma \ref{dominant} remain valid in a $D(\del)$-invariant neighbourhood of $\Sigma_\cD$, we may  assume that \eqref{dominant'} hold uniformly on the support of each $s_{k'}\eta_{t,j}$.

Since there are $\tau,\tau'>0$ such that, on the support of $\eta_{t,j}$, $|\xi_{\fz_0}|\le r^{\tau j}$ if $j\le0$, and $|\xi_{\fz_0}|\le r^{\tau' j}$ if $j>0$, we can choose $k'$ such that  $\|s_{k'}\eta_{t,j}\|_{\cS(\bR^d),m}=O(r^{-|j|})$.

Finally, we choose $k=\max\{[\al]:|\al|\le k'\}$. Since $k>k'$,
$$
r_k-s_{k'}=\sum_{|\al|> k'\,,\,[\al]\le k}\frac1{\al!}h_\al(\xi_{\check\fz},\xi_\fv)\xi_{\fz_0}^{\al'}\xi_{\fv,\fz_0}^{\al''}\ .
$$

The estimates on $\|(r_k-s_{k'})\eta_{t,j}\|_{\cS(\bR^d),m}$ are basically the same.

Now, $\Big(G-\sum_{[\al]\le k}\frac1{\al!}D^\al F_\al\Big)*_{\fz_0}\psi_{t,j}\in\cS_0(N)^K$, because its $\fz_0$-Fourier transform vanishes for $\zeta\in\fz_0$ close to the origin.  To conclude, observe that integration against $z^\beta$ is a continuous operation in the $(\cS(N),|\beta|)$-norm.
\end{proof}

\begin{corollary}\label{induction}
Let $(N,K)$ be one of the pairs in Table \ref{vinberg}, and assume that Property (S) holds on all its proper quotient pairs. Then Property (S) holds on $(N,K)$.
\end{corollary}

\begin{proof}
Given $F\in\cS(N)^K$, let $G$ be as in Proposition \ref{F-G}. By Proposition \ref{S_0}, $\cG(F-G)$ admits a Schwartz extension $h$. Then $g+h$ is a Schwartz extension of $\cG F$. This shows that the map $\cG\inv:\cS(\Sigma_\cD)\longrightarrow\cS(N)^K$, which is continuous by Theorem \ref{hulanicki}, is also surjective. By the open mapping theorem for Fr\'echet spaces \cite{T}, it is an isomorphism.
\end{proof}

 An inductive application of Corollary~\ref{induction} gives us Theorem~\ref{main} for pairs in the first two blocks of Table~\ref{vinberg}.
In fact, the classification of quotient pairs given in the appendix (Section~\ref{appendix}) shows that this set of pairs  is essentially self-contained, in the sense of Remark~\ref{QP-sec5}. 

The lowest-rank summands that need to be considered in order to start the induction are the following:
\begin{itemize}
 \item line 1: the trivial pair $(\bR,\{1\})$ and $(H_1,{\rm SO_2})$;
 \item line 2: $(H_1,{\rm U}_1)\cong(H_1,{\rm SO_2})$;
 \item line 3: the quaternionic Heisenberg group with Lie algebra $\bH\oplus\IM\bH$, and with ${\rm Sp}_1$ acting nontrivially only on $\fv=\bH$;
 \item lines 4, 5,  6: $(\bC,\{1\})$, $(H_1,{\rm U}_1)$, and $(\bR,\{1\})$, respectively.
 \end{itemize}

Since in all these cases Property (S) is either trivial or proved in \cite{ADR1}, 
we can state the following restricted version of Theorem~\ref{main}.

\begin{corollary}\label{blocks12}
Property (S) holds for the pairs in the first two blocks of Table \ref{vinberg}.
\end{corollary}

\vskip1cm
\section{The third block of Table \ref{vinberg}}\label{third-block}
\bigskip

The pairs in the third block of Table \ref{vinberg} must be treated separately because their quotient pairs (cf. Table \ref{quotient-invariants}) do not satisfy Vinberg's condition. In order to apply Corollary \ref{induction}, we need to prove Property (S) for each of the quotient pairs listed in Sections \ref{lines7,8}, \ref{line9},~\ref{line10}.  We call them ``first-generation'' quotient pairs.

Let $(N,K)$ be one of these  quotient pairs. The group $N$ is a direct product, $N_1\times N_2$, with $\fn_j=\fv_j\oplus\fz_j$, $j=1,2$, and $\fv_1=\fv_2$.
The group $K$ is the direct product of three simple factors, $K=K_1\times K_{1,2}\times K_2$, where the subscript indicates on which of the two factors of $N$ the given factor of $K$ acts nontrivially. In particular, $K_{1,2}$ acts by its defining representation on both $\fv_1$ and $\fv_2$, and trivially on $\fz_1$ and $\fz_2$. Moreover, when  $K_j={\rm Sp}_1$, it acts on $\fz_j=\IM\bH\cong \fs\fp_1$ by the adjoint representation.

To prove Property (S), we  adapt the paradigm used so far, and this involves the following steps:
\begin{enumerate}
\item[(i)] identify a critical subset $\check\Sigma_\cD$ of the spectrum such that its complement can be locally identified with spectra of certain  ``second-generation'' quotient pairs;
\item[(ii)] prove that Property (S) holds for each second-generation quotient pair\footnote{This will require a repetition of the paradigm involving some ``third-and-last-generation'' quotient pairs, for which Property (S) is known to hold.};
\item[(iii)] deduce the analogue of Proposition \ref{S_0} for the first-generation quotient pairs;
\item[(iv)] find a  group $\check N$ such that the spectrum of $(\check N,K)$ can be identified with $\check\Sigma_\cD$ and prove a corresponding form of Proposition \ref{hadamard}.
\end{enumerate}

{\scriptsize
\begin{table}[htdp]
\begin{center}
\begin{tabular}{|c|l|l||l|l|l|}
\hline
&$K_1\times K_{1,2}\times K_2$&$\fn=\fn_1\oplus\fn_2$&${\boldsymbol\rho}_\fz$&${\boldsymbol\rho}_\fv={\boldsymbol\rho}_{\fv_1}\cup{\boldsymbol\rho}_{\fv_2}\cup{\boldsymbol\rho}_{\fv_1.\fv_2}$&${\boldsymbol\rho}_{\fv,\fz}$\\
\hline
$a$&${\rm U}_1\times{\rm SU}_n\times{\rm U}_1$&$\fh_n\oplus\fh_n$&$z_1\,,\,z_2$&$\begin{array}{l}|v_1|^2\,,\,|v_2|^2\\|v_1v_2^*|^2\text{ (if }n\ge2)\end{array}$&\\
\hline
$b$&$\begin{array}{l}{\rm U}_1\times{\rm Sp}_n\times{\rm U}_1\\(n\ge2)\end{array}$&$\fh_{2n}\oplus\fh_{2n}$&$z_1\,,\,z_2$&$\begin{array}{l}|v_1|^2\,,\,|v_2|^2\\|v_1v_2^*|^2\\|v_1J\trans v_2|^2\end{array}$&\\
\hline
$c$&${\rm U}_1\times{\rm Sp}_n\times{\rm Sp}_1$&$(\bH^n\oplus\bR)\oplus(\bH^n\oplus\IM\bH)$&$z_1\,,\,|z_2|^2$&$\begin{array}{l}|v_1|^2\,,\,|v_2|^2\\|v_1v_2^*|^2\text{ (if }n\ge2)\end{array}$&$\RE\big(i(v_1v_2^*)z_2(v_2v_1^*)\big)$\\
\hline
\hline
\hline
&$K_1\times K_{1,2}\times K_2$&$\fn_1\oplus\fv_2$&${\boldsymbol\rho}_{\fz_1}$&${\boldsymbol\rho}_\fv$&${\boldsymbol\rho}_{\fv,\fz}$\\
\hline
$a'$&${\rm U}_1\times{\rm SU}_n\times{\rm U}_1$&
$\fh_n\oplus\bC^n$&$z_1$&same as line $a$&\\
\hline
$b'$&$\begin{array}{l}{\rm U}_1\times{\rm Sp}_n\times{\rm U}_1\\(n\ge2)\end{array}$
&
$\fh_{2n}\oplus\bC^{2n}$
&$z_1$&same as line $b$&\\
\hline
$c'$&${\rm U}_1\times{\rm Sp}_n\times{\rm Sp}_1$&
$(\bH^n\oplus\bR)\oplus \bH^n$&$z_1$&same as line $c$&\\
\hline
\end{tabular}
\end{center}
\bigskip
\caption{First- and second-generation quotient pairs for lines 7--10 of Table \ref{vinberg}}
\label{quotient-invariants}
\end{table}
}

The first-generation quotient pairs are shown at lines (a), (b), (c) of Table \ref{quotient-invariants} with their fundamental invariants, and  the second-generation  pairs are shown at lines (a'), (b'), (c'). When $v_1,v_2$ have several components, they must be understood as row-vectors. At lines $a, b, a', b'$, the vectors $v_1,v_2$ are complex, while they are quaternionic  at lines $c$ and $c'$. At lines $c, c'$, the factor ${\rm U}_1$ in $K$ acts by (left) scalar multiplication by $e^{i\theta}$ on $\fv_1=\bH^n$.

For a better organisation of this material, we start with the second-generation quotient pairs, prove Property (S) for them, and then pass to the first generation.

\bigskip

\subsection{Second-generation quotient pairs}\label{second}\quad
\medskip

The pairs involved at this stage, listed  in the second block of Table \ref{quotient-invariants}, will be the ones obtained from the first-generation pairs by factoring out of $\fn$ the centre of the second summand. The group $K$ remains unchanged. 

In this subsection,
we  prove the following statement.
\begin{proposition}
\label{prop_quotient}
The nilpotent Gelfand pairs $(N_1\times\fv_2,K)$ at lines $a', b', c'$ of Table \ref{quotient-invariants} satisfy Property (S).
\end{proposition}

As it was mentioned before, the proof involves the introduction of third-generation quotient pairs, which are constructed in analogy with Section \ref{quotient}. Here we let  $\fz_1\oplus\fv_2$, the centre of $\fn_1\oplus\fv_2$, play the r\^ole that was of $\check\fz\oplus\fz_0$ in Section \ref{quotient}. 

Given $t\ne0$ in $\fv_2$, we factor out the tangent space $t^\perp$ from $\fv_2$, i.e. the tangent space to the $K$-orbit in $\fv_2$. The resulting quotient Lie algebra, $\fn_t=\fn_1\oplus\bR$, is either  $\fh_n\oplus\bR$ (line $a'$), or $\fh_{2n}\oplus\bR$ (lines $b',c'$), with 
\begin{equation}\label{K_t}
K_t=\begin{cases}
{\rm U}_1\times{\rm U}_{n-1}&{\rm line}\ a'\\ {\rm U}_1^2\times{\rm Sp}_{n-1}&{\rm line}\ b'\\{\rm U}_1\times({\rm Sp}_1\times{\rm Sp}_{n-1})&{\rm line}\ c'\ ,
\end{cases}
\end{equation}
where the first two actions are related to decompositions of $\fv_1$ as $\bC^n=\bC\oplus\bC^{n-1}$, $\bC^{2n}=\bC^2\oplus\bC^{2n-2}$, and $\bH^n=\bH\oplus\bH^{n-1}$ respectively. 

\begin{lemma}
The third-generation quotient pairs $(N_t,K_t)$ satisfy Property (S).
\end{lemma}

\begin{proof} Since $N_t$ is a Heisenberg group, the statement follows from  \cite{ADR2} and Proposition \ref{product}. 
\end{proof}

Let $(N_1\times\fv_2,K)$ be one of the second-generation pairs, and $(N_1\times\bR,K_t)$ the corresponding third-generation pair. 

A homogeneous Hilbert basis $\boldsymbol\rho_t$ on $\fn_t$ is given by the norm squared on the irreducible components of $\fv_1$ and by the two coordinate functions on $\fz_1$ and $\bR$. We let $\cD_t$ be the system of their symmetrisations on $N_1\times\bR$.

Let $\cD$ be the systems of symmetrisations on $N_1\times\fv_2$ of the invariants in Table \ref{quotient-invariants}. To points $\xi\in\Sigma_\cD$ we assign coordinates $(\xi_{\fz_1},\xi_{\fv_1},\xi_{\fv_2},\xi_{\fv_1,\fv_2})$.

Let $\Sigma_\cD$, resp. $\Sigma_{\cD_t}$, the two embedded Gelfand spectra. In analogy with Lemma \ref{Pi} and Proposition \ref{epi}, we have surjections
$$
\Sigma_{\cD_t}\overset\Lambda \longrightarrow\Sigma_\cD\overset\Pi\longrightarrow \boldsymbol\rho_{\fv_2}(\fv_2)=[0,+\infty)\ .
$$

We set
$$
\check\Sigma_\cD=\Pi\inv(0)=\{\xi:\xi_{\fv_2}=\xi_{\fv_1,\fv_2}=0\}\subset\Sigma_\cD\ .
$$

This set represents the spherical functions that do not depend on the $\fv_2$-variable. Hence it is naturally identified with the spectrum of $(N_1,K)$.

Away from $\Lambda\inv(\check\Sigma_\cD)$, $\Lambda$ is a homeomorphism, and both $\Lambda$ and $\Lambda\inv$ are restrictions of smooth maps (cf. Propositions \ref{Lambda-smooth} and \ref{tildeQ}). Therefore, setting
$$
\cS_0(N_1\times\fv_2)=\big\{F\in\cS(N_1\times\fv_2):\int_{\fv_2} F(v_1,z_1,v_2)\,dv_2=0\ \ \forall\,(v_1,z_1)\big\}\ ,
$$
 the spherical transform $\cG F$ of any $F\in\cS_0(N_1\times\fv_2)^K$ can be extended to a Schwartz function vanishing with all derivatives where $\xi_{\fv_2}=0$ (cf. Proposition \ref{S_0}).

The next task is then to establish the following analogue of Proposition \ref{hadamard} for the second-generation pairs. The different formulation that we give below is required by the fact that, due to the presence of the abelian factor $\fv_2$, $\fv_1$ does not generate $\fn$. This implies that smooth multipliers of operators in $\la'_N(\fn_1)$ cannot have Schwartz kernels on $N_1\times\fv_2$. 

Let $\boldsymbol\rho'=\boldsymbol\rho_{\fv_2}\cup\boldsymbol\rho_{\fv_1,\fv_2}\subset\boldsymbol\rho$, $\cD'=\la'_{N_1\times\fv_2}(\boldsymbol\rho')$ and $d'$ its cardinality. For $\al\in\bN^{d'}$, by $[\al]$ we denote the order of differentiation in the $\fv_2$-variables of the operator $D^\al$, as a monomial in the elements of $\cD'$.

The analogue of Proposition \ref{hadamard} is as follows.

\begin{proposition}\label{hadamard2}
Let $F\in\cS(N_1\times\fv_2)^K$, and assume that 
$$
F(v_1,z_1,v_2)=\sum_{|\gamma|=k}\de_{v_2}^\gamma R_\gamma(v_1,z_1,v_2)\ ,
$$
with $R_\gamma\in\cS(N_1\times\fv_2)$ for every $\gamma$. Let also $\Psi\in \cS(\fv_2)^K$, with Fourier transform $\widehat\Psi$ equal to 1 on a neighbourhood of the origin.

Then, for every $\al$ with $[\al]=k$, there exists a function $F_\al\in\cS(N_1)^K$ such that
$$
F(v_1,z_1,v_2)=\sum_{[\al]=k}\frac1{\al!}D^\al \big(F_\al(v_1,z_1)\Psi(v_2)\big)+\sum_{|\gamma'|=k+1}\de_{v_2}^{\gamma'} R'_{\gamma'}(v_1,z_1,v_2)\ ,
$$
with $R'_{\gamma'}\in\cS(N_1\times\fv_2)$ for every $\gamma'$.
\end{proposition}

 \begin{proof}[{\it Proof}]
As in the proof of Proposition~\ref{hadamard}, we begin by taking partial Fourier transforms  in the $\fv_2$-variables (denoted by $\widehat F$, $\widehat R_\gamma$ etc.).
We have
$$
\begin{aligned}
\widehat F(v_1,z_1,w_2)&=i^k\sum_{|\gamma|=k}w_2^\gamma \widehat R_\gamma(v_1,z_1,w_2)\\
&=i^k\widehat\Psi(w_2)\sum_{|\gamma|=k}w_2^\gamma \widehat R_\gamma(v_1,z_1,0)\\
&\qquad +i^k\sum_{|\gamma|=k}w_2^\gamma \widehat\Psi(w_2)\big(\widehat R_\gamma(v_1,z_1,w_2)-\widehat R_\gamma(v_1,z_1,0)\big)\\
&\qquad +i^k\sum_{|\gamma|=k}w_2^\gamma \big(1-\widehat\Psi(w_2)\big)\widehat R_\gamma(v_1,z_1,w_2)\ .
\end{aligned}
$$

The functions $\widehat\Psi(w_2)\big(\widehat R_\gamma(v_1,z_1,w_2)-\widehat R_\gamma(v_1,z_1,0)\big)$ and $\big(1-\widehat\Psi(w_2)\big)\widehat R_\gamma(v_1,z_1,w_2)$ are Schwartz and vanish for $w_2=0$. Applying Hadamard's lemma, we have
\begin{equation}\label{hatF}
\widehat F(v_1,z_1,w_2)=i^k\widehat\Psi(w_2)\sum_{|\gamma|=k}w_2^\gamma \widehat R_\gamma(v_1,z_1,0)+i^{k+1}\sum_{|\gamma'|=k+1}w_2^{\gamma'} \widehat {R'}_{\gamma'}(v_1,z_1,w_2)\ ,
\end{equation}
with $R'_{\gamma'}\in\cS(N_1\times\fv_2)$. We set
\begin{equation}\label{rem}
G(v_1,z_1,w_2)=\sum_{|\gamma|=k}w_2^\gamma \widehat R_\gamma(v_1,z_1,0)\ ,
\end{equation}
and 
$$
\widehat D_j=\la'_{N_1}(\rho_j)\in\bD(N_1)\otimes\cP(\fv_2)\ ,
$$
so that
\begin{equation}\label{tildeD}
\widehat{D_jF}=\widehat  D_j\widehat F\ ,
\end{equation}
for $F\in\cS(N_1\times\fv_2)$.

 We prove the analogue of Proposition \ref{prop_FRY2}, with $\tilde D_j$ replaced by $\widehat D_j$. 

By Lemma 4.1 of \cite{FRY2}, the function $G$ in \eqref{rem} can be expanded as a finite sum,
\begin{equation}\label{gm}
G(v_1,z_1,w_2)=\sum_{m,\al}|w_2|^{2m} p^\al(v_1,w_2)g_\al(v_1,z_1)\ ,
\end{equation}
where the $p^\al$ are products of elements of $\boldsymbol\rho_{\fv_1,\fv_2}$ of degree at least $k-2m$ in $w_2$ and the $g_\al$ are in $\cS(H_n)^K$.
Since the elements in $\boldsymbol\rho_{\fv_1,\fv_2}$ have degree 2 in $w_2$, we may assume that the integer $k$ in \eqref{rem} is even\footnote{In fact, if \eqref{rem} holds with $k$ odd, it holds as well with $k+1$ instead of $k$.} and that $2|\al|+2m=k$ in \eqref{gm}.

  We need at this point the following statement, whose proof will take the second part of this subsection.

\begin{lemma}\label{hadamard3}
Given
\begin{equation}\label{G}
G(v_1,z_1,w_2)=\sum_{|\gamma|=2k}w_2^\gamma G_\gamma(v_1,z_1)\in \big(\cS(H_n)\otimes\cP^{2k}(\fv_2)\big)^K\ ,
\end{equation}
 then
\begin{equation}\label{GtoH}
G(v_1,z_1,w_2)=\sum_{|\al|\le k}|w_2|^{2(k-|\al|)}\makebox{$\widehat  D$}^{\al} H_\al(v_1,z_1)\ ,
\end{equation}
where $\makebox{$\widehat  D$}^{\al}=\prod_{D_j\in\la'(\boldsymbol\rho_{\fv_1,\fv_2})}\makebox{$\widehat  D_j$}^{\al_j}$ and $H_\al\in\cS(H_n)^K$ for every $\al$.
\end{lemma}

 Taking Lemma \ref{hadamard3} for granted, we insert \eqref{GtoH} into \eqref{hatF} and undo the Fourier transform to obtain the conclusion of Proposition \ref{hadamard2}.
\end{proof}

Before giving the proof of Lemma \ref{hadamard3}, we recall some notation from \cite{FRY2} and quote two preliminary results.

If $V$ is a real vector space, $\cH^k(V)$ denotes the subspace of $\cP^k(V)$ consisting of harmonic polynomials, i.e., orthogonal to $|v|^2$. If $V$ is also a complex space, we refer to the holomorphic-antiholomorphic bigrading by replacing the single exponent $k$ by a double exponent.

The first of the two statements we alluded to 
can be found, for example, in \cite{knop-mf}.

\begin{lemma}\label{irred-Sp}
Under the action of ${\rm Sp}_n$, $\cH^{m,m}(\bC^{2n})$ decomposes into irreducibles as
\begin{equation}\label{Sp_n}
\cH^{m,m}(\bC^{2n})=\sum_{i=0}^m V_{m,i}\ ,\qquad V_{m,i}\cong R\big(2(m-i)\varpi_1+i\varpi_2\big)\ ,
\end{equation}
where $\varpi_1$ and $\varpi_2$  denote, respectively, the highest weights of the defining representation and of the adjoint representation of ${\rm Sp}_n$.

Under the action of ${\rm Sp}_1\times{\rm Sp}_n$, $\cH^{2m}(\bH^n)$ decomposes into irreducibles as
\begin{equation}\label{Sp_1xSp_n}
\cH^{2m}(\bH^n)=\sum_{i=0}^m W_{m,i}\ ,
\end{equation}
with $W_{m,i}\cong S^{2(m-i)}\otimes V_{m,i}$, $S^j$ denoting the $(j+1)$-dimensional irreducible representation of~${\rm Sp}_1$.
In particular, ${\rm Sp}_1$ acts trivially on $W_{m,m}=V_{m,m}$.
\end{lemma}

The second statement is extracted from Sections 4 and 5 of \cite{FRY2}.

\begin{lemma}\label{heisenberg}
Let $N_1=H_n$, $K\subseteq{\rm U}_n$,  and  $\cP^{s,0}(\fv_1)$  irreducible with respect to $K$ for every~$s$.

Let $p_{V_1,V_2}$ be the polynomial in \eqref{p_VW}, with $V_1$, $V_2$ equivalent, irreducible representation spaces of $K$, with $V_1\subset \cH^{m,m}(\fv_1)$, $V_2\subset\cP^k(\fv_2)$. Set $M_{V_1,V_2}=\la'_{N_1}(p_{V_1,V_2})\in \big(\bD(N_1)\otimes V_2\big)^K$.

Assume that 
\begin{enumerate}
\item[\rm(i)] $\cP^{s,0}(\fv_1)$ is contained, as a representation space of $K$, inside $\cP^{s,0}(\fv_1)\otimes V_2$ if and only if $s\ge m$, and in this case with multiplicity one;
\item[\rm(ii)] $d\pi(M_{V_1,V_2})\ne0$ on $\cP^{s,0}(\fv_1)$ for every $s\ge m$. 
\end{enumerate}
Then, given $g\in\cS(H_n)^K$, there exists $H\in\cS(H_n)^K$ such that $p_{V_1,V_2}g=M_{V_1,V_2}H$.
\end{lemma}

\begin{proof}[ Proof of Lemma \ref{hadamard3}]
In the case $n=1$ (which is admitted at lines $a'$ and $c'$) there is nothing to prove, because $\boldsymbol\rho'=\{|w_2|^2\}$ and $\widehat{\la'(|w_2|^2)}=|w_2|^2$. Hence \eqref{GtoH} and \eqref{rem} coincide. Therefore, we assume that $n\ge2$. Following the procedure of \cite{FRY2}, we want to split the factor $\cP^{2k}(\fv_2)$ in the tensor product into  irreducible components, and select those components $V$ for which $\big(\cS(H_n)\otimes V\big)^K$ is nontrivial.

From \eqref{gm}, we isolate a single summand. Disregarding the terms containing a positive power of $|w_2|^2$, which can be dealt with by induction, we may then assume that
\begin{equation}\label{G-p^al}
G=p^\al(v_1,w_2)g_\al(v_1,z_1)\ ,
\end{equation}  
with  $k=|\al|$. Now, 
$$
p^\al\in\big(\cP^{2k}(\fv_1)\otimes \cP^{2k}(\fv_2)\big)^K=\sum_{i,j\le k}|v_1|^{2(k-i)}|w_2|^{2(k-j)}\big(\cH^{2i}(\fv_1)\otimes\cH^{2j}(\fv_2)\big)^K\ .
$$

For $i\ne j$, any polynomial in $\big(\cP^{2i}(\fv_1)\otimes\cP^{2j}(\fv_2)\big)^K$ must be divisible by a power of $|v_1|^2$ or $|w_2|^2$. Hence $\big(\cH^{2i}(\fv_1)\otimes\cH^{2j}(\fv_2)\big)^K$ is trivial for $i\ne j$.

As in \cite{FRY2}, given $V_1\subset\cP(\fv_1)$, $V_2\subset\cP(\fv_2)$, $K$-invariant, irreducible, and equivalent, we 
set
\begin{equation}\label{p_VW}
p_{V_1,V_2}(v_1,w_2)=\sum_\ell a_\ell(v_1)\overline{b_\ell(w_2)}\in (V_1\otimes V_2)^K\ ,
\end{equation} 
where $\{a_\ell\}$, $\{b_\ell\}$ are equivalent orthonormal bases of $V_1$ and $V_2$ respectively. We say that $p_{V_1,V_2}$ is an {\it irreducible mixed invariant} on $\fv_1\oplus\fv_2$. For every $\al$ with $|\al|=k$, $p^\al$ can be  decomposed as a finite sum
\begin{equation}\label{p-decomposition}
p^\al=\sum_jc_j|v_1|^{2(k-m_j)}|w_2|^{2(k-m_j)}p_{V_{1,j},V_{2,j}}\ ,
\end{equation}
with $V_{1,j}\subset\cH^{2m_j}(\fv_2)$, $V_{2,j}\subset\cH^{2m_j}(\fv_2)$, and the function $G$ in \eqref{G-p^al}

We can then further restrict the study of \eqref{gm} to the case 
\begin{equation}\label{reducedG}
G=p_{V_1,V_2}(v_1,w_2)g(v_1,z_1)\ ,
\end{equation}
with $V_1\subset\cH^{2m}(\fv_1)$, $V_2\subset\cH^{2m}(\fv_2)$, equivalent $K$-invariant, irreducible subspaces.

Obviously, if $V_1$ and $V_2$ are as above, then $K_1$ acts trivially on $V_1$ and $K_2$ acts trivially on $V_2$. In particular, if $K_i={\rm U}_1$, in which case $\fv_i$ has a $K$-invariant complex structure, we have $V_i\subset \cH^{m,m}(\fv_i)$.

 It follows from Lemma \ref{irred-Sp} that the irreducible mixed invariants $p_{V_1,V_2}$ in $\big(\cH^{2m}(\fv_1)\otimes\cH^{2m}(\fv_2)\big)^K$ correspond to the following pairs of subspaces
\begin{equation}\label{mixed-irred}
\begin{array}{lll}
\text{\rm line}\ a':& \fv_1=\fv_2=\bC^n\ ,& V_1=V_2=\cH^{m,m}\ ;\\
\text{\rm line}\ b':& \fv_1=\fv_2=\bC^{2n}\ ,& V_1=V_2=V_{m,i}\ ,\ i=0,\dots,m\ ;\\
\text{\rm line}\ c':& \fv_1=\fv_2=\bH^n\ ,& V_1=V_2=V_{m,m}\ .\\
\end{array}
\end{equation}

Suppose therefore that $G\in \big(\cS(H_n)\otimes V\big)^K$ is the function in \eqref{reducedG}, with $V=\cH^{m,m}(\fv_2)$, $V_{m,i}(\fv_2)$, or $V_{m,m}(\fv_2)$, depending on the case. 

Setting $M_V=\la'_{N_1}(p_{V(\fv_1),V(\fv_2)})$
we prove that $G=M_VH$ for some $H\in\cS(N_1)^K$. This will give the conclusion, for the following reason. The operator $D_V=\la'_{N_1\times\fv_2}(p_{V(\fv_1),V(\fv_2)})\in\bD(N_1\times\fv_2)^K$ is a polynomial in the elements of $\cD=\cD'\cup\{L,i\inv T\}$, where $L$ and $T$ are, respectively, the sublaplacian and the central derivative on $N_1$. Hence $M_V=\widehat D_V$ is a polynomial in $|w_2|^2$, $L$, $T$ and the $\widehat D_j$ in \eqref{tildeD}. In each monomial, the powers of $L$ and $T$ can be incorporated in the function $H$, leading to a sum of the form \eqref{GtoH}.

Let then $\pi$ be a Bargmann representation of $N_1=H_n$, acting on the Fock space $\cF(\fv_1)$. By $\sigma$ we denote the natural representation of $K$ (effectively of $K_1\times K_{1,2}$) on $\cF(\fv_1)$, such that $\sigma(k)$ intertwines $\pi$ with $\pi\circ k$ for $k\in K$. Since $K_{1,2}$ is either the unitary or the symplectic group on $\fv_1$, the subspaces $\cP^{s,0}(\fv_1)$, $s\in\bN$, are the irreducible components of $\sigma$.

We verify the hypotheses of Lemma \ref{heisenberg} in our cases.
By Lemma 4.4 of \cite{FRY2}, $\cP^{s,0}$ is contained inside $\cP^{s,0}\otimes V_2$ if and only if the same is true for $V_2$ inside $\cP^{s,0}\otimes\cP^{0,s}=\cP^{s,s}$, and the two multiplicities are the same. Since $\cP^{s,s}=\sum_{i=0}^s |v_1|^{2(s-i)}\cH^{i,i}$, condition (i) is easily verified on the basis of \eqref{Sp_n} and \eqref{Sp_1xSp_n}. 

Condition (ii) for the pairs at line $a'$ is already contained in Proposition 4.10 of \cite{FRY2}. We consider then the pairs of line $b'$, with $V_2=V_{m,i}$. 

We first replace the elements of $\boldsymbol\rho_{\fv_1,\fv_2}$ in Table~\ref{quotient-invariants} with  $p_{V_{1,0}(\fv_1),V_{1,0}(\fv_2)}$ and $p_{V_{1,1}(\fv_1),V_{1,1}(\fv_2)}$, which we simply write $p_{V_{1,0}}$ and $p_{V_{1,1}}$.

We then expand $p_{V_{1,0}}$ and $p_{V_{1,1}}$ in terms of equivalent orthonormal bases $\{\ell_{0,j}\}$ and $\{\ell_{1,j'}\}$ of $V_{1,0}$ and $V_{1,1}$ respectively, as in \eqref{p_VW}.
Then, 
$$
\span_\bC\big\{\ell_{j,0},\ell_{j',1}\big\}_{j,j'}=\cH^{1,1}=\big\{v_1Cv_1^*:C\in\fs\fl_{2n}(\bC)\big\}\ . 
$$

For $C\in\fs\fl_{2n}(\bC)$, define $\ell_C(v)=vCv^*$ on $\bC^{2n}$, and $L_C=\la'_{N_1}(\ell_C)$. According to \cite[Lemma 4.11]{FRY2}, the restriction of the operators
$d\pi(L_C)$ to $\cP^{s,0}(\fv_1)$ coincides, up to a nonzero scalar factor depending on $\pi$, with $d\sigma(C)$, the differential of the natural action of ${\rm SL}_{2n}(\bC)$.

We choose the orthonormal bases so that $\ell_{1,0}=\ell_{C_0}$, $\ell_{1,1}=\ell_{C_1}$ are lowest-weight vectors, respectively in $V_{1,0}$ of weight $-2\varpi_1$, and in $V_{1,1}$ of weight $-\varpi_2$. 
Given  $D=\sum_r D_r\otimes q_r\in\bD(N_1)\otimes\cP^k(\fv_2)$ and $\ell\in\cP^k(\fv_2)$, the notation $\lan D,\ell\ran$ stands for $\sum_r\lan q_r,\ell\ran D_r\in \bD(N_1)$, and similarly with differential operators replaced by  their images $\pi(D_r)$ in the Bargmann representations, or by polynomials on $\fv_1$.

We claim that
\begin{enumerate}
\item[({\sf i})] for every $m,i$, $i\le m$, $\ell_{1,0}^{m-i}\ell_{1,1}^i\in V_{m,i}$;
\item[({\sf ii})] $\lan M_{m,i},\ell_{1,0}^{m-i}\ell_{1,1}^i\ran$ equals, up to a nonzero scalar,  $\lan\widehat{D_{1,0}^{m-i}D_{1,1}^i},\ell_{1,0}^{m-i}\ell_{1,1}^i\ran$;
\item[({\sf iii})] for $s\ge m$, $\big\lan d\pi(M_{V_{m,i}}), \ell_{1,0}^{m-i}\ell_{1,1}^i\big\ran\in\cL\big(\cP^{s,0}(\fv_1)\big)$ is nonzero\footnote{The scalar product is taken in the second factor, $V_{m,i}$, of the tensor product.}. 
\end{enumerate}

To prove ({\sf i}), notice first that each $\ell_{1,0}^{m-i}\ell_{1,1}^i$ 
is a lowest-weight vector vector and its weight, $-2(m-i)\varpi_1-i\varpi_2$, does not appear 
among the lowest-weights 
in the lower-degree harmonic spaces. 
Therefore $\ell_{1,0}^{m-i}\ell_{1,1}^i\in\cH^{m,m}$. Being a lower-weight vector means 
that the one-dimensional spaces that it generates is invariant under the action of a Borel subalgebra of $\fs\fl_{2n}$. 
Then the same is true for $\ell_{1,0}^{m-i}\ell_{1,1}^i$. Since, by \eqref{Sp_n}, $V_{m,i}$ is the only irreducible subspace of $\cH^{m,m}$ with the highest weight $2(m-i)\varpi_1+i\varpi_2$,  $\ell_{1,0}^{m-i}\ell_{1,1}^i$ must be the lowest-weight vector in $V_{m,i}$.

To prove ({\sf ii}), we first observe that the statement is true with $M_{m,i}$ and $\widehat{D_{1,0}^{m-i}D_{1,1}^i}$ replaced by $p_{V_{m,i}}$  and $p_{V_{1,0}}^{m-i}p_{V_{1,1}}^i$ respectively. To see this, it suffices to replace $p_{V_{1,0}}^{m-i}p_{V_{1,1}}^i$ with its component $q_{m,i}$   in $\cH^{m,m}(\fv_1)\otimes\cH^{m,m}(\fv_2)$, which is a linear combination of the $p_{V_{m,j}}$ with $j=0,\dots,m$. On the other hand, comparing the expansions  \eqref{p_VW} of $p_{V_{m,i}}$, $p_{V_{1,0}}$ and $p_{V_{1,1}}$ with bases consisting of weight vectors, we must have a triangular set of linear relations
$$
\begin{array}{rl}
q_{m,m}=&a_mp_{V_{m,m}}\\
q_{m,m-1}=&b_{m,m-1}p_{V_{m,m}}+a_{m-1}p_{V_{m,m-1}}\\
\cdots=&\cdots
\end{array}
$$
just because the weights $2(m-j)\varpi_1+j\varpi_2$ are decreasing in $j$ (i.e., $2\varpi_1-\varpi_2$ is a positive root). The fact that the homogeneous Hilbert bases in Table \ref{quotient-invariants} is free \cite{knop-mf} implies that each coefficient $a_j$ is nonzero. It is then clear that 
$$
\lan p_{V_{1,0}}^{m-i}p_{V_{1,1}}^i, \ell_{1,0}^{m-i}\ell_{1,1}^i\ran= a_i\lan p_{V_{m,i}},\ell_{1,0}^{m-i}\ell_{1,1}^i\ran\ .
$$

It is now sufficient to apply the symmetrisation $\la'_{N_1}$ to both sides and observe that $\la'_{N_1}$ is multiplicative up to terms containing factors in $\fz_1$, which will necessarily have lower degrees in the vector fields in $\fv_1$, and hence in $\fv_2$.

To prove ({\sf iii}), we use the identity
$$
\begin{aligned}
\big\lan d\pi(M_{m,i}), \ell_{1,1}^{m-i}\ell_{2,1}^i\big\ran&=\big\lan d\pi(\widehat{D_1^{m-i}D_2^i}), \ell_{1,1}^{m-i}\ell_{2,1}^i\big\ran\\
&=d\pi(L_{C_1}^{m-i}L_{C_2}^i)\\
&=d\sigma(C_1)^{m-i}d\sigma(C_2)^i\ .
\end{aligned}
$$

In an appropriate coordinate system $(\zeta_1,\dots,\zeta_{2n})$ on $\fv_1$ (recall that $n\ge 2$), we can take 
$$
d\sigma(C_1)=\zeta_{n+1}\de_{\zeta_1}\ ,\qquad d\sigma(C_2)=\zeta_{n+2}\de_{\zeta_1}-\zeta_{n+1}\de_{\zeta_2}\ ,
$$
so  that $d\sigma(C_1)^{m-i}d\sigma(C_2)^i$ does not vanish on $\cP^{s,0}$ for $s\ge m$ (e.g., check the action on $\zeta_1^s$).

Finally, the case of line $c'$ leads to essentially the same situation, with $i=m$. This concludes the proof of Lemma \ref{hadamard3}.
\end{proof}

\bigskip

\subsection{First-generation quotient pairs: proof of Property (S) for pairs at lines $a$ and $b$}\label{first-ab}\quad
\medskip

As usual, we denote by $\cD$ the homogeneous basis of $\bD(N)^K$ obtained from the invariants in Table \ref{quotient-invariants}, with $D_1=\la'_N(z_1)$, $D_2=\la'_N(z_2)$. Given a point $\xi\in\Sigma_\cD$ with $\xi_1\xi_2\ne0$ (a {\it regular point}), we set $(\xi_1,\xi_2)= \la u_\theta$, with $u_\theta=(\cos\theta,\sin\theta)$. Then the corresponding spherical function $\ph_\xi$ factors to the Heisenberg group $N_\theta$ with Lie algebra $\fn_\theta=(\fv_1\oplus\fv_2)\oplus\big(\fz/u_\theta^\perp\big)\cong(\fv_1\oplus\fv_2)\oplus\bR u_\theta$. 

Notice that the pairs $(N_\theta,K)$ are all isomorphic to the Heisenberg pair $(H_m,K)$ with $m=\dim_\bC(\fv_1\oplus\fv_2)$. The map
$(v_1,v_2,t)\longmapsto(v_1,v_2,tu_\theta)$
is an isomorphism from
 $H_m$ to $N_\theta$.

On $H_m$ we keep the invariants in $\boldsymbol\rho_\fv$ of Table \ref{quotient-invariants}, adding to them the coordinate function $t$ on the centre. Calling $\cD_H$ the resulting system of differential operators, the point of $\Sigma_{\cD_H}$ corresponding to $\ph_\xi$ has coordinates\footnote{The coordinate $\xi_6$ is only present  in the pairs of line $b$, $n\ge2$.}
$$
\Psi(\xi)=\Big(\la\,,\,\frac{\xi_3}{\cos\theta}\,,\,\frac{\xi_4}{\sin\theta}\,,\,\frac{\xi_5}{\sin\theta\cos\theta}\,,\,\frac{\xi_6}{\sin\theta\cos\theta}\Big)\ .
$$

For $F\in \cS(N)^K$, $\cG F(\xi)=\cG_H (\cR_\theta F)\big(\Psi(\xi)\big)$, with
$$
\cR_\theta F(v_1,v_2,t)=\int_{u_\theta^\perp}F(v_1,v_2,tu_{\theta+\pi/2})\,dt\in\cS(H_m)^K\ .
$$

Let $\eta(\theta)$ be a $C^\infty$ function supported on a compact interval $I\subset (-\pi,\pi)\setminus\{0,\pm\frac\pi2\}$ and equal to 1 on a neighbourhood of a point $\theta_0$. Then the  function $F^\#(v_1,v_2,t,s)=\cF\inv_\theta\big(\eta(\theta)\cR_\theta(v_1,v_2,t) \big)(s)$ is $K$-invariant and Schwartz on $N^\#=H_m\times\bR$. With respect to the system $\cD^\#=\cD_H\cup\{i\inv\de_s\}$, the spherical transform of $F^\#$ on $\Sigma_{\cD^\#}=\Sigma_{\cD_H}\times\bR$ is 
$$
\cG^\#(F^\#)\big(\Psi(\xi),\theta\big)=\cG_H (\cR_\theta F)\big(\Psi(\xi)\big)=\cG F(\xi)\ .
$$

By Proposition \ref{product} and \cite{ADR2}, $\cG^\#(F^\#)\big(\Psi(\xi),\theta\big)$ admits a Schwartz extension from $\Sigma_{\cD^\#}$ to the ambient space. Since the map $\xi\longmapsto \big(\Psi(\xi),\theta\big)$ is a diffeomorphism on the set of $\xi$ with $\la>0$ and $\theta\in I$, we conclude that $\cG F$ admits a smooth extension to any set of $\xi\in\bR^d$ ($d=5$ or 6) with $(\xi_1,\xi_2)$ varying in a compact set with $\xi_1\xi_2\ne0$, with derivatives of any order decaying rapidly in the remaining variables.

Denote by $\cS_{0,0}(N)$, resp.  $\cS_0(N)$, the subspaces of $\cS(N)$ whose elements have vanishing moments of any order in $z_1$ and $z_2$ separately, i.e.
$$
\int_{\fz_1}F(v_1,v_2,z_1,z_2)z_1^k\,dz_1=\int_{\fz_2}F(v_1,v_2,z_1,z_2)z_2^k\,dz_2=0\ ,
$$
for every $k$ and every choice of the non-integrated variables,
resp. vanishing moments of any order in $z_1$ and $z_2$ jointly, i.e.
$$
\int_{\fz_1\oplus\fz_2}F(v_1,v_2,z_1,z_2)z_1^{k_1}z_2^{k_2}\,dz_1\,dz_2=0\ ,
$$
for every $k_1,k_2$ and every choice of $v_1,v_2$. Clearly, $\cS_{0,0}(N)\subset\cS_0(N)$. 

Then $\cS_{0,0}(N)^K$ consists of the functions $F\in\cS(N)^K$ with $\cG F$ vanishing of infinite order on the singular set (where $\xi_1\xi_2=0$), and $\cS_0(N)^K$  of the functions $F\in\cS(N)^K$ with $\cG F$ vanishing of infinite order where $\xi_1=\xi_2=0$.

Since the derivatives of $\Psi$ grow at most polynomially as $\xi_1$ or $\xi_2$ tends to 0, we can say, using an appropriate partition of unity on the regular set as in the proof of Proposition \ref{S_0}, that, for every $F\in\cS_{0,0}(N)^K$, $\cG F$ admits a Schwartz extension to $\bR^d$.

We invoke now the following analogue of Proposition \ref{hadamard}.

\begin{lemma}\label{new-hadamard}\quad
\begin{enumerate}
\item[\rm(i)] Let $F\in\cS(N)^K$, and assume that
$$
F(v_1,v_2,z_1,z_2)=\de_{z_2}^kR_k(v_1,v_2,z_1,z_2)\ ,
$$
with $R_k\in\cS(N)^K$. Then there is a function $F_k\in\cS(N)^K$ such that
$\cG F_k$ does not depend on $\xi_2$ and
$$
F=\de_{z_2}^kF_k+\de_{z_2}^{k+1}R_{k+1}\ ,
$$
with $R_k\in\cS(N)^K$.
\item[\rm(ii)] Let $F\in\cS(N)^K$, and assume that
$$
F(v_1,v_2,z_1,z_2)=\sum_{i+j=k}\de_{z_1}^i\de_{z_2}^jR_{i,j}(v_1,v_2,z_1,z_2)\ ,
$$
with $R_{i,j}\in\cS(N)^K$ for every $i,j$. Then there are functions $F_{i,j}\in\cS(N)^K$, with $i+j=k$ with spherical transforms
$\cG F_{i,j}$ that do not depend on $\xi_1,\xi_2$ and
$$
F=\de_{z_1}^i\de_{z_2}^jF_{i,j}+\sum_{r+s=k+1}\de_{z_1}^r\de_{z_2}^sR_{r,s}\ .
$$
with $R_{r,s}\in\cS(N)^K$ for every $r,s$.
\end{enumerate}
\end{lemma}

This is essentially Geller's lemma as cited in \cite{ADR1}, cf. Theorem 2.2 and Lemma 5.2 therein. We sketch the proof for completeness.

\begin{proof} Let $R_k^\flat(v_1,v_2,z_1)=\int_{\fz_2}R_k(v_1,v_2,z_1,z_2)\,dz_2$. Then $R_k^\flat\in\cS(N_1\times\fv_2)^K$ and the spherical transform $\cG'R_k^\flat(\xi_1,\xi_3,\cdots)$ -- for the pair $(N_1\oplus\fv_2,K)$ -- coincides with $\cG R_k(\xi_1,0,\xi_3,\cdots)$. By Proposition \ref{prop_quotient}, $\cG'R_k^\flat$ extends to a Schwartz function $g_k(\xi_1,\xi_3,\cdots)$. 

Setting $g_k^\sharp(\xi_1,\xi_2,\xi_3,\cdots)=g_k(\xi_1,\xi_3,\cdots)$, the restriction of $g_k^\sharp$ to the Gelfand spectrum $\Sigma_\cD$ for the pair $(N,K)$ is in $\cS(\Sigma_\cD)$. By Theorem \ref{hulanicki}, there is $G_k\in\cS(N)^K$ such that $\cG G_k={g_k^\sharp}_{|_{\Sigma_\cD}}$. Then $\cG(R_k-G_k)$ vanishes for $\xi_2=0$, and this implies, by Hadamard's lemma, that $R_k-G_k$ is the $z_2$-derivative of a $K$-invariant Schwartz function $R_{k+1}$.

This proves (i). The proof of (ii) is similar.
\end{proof}

Using part (i) of Lemma \ref{new-hadamard}, we can repeat the proof of Proposition \ref{F-G} to prove that the spherical transforms of functions in $\cS_0(N)^K$ admit a Schwartz extension, and next, using part (ii), that the same is true for general functions in $\cS(N)^K$.

\bigskip

\subsection{First-generation quotient pairs: proof of Property (S) for pairs at line $c$}\label{first-c}\quad
\medskip

In this last case, $K$ acts nontrivially on one component, $\fz_2$, of the centre. As in Section \ref{quotient}, we reduce part of the proof to a quotient pair $(N',K')$, factoring $\fz_2$ by (any) two-dimensional subspace and taking $K'$ as the stabilizer of the factored subspace in $K$. This quotient pair is isomorphic to the pair at line $b$ with the same $n$, only with $K_2={\rm U}_1\rtimes\bZ_2$. For simplicity, we factor out the subspace orthogonal to $i$ in $\IM\bH$.

Denoting by $p_1,\dots,p_6$, resp. $p'_1,\dots,p'_6$, the invariants in Table \ref{quotient-invariants} on $\fn$ and $\fn'$ respectively (in the same order and with $p'_2=z_2^2$ to take into account the extra $\bZ_2$), we have the following relations:
\begin{equation}\label{restrictions-p}
{p_j}_{|_{\fn'}}=p'_j\ ,\quad (1\le j\le 5)\ ,\qquad {p_6}_{|_{\fn'}}=\sqrt{p'_2}(p'_6-p'_5)\ .
\end{equation}

Denoting by $\cD$, $\cD'$ the symmetrisations of the systems  $\{p_j\}$ and $\{p'_j\}$ respectively, and by $\Sigma_\cD$, 
$\Sigma_{\cD'}$ the corresponding spectra, we have a continuous surjection of $\Sigma_{\cD'}$ onto $\Sigma_\cD$ (cf. Proposition \ref{Lambda^gamma}) which is a diffeomorphism from the complement of the singular set $\xi'_2=0$ in $\Sigma_{\cD'}$ to the complement of the singular set $\xi_2=0$ in $\Sigma_\cD$ (cf. Proposition \ref{tildeQ}).

By Proposition \ref{normal}, the pair $(N',K')$ satisfies Property (S),   hence a repetition of the arguments used in Section~\ref{sec_towards} implies that the spherical transforms of functions $F\in\cS_0(N)^K$ admit a Schwartz extension to $\bR^6$, where $\cS_0(N)$ is the space of Schwartz functions with
$$
\int_{\fz_2}F(v_1,v_2,z_1,z_2)z_2^\al\,dz_2=0\ ,
$$
for every $\al\in\bN^3$ and every $v_1,v_2,z_1$.

At this point, we are back at the situation of Section \ref{section_checkN}, with $\fz_1$ and $\fz_2$ playing the r\^ole that was, respectively, of $\check\fz$ and $\fz_0$. Setting, as in Section \ref{section_checkN}, $\check N=N_1\times\fv_2$, we observe that, once again, the singular part of $\Sigma_\cD$,
$$
\check\Sigma_\cD=\{\xi\in\Sigma_\cD:\xi_2=0\}\ ,
$$
is naturally identified with the spectrum $\Sigma_{\check\cD}$ of the pair $(\check N, K)$, i.e., the pair at line $c'$ of Table \ref{quotient-invariants}, with $\check\cD$ constructed from the invariants in the same table.

Following again the procedure of Section~\ref{section_checkN}, we reduce matters to proving an adapted version of Proposition~\ref{prop_FRY2}. 

We set $\widetilde D_j=\la'_{\check N}(p_j)$, where $p_1,\dots,p_6$ are the invariants at line $c$ of Table \ref{quotient-invariants}.
Notice that  only $p_2=|z_2|^2$ and $p_6=\RE\big(i(v_1v_2^*)z_2(v_2v_1^*)\big)$ contain the variable $z_2$. 

\begin{lemma}\label{last-hadamard}
Let $G$ be a $K$-invariant function on $\check N\times\fz_2$ of the form
\begin{equation}\label{Gnew}
 G(v_1,v_2,z_1,z_2)=\sum_{|\al|=k}z_2^\al G_\al(v_1,v_2,z_1)\ ,
\end{equation}
 with $G_\al\in\cS(\check N)$ and $z_2\in\fz_2=\IM\bH$. 
 There exist functions $H_{j,m}\in\cS(\check N)^K$, for $2j+m=k$, such that
\begin{equation}\label{Halnew}
 G=\sum_{2j+m= k}\frac1{j!m!}|z_2|^{2j}\widetilde D_6^m H_{j,m}\ .
\end{equation}
\end{lemma}

Once this is proved\footnote{Once again, we do not need to worry about control of the Schwartz norms of $H_{j,m}$ in terms of Schwartz norms of the $G_\al$.}, we can repeat the proofs of Proposition~\ref{hadamard} and Proposition~\ref{F-G} to obtain the conclusion. To prove the lemma, we must adapt part of the proof of Proposition~\ref{prop_FRY2}, given in \cite[Section 5]{FRY2} to the new situation where $\check N$ has an abelian factor.

\begin{proof}[Proof of Lemma \ref{last-hadamard}] Repeating arguments used before, it is easy to see that,
for every $m$, there is a unique $K$-invariant polynomial $\widetilde p_m$ on $\fn$, belonging to $\big(\cH^{m,m}(\fv_1)\otimes\cH^{2m}(\fv_2)\otimes\cH^m(\fz_2)\big)^K$, and that  the $\big(\cH^{m,m}(\fv_1)\otimes\cH^{2m}(\fv_2)\otimes\cH^m(\fz_2)\big)$-component of $p_6^m$ is nonzero.
Then, every $K$-invariant polynomial $p$ on $\fn$ can be uniquely expressed as a finite sum
$$
p=\sum_{i,m} |z_2|^{2i}\,\widetilde p_m\,q_{i,m}(v_1,v_2,z_1)\ ,
$$
with $q_{i,m}\in\cP(\check \fn)^K$.

We  decompose $G\in\big(\cS(\check N)\otimes\cP^k(\fz_2)\big)^K$ as
$$
G(v_1,v_2,z_1,z_2)=\sum_{2j+m=k}|z_2|^{2j}G_m(v_1,v_2,z_1,z_2)\ ,
$$
with $G_m\in\big(\cS(\check N)\otimes\cH^m(\fz_2)\big)^K$. It will suffice to show that
\begin{equation}\label{M_6}
G_m=M_m H_m\ ,
\end{equation}
where $M_m=\la'_{\check N}(\widetilde p_m)\in\big(\bD(\check N)\otimes\cH^m(\fz_2)\big)^K$ and $H_m\in\cS(\check N)^K$.

Since $\check N$ is the product of the Heisenberg group $N_1$ and the abelian group $\fv_2$, the infinite dimensional irreducible unitary representations of $\check N$ are the tensor products $\pi_\la\otimes \chi_\omega$, where $\pi_\la$ is a Bargmann representation of $N_1$ on $\cF(\fv_1)$, and $\chi_\omega$ is the character $e^{i\lan\cdot,\omega\ran}$ of $\fv_2$, with $\omega\in\fv_2$ (cf. Lemma \ref{representations}). 

In the generic case $\omega\ne0$, the stabiliser $K_\omega$ of $\pi_\la\otimes\chi_\omega$ in $K$ is isomorphic to ${\rm U}_1\times {\rm Sp}_1\times {\rm Sp}_{n-1}$, with actions on $\fv_1=\bH\omega\oplus(\bH\omega)^\perp\cong\bH\oplus\bH^{n-1}$ and on $\fz_2=\IM\bH$ given by
$$
{\rm U}_1\times {\rm Sp}_1\times {\rm Sp}_{n-1}\ni(e^{i\theta},k,k')\ :\begin{cases}(v,v')\longmapsto e^{i\theta}(vk\inv,v'{k'}\inv) &\text{ for }(v,v')\in \bH\oplus\bH^{n-1}\ ,\\ z_2\longmapsto kz_2k\inv&\text { for }z_2\in\IM\bH\ .
\end{cases}
$$

Then the $K_\omega$-invariant irreducible subspaces of $\cF(\fv_1)$ are  the tensor products $V^\omega_{s_1,s_2}=\cP^{s_1,0}(\bH)\otimes\cP^{s_2,0}(\bH^{n-1})$. By  \cite[Proposition 4.5]{FRY2}, $\pi_\la\otimes\chi_\omega(G)$ can be nonzero only if
$V^\omega_{s_1,s_2}$ is contained, as a representation space of $K_\omega$, in  $V^\omega_{s_1,s_2}\otimes\cH^m(\fz_2)$. This is equivalent to saying that $\cP^{s_1,0}(\bH)$ is contained (and with the same multiplicity) in $\cP^{s_1,0}(\bH)\otimes\cH^m(\fz_2)$ as a representation space of ${\rm U}_1\times {\rm Sp}_1$. By \cite[Proposition~4.6]{FRY2}, this happens if and only if $s_1\ge m$, and with multiplicity 1. 

We show next that, for $s_1\ge m$, $d(\pi_\la\otimes\chi_\omega)(M_m)$ is nonzero. Observe that
$$
d(\pi_\la\otimes\chi_\omega)(M_m)=i^md\pi_\la\big(\la'_{N_1}\big(\widetilde p_m(v_1,\omega, z_1)\big)\big)\ ,
$$
and that 
$$
p_m(\,\cdot\,,\,\omega\,,\, \cdot\,)\in \big(\cH^{m,m}(\fv_1)\otimes\cH^m(\fz_1)\big)^{K_\omega}=  \big(\cH^{m,m}(\bH\omega)\otimes\cH^m(\fz_1)\big)^{{\rm U}_1\times {\rm Sp}_1}\ .
$$

This is one of the cases considered in the proof of  \cite[Proposition 4.10]{FRY2}, and we conclude that $d(\pi_\la\otimes\chi_\omega)(M_m)\ne0$.

So we are in the following situation: the spectrum $\Sigma_{\check\cD}$ of $(\check N,K)$ is the closure in $\bR^4$   of the set of points $\xi$ with
$$
\xi_1=\la\ ,\quad\xi_2=2|\la|(s_1+s_2+n)\ ,\quad\xi_3=|\omega|^2\ ,\quad\xi_4=2|\omega|^2|\la|(s_1-s_2-n+2)\ .
$$

It is convenient to replace $\xi_4$ by 
$$
\xi'_4=\xi_4+\xi_2\xi_3=4|\omega|^2|\la|(s_1+1)\ ,
$$
and set $\xi'=(\xi_1,\xi_2,\xi_3,\xi'_4)$. Correspondingly, we replace $\check D_4\in\check\cD$ by $\check D'_4=\check D_4+\check D_2\check D_3$ and call $\check \cD'$ the resulting system of differential operators.

The operator $U_m=M_m^*M_m$ is in $\bD(\check N)^K$, so there is a polynomial $u_m$ such that 
$U_m=u_m(\check\cD')$, and 
$$
d(\pi_\la\otimes\chi_\omega)(U_m)_{V^\omega_{s_1,s_2}}=u_m(\xi')\ .
$$

Then $u_m$ vanishes on the set $E_m\subset \Sigma_{\check\cD'}$ of points $\xi'$ corresponding to $s_1=0,\dots,m-1$, i.e. where $\xi_4=j|\xi_1|\xi_3$, with $j=1,\dots,m$. Repeating the proof of  \cite[Lemma 5.1]{FRY2}, we conclude that 
$$
u_m(\xi)=c_m\prod_{j=1}^m(\xi'_4-j\xi_1\xi_3)\,\prod_{j=1}^m(\xi'_4+j\xi_1\xi_3)\ ,
$$
with $c_m\ne0$.

Consider now $M_m^*G_m\in\cS(\check N)^K$. Then $\check\cG(M_m^*G_m)$ vanishes on $E_m$. We can then apply  \cite[Proposition 5.2]{FRY2} to conclude that $\check\cG(M_m^*G_m)$ admits an extension to $\bR^4$ of the form $u_m\psi$ with $\psi$ Schwartz. If $H_m\in\cS(\check N)^K$ is such that $\check\cG H_m=\psi_{|_{\Sigma_{\check\cD'}}}$, we have $M_m^*G_m=U_mH_m$. Repeating the conclusion   of \cite[Section 5]{FRY2}, we then have
$$
\big(d(\pi_\la\otimes\chi_\omega)(M_m)\big)^*\,(\pi_\la\otimes\chi_\omega)(M_mH_m-G)=0\ .
$$

The conclusion follows by the invertibility of $d(\pi_\la\otimes\chi_\omega)(M_m)$ on the spaces $V^\omega_{s_1,s_2}$ where $(\pi_\la\otimes\chi_\omega)(M_mH_m-G)$ might not vanish.
\end{proof}

\bigskip
\section{Appendix: quotient pairs generated by pairs in Table \ref{vinberg}}\label{appendix}\quad
\bigskip


Let  $I_p$ (resp $0_p$) be the identity (resp. zero) $p\times p$ matrix and
$$
J_p:={\rm diag}(\underbrace{J,J,\dots,J}_{\mbox{$p$ times}})\ .
$$

The Lie bracket $[\ ,\ ]$ is understood as a map from $\fv\times\fv$ to $\fz$, elements of $\bR^n,\,\bC^n$ etc. as column vectors.

\bigskip

\subsection{The pair $(\bR^n\oplus \fs\fo_n,{\rm SO}_n)$}\label{line1}\quad
\medskip

The Lie bracket is
$$
[v,v']=\half(v\,\trans v'-v'\,\trans v)\ .
$$

Up to conjugation by an element of $K$, we may assume that $t$ has the form
$$
t={\rm diag}(t_1J_{p_1},\dots,t_kJ_{p_k},0_q)\ ,
$$
with $2p_1+\cdots+2p_k+q=n$ and $t_i\ne t_j\ne0$ for every $i\ne j$. Then we have
$$
\begin{aligned}
K_t&={\rm U}_{p_1}\times\cdots\times {\rm U}_{p_k}\times{\rm SO}_q\ ,\\
\fz_t&=\fu_{p_1}\oplus\cdots\oplus \fu_{p_k}\oplus\fs\fo_q\ ,\\
\fn_t&=(\bC^{p_1}\oplus\fu_{p_1})\oplus\cdots\oplus (\bC^{p_1}\oplus\fu_{p_k})\oplus(\bR^q\oplus\fs\fo_q)\ .
\end{aligned}
$$



\subsection{The pair $(\bC^n\oplus \fu_n,{\rm U}_n)$}\label{line2}\quad
\medskip

The Lie bracket is
$$
[v,v']=\half(v{v'}^*-v'v^*)\ .
$$

For 
$$
t={\rm diag}(it_1I_{p_1},\dots,it_kI_{p_k})\ ,
$$
with $p_1+\cdots+p_k=n$ and $t_i\ne t_j$ for every $i\ne j$, we have
$$
\begin{aligned}
K_t&={\rm U}_{p_1}\times\cdots\times {\rm U}_{p_k}\ ,\\
\fz_t&=\fu_{p_1}\oplus\cdots\oplus \fu_{p_k}\ ,\\
\fn_t&=(\bC^{p_1}\oplus\fu_{p_1})\oplus\cdots\oplus (\bC^{p_1}\oplus\fu_{p_k})\ .
\end{aligned}
$$



\subsection{The pair $\big(\bH^n\oplus (HS^2_0\bH^n\oplus\IM\bH),{\rm Sp}_n\big)$}\label{line3}\quad
\medskip

The Lie bracket is
$$
\begin{aligned}[]
[v,v']&=\half\Big(vi{v'}^*-v'iv^*-\frac1n\tr(vi{v'}^*-v'iv^*)I_n\Big)\oplus\IM(v^*v')\\
&=\Big(\half (vi{v'}^*-v'iv^*)-\frac1n\IM\!_i(v^*v')I_n\Big)\oplus\IM(v^*v')\ ,
\end{aligned}
$$
where $\IM\!_i$ denotes the $i$-component of the argument.

For 
$$
t={\rm diag}(t_1I_{p_1},\dots,t_kI_{p_k})\oplus (u_1i+u_2j+u_3k)\ ,
$$
with $p_1+\cdots+p_k=n$ and $t_i\ne t_j$ for every $i\ne j$, we have
$$
\begin{aligned}
K_t&={\rm Sp}_{p_1}\times\cdots\times {\rm Sp}_{p_k}\ ,\\
\fz_t&=(HS^2\bH^{p_1}\oplus\cdots\oplus HS^2\bH^{p_k})_0\oplus\IM\bH\ .
\end{aligned}
$$

Decomposing $v\in\bH^n$ as $v_1\oplus\cdots\oplus v_k$ with $v_j\in\bH^{p_j}$, the Lie bracket in  $\fn_t$  is 
$$
[v,v']_t=\begin{pmatrix}\half(v_1i{v'}_1^*-v'_1iv_1^*)-\frac1n\IM\!_i(v^*v')I_{p_1}&&\\&\ddots&\\&&v_ki{v'}_k^*-v'_kiv_k^*-\frac1n\IM\!_i(v^*v')I_{p_k}\end{pmatrix}\oplus\IM(v^*v')\ .
$$

If we consider, for $1\le j\le k$, the subalgebra $\fh_j$ of $\fn_t$ generated by $\bH^{p_j}$,
$$
\fh_j=\bH^{p_j}\oplus (HS^2\bH^{p_j}\oplus\IM\bH)\big)=\big(\bH^{p_j}\oplus (HS^2_0\bH^{p_j}\oplus\IM\bH)\big)\oplus\bR\ ,
$$
we easily see that it is $K_t$-invariant and  only the factor ${\rm Sp}_{p_j}$ of $K_t$ acts nontrivially on it.  Since $\fh_j$ commutes with $\fh_{j'}$ for $j\ne j'$, it follows that $\fn_t$ is the quotient, modulo a central ideal, of the product of the $\fh_j$.

We conclude that $(\fn_t,K_t)$ is a central reduction of the product of the pairs $(\fh_j,{\rm Sp}_{p_j})$, where, in turn, each $(\fh_j,{\rm Sp}_{p_j})$ is the product of $\big(\bH^{p_j}\oplus (HS^2_0\bH^{p_j}\oplus\IM\bH),{\rm Sp}_{p_j}\big)$ and the trivial pair $(\bR,\{1\})$.


\subsection{The pairs $(\bC^{2n+1}\oplus \Lambda^2\bC^{2n+1},{\rm SU}_{2n+1})$ and $\big(\bC^{2n+1}\oplus (\Lambda^2\bC^{2n+1}\oplus\bR)\big),{\rm U}_{2n+1})$}\label{lines45}\quad
\medskip

To fix the notation, we consider the second family of pairs, the other being analogous and simpler. The Lie bracket is
\begin{equation}\label{bracketline5}
\begin{aligned}[]
[v,v']&=\half(v\trans v'-v'\trans v)\oplus \IM(v^*v')\ .
\end{aligned}
\end{equation}

For 
$$
t={\rm diag}(t_1J_{p_1},\dots,t_kJ_{p_k},0_{2q+1})\oplus u\ ,
$$
with $p_1+\cdots+p_k+q=n$ and $t_i\in\bR$, $t_i\ne t_j\ne0$ for $i\ne j$, we have
$$
\begin{aligned}
K_t&={\rm Sp}_{p_1}\times\cdots\times {\rm Sp}_{p_k}\times {\rm U}_{2q+1}\\
\fz_t&=HS^2\bH^{p_1}J_{p_1}\oplus\cdots\oplus HS^2\bH^{p_k}J_{p_k}\oplus \Lambda^2\bC^{2q+1}\oplus\bR\ .
\end{aligned}
$$

Like in the previous case, we split $\bC^{2n+1}$ as $\bC^{2p_1}\oplus\cdots\oplus \bC^{2p_k}\oplus\bC^{2q+1}$, and set
$$
\fh_j=\bC^{2p_j}\oplus(HS^2\bH^{p_j}J_{p_j}\oplus\bR)\ ,\quad(j=1,\dots,k)\ ,\qquad \fh_{k+1}=\bC^{2q+1}\oplus  (\Lambda^2\bC^{2q+1}\oplus\bR)\ ,
$$
i.e., the subalgebra generated by the $j$-th summand in $\bC^{2n+1}$. Then, for $1\le j\le k$,
$$
\fh_j\cong \bH^{p_j}\oplus(HS^2_0\bH^{p_j}\oplus\bR^2)\ ,
$$
and $(\fh_j,{\rm Sp}_{p_j})$ is isomorphic to a central reduction of the pair in subsection \ref{line3}. Finally, $(\fn_t,K_t)$ is isomorphic to a central reduction of the product of $k$ pairs of this kind and the pair $(\fh_{k+1},{\rm U}_{2q+1})$.


\subsection{The pair $\big(\bC^{2n}\oplus (\Lambda^2\bC^{2n}\oplus\bR),{\rm SU}_{2n}\big)$}\label{line6}\quad
\medskip

The Lie bracket is given by \eqref{bracketline5}. Any element $z$ of $\fz$  is conjugate, modulo an element of ${\rm U}_{2n}$, to an element of the form
\begin{equation}\label{diag}
t={\rm diag}(t_1J_{p_1},\dots,t_kJ_{p_k},0_{2q})\oplus u\ ,
\end{equation}
with $p_1+\cdots+p_k+q=n$ and $t_i\in\bR$,  $t_i\ne t_j$ for $i\ne j$. 

If $q>0$, then $z$ and $t$ are also conjugate under ${\rm SU}_{2n}$. If $q=0$, then there exists $e^{i\theta}$, unique up to a $2n$-th root of unity, such that $z$ is conjugate to 
$$
t_\theta={\rm diag}(t_1e^{i\theta}J_{p_1},\dots,t_ke^{i\theta}J_{p_k})\oplus u\ .
$$

Then
$$
K_{e^{i\theta}t}=K_t\ ,\qquad
\fz_{e^{i\theta}t}=e^{i\theta}\fz_t\ .
$$

For an element $t$ as in \eqref{diag}, we have
$$
\begin{aligned}
K_t&={\rm Sp}_{p_1}\times\cdots\times {\rm Sp}_{p_k}\times {\rm SU}_{2q}\\
\fz_t&=\begin{cases}HS^2\bH^{p_1}J_{p_1}\oplus\cdots\oplus HS^2\bH^{p_k}J_{p_k}\oplus \Lambda^2\bC^{2q}\oplus\bR &\text{ if } q\ne0\ ,\\
HS^2\bH^{p_1}J_{p_1}\oplus\cdots\oplus HS^2\bH^{p_k}J_{p_k}
\oplus i\bR\, 
{\rm diag}(  t_1^{-1} J_{p_1}, \ldots, t_k^{-1} J_{p_k})
\oplus\bR &\text{ if } q=0\ .
\end{cases}
\end{aligned}
$$

The discussion proceeds as in subsection \ref{lines45}.


\subsection{The pairs $(\bC^2\otimes\bC^n\oplus \fu_2,{\rm U}_2\times{\rm SU}_n)$, $(\bC^2\otimes\bC^{2n}\oplus \fu_2,{\rm U}_2\times{\rm Sp}_n)$}\label{lines7,8}\quad
\medskip

Realizing the elements of $\fv$ as $n\times 2$ (resp. $2n\times2$) complex matrices, the Lie bracket is
$$
[v,v']=\half(v^*v'-{v'}^*v)\ .
$$

If $t={\rm diag}(it_1,it_2)$, then $K_t=K$, $\fn_t=\fn$ if $t_1=t_2$. 

If $t_1\ne t_2$,
we have
$$
(\fn_t,K_t)=\begin{cases}(\fh_n\oplus\fh_n,{\rm U}_1\times{\rm SU}_n\times {\rm U}_1)&\\
(\fh_{2n}\oplus\fh_{2n},{\rm U}_1\times{\rm Sp}_n\times {\rm U}_1)&,\\
\end{cases}
$$
where $\fh_n=\bC^n\oplus\bR$ is the $2n+1$-dimensional Heisenberg algebra, the factor ${\rm SU}_n$, resp. ${\rm Sp}_n$, acts simultaneously on the two summands $\bC^n$, resp. $\bC^{2n}$ of $\fv$, and the two copies of ${\rm U}_1$ act independently each on one summand.


\subsection{The pair $(\bR^2\otimes\bO\oplus (\IM\bO\oplus\bR),{\rm U}_1\times{\rm Spin}_7)$}\label{line9}\quad
\medskip

Realizing $\fv$ as $\bO^2$, the Lie bracket is
$$
\big[(v_1,v_2),(v'_1,v'_2)\big]=\IM(v_1\overline{v'_1}+v_2\overline{v'_2})\oplus\RE(v_1\overline{v'_2}-\overline{v'_1}v_2)\ .
$$

Take $t=t_1i\oplus a\in\fa$, where $i$ is an imaginary unit in $\bO$. If $t_1=0$, then $K_t=K$ and $\fn_t=\fn$.

If $t_1\ne0$, then $K_t={\rm SO}_2\times{\rm Spin}_6$, where ${\rm Spin}_6\cong  {\rm SU}_4$ consists of the elements $k$ of ${\rm Spin}_7$ whose action on $\fv$ commutes with left multiplication by $i$. Taking this as a complex structure on $\fv$, we can now realize $\fv$ as $\bC^4\oplus\bC^4$.
So, $\fn_t=(\bC^4\oplus\bC^4)\oplus\fa$, with Lie bracket
 $$
\big[(v_1,v_2),(v'_1,v'_2)\big]_t=\IM(v_1\overline{v'_1}+v_2\overline{v'_2})\oplus\RE(v_1\overline{v'_2}-\overline{v'_1}v_2)\ ,
$$
where conjugation and imaginary part are meant now in the complex sense. 

Also, ${\rm SO}_2$ acts by two conjugate characters on the two subspaces $\fv_\pm=\{(v,\pm iv):v\in\bC^4\}$ of $\fv$, which are also ${\rm SU}_4$-invariant. Noticing that $[\fv_+,\fv_-]_t=\{0\}$, it is easy to verify that $\fn_t$ is the direct product
$$
\fn_t=(\fv_+\oplus\bR)\oplus(\fv_-\oplus\bR)\cong\fh_4\oplus\fh_4\ .
$$

We can finally add an extra torus ${\rm U}_1$ to $K_t$, acting on $\fv$ by scalar multiplication, without changing the orbits. In conclusion,
$$
(\fn_t,K_t)=(\fh_4\oplus\fh_4,{\rm U}_1\times {\rm SU}_4\times{\rm U}_1)\ ,
$$
where ${\rm SU}_4$ acts simultaneously on each factor and the two copies of ${\rm U}_1$ act independently each on one summand.


\subsection{The pair $(\bH^2\otimes\bH^n\oplus \fs\fp_2,{\rm Sp}_2\times{\rm Sp}_n)$}\label{line10}\quad
\medskip

Realizing the elements of $\fv$ as $n\times 2$ quaternionic matrices, the Lie bracket is
$$
[v,v']=\half(v^*v'-{v'}^*v)\ .
$$

If $t={\rm diag}(it_1,it_2)$, then $K_t=K$, $\fn_t=\fn$ if $t_1=t_2=0$. 
The proper quotient pairs that appear in the other cases are as follows:
\begin{enumerate}
\item[(i)]
$t_1=t_2\ne0$: then $K_t= {\rm U}_2\times{\rm Sp}_n$ and $\fz_t= \fu_2$, where ${\rm U}_2$, resp. $\fu_2$, is embedded in ${\rm Sp}_2$, resp. in $\fs\fp_2$, as the stabilizer of $iI_2$. So $(\fn_t,K_t)$ is one of the pairs at line 8 of Tables \ref{vinberg} and \ref{invariants}.
\item[(ii)]
$t_1\ne t_2=0$: then $(\fn_t,K_t)=\big((\bH^n\oplus\bR)\oplus(\bH^n\oplus\IM\bH),{\rm U}_1\times{\rm Sp}_n\times {\rm Sp}_1\big)$.
\item[(iii)] $t_1\ne t_2$, $t_1t_2\ne0$: then $(\fh_{2n}\oplus\fh_{2n},{\rm U}_1\times{\rm Sp}_n\times {\rm U}_1)$.
\end{enumerate}


{\scriptsize


  \begin{table}[htdp]
\begin{center}
\begin{tabular}{|r||l|l|l||c|c|c|}
\hline
&$K$&$\fv$&$\fz$&${\boldsymbol\rho}_\fz$&${\boldsymbol\rho}_\fv$&${\boldsymbol\rho}_{\fv,\fz_0}$\\
\hline
\hline
1a&${\rm SO}_{2n}$&$\bR^{2n}$&$\fs\fo_{2n}$&$\begin{matrix}\tr(z^{2k})\\ (1\le k\le n-1)\\{\rm Pf}(z)\end{matrix}$&$|v|^2$&$\begin{matrix}\trans vz^{2k}v\\ (1\le k\le n-1)\end{matrix}$\\
\hline
1b&${\rm SO}_{2n+1}$&$\bR^{2n+1}$&$\fs\fo_{2n+1}$&$\begin{matrix}\tr(z^{2k})\\ (1\le k\le n)\end{matrix}$&$|v|^2$&$\begin{matrix}\trans vz^{2k}v\\ (1\le k\le n-1)\\{\rm Pf}(z|v)\end{matrix}$\\
\hline
2&${\rm U}_n$&$\bC^n$&$\fu_n$&$\begin{matrix}\tr\big(iz)^k\big)\\ (1\le k\le n)\end{matrix}$&$|v|^2$&$\begin{matrix}v^*(iz)^kv\\ (1\le k\le n{-}1)\end{matrix}$\\
\hline
3&$\text{Sp}_n$&$\bH^n$&$HS^2_0\bH^n\oplus\IM\bH$ &$\begin{matrix}\check z_1\,,\,\check z_2\,,\,\check z_3\,,\,\tr z_0^k\\ (2\le k\le n)\end{matrix}$&$|v|^2$&$\begin{matrix}v^*z_0^kv\\ (1\le k\le n{-}1)\end{matrix}$\\
\hline
\hline
4&${\rm SU}_{2n+1}$&$\bC^{2n+1}$&$\Lambda^2\bC^{2n+1}$&$\begin{matrix}\tr\big((\bar zz)^k\big)\\ (1\le k\le n)\end{matrix}$&$|v|^2$&$\begin{matrix}v^*(\bar zz)^kv\\ (1\le k\le n-1)\\ {\rm Pf}(z|v)\ ,\ \overline{{\rm Pf}(z|v)} \end{matrix}$ \\
\hline
5&${\rm U}_{2n+1}$&$\bC^{2n+1}$&$\Lambda^2\bC^{2n+1}\oplus\bR$&$\begin{matrix}\check z\,,\,\tr\big((\bar z_0z_0)^k\big)\\ (1\le k\le n)\end{matrix}$&$|v|^2$&$\begin{matrix}v^*(\bar z_0z_0)^kv\\ (1\le k\le n)\end{matrix}$ \\
\hline
6& ${\rm SU}_{2n}$& $\bC^{2n}$&$\Lambda^2\bC^{2n}\oplus\bR$&$\begin{matrix}\,\check z\,,\tr\big((\bar z_0z_0)^k\big)\\ (1\le k\le n{-}1)\\ {\rm Pf}(z)\,,\, \overline{{\rm Pf}(z)}\end{matrix}$&$|v|^2$&$\begin{matrix}v^*(\bar z_0z_0)^kv\\ (1\le k\le n{-}1)\end{matrix}$\\
\hline
\hline
7&$\begin{matrix}{\rm U}_2{\times} {\rm SU}_n\\  (n\ge2)\end{matrix}$&$\bC^2{\otimes}\bC^n$&$\fu_2$&$\begin{matrix}\tr\big((iz)^k\big)\\  (k=1,2)\end{matrix}$&$\begin{matrix}\tr\big((vv^*)^k\big)\\  (k=1,2)\end{matrix}$&$i\tr(v^*zv)$\\
\hline
8&$\begin{matrix}{\rm U}_2{\times} {\rm Sp}_n\\  (n\ge2)\end{matrix}$&$\bC^2{\otimes}\bC^{2n}$&$\fu_2$&$\begin{matrix}\tr\big((iz)^k\big)\\  (k=1,2)\end{matrix}$&$\begin{matrix}\tr\big((vv^*)^k\big)\\  (k=1,2)\\ |x|^2|y|^2-(\trans xy)^2\end{matrix}$&$i\tr(v^*zv)$\\
\hline
9&${\rm U}_1{\times}\text{Spin}_7$&$\bC{\otimes}\bO$&$\IM\bO\oplus\bR$&$|z_0|^2\,,\,\check z$&$\begin{matrix}|v|^2\\  |v_1|^2|v_2|^2-\big(\RE(v_1\bar v_2)\big)^2\end{matrix}$&$\RE\big(z_0(v_1\bar v_2)\big)$\\
\hline
10a&$\text{Sp}_2\times\text{Sp}_1$&$\bH^2$&$\fs\fp_2$&$\begin{matrix}\tr(z^{2k})\\  (k=1,2)\end{matrix}$&$|v|^2$&$\begin{matrix}\tr\big(zv(zv)^*\big)\\  \tr\big((zvv^*-vv^*z)^2\big)\end{matrix}$\\
\hline
10b&$\begin{matrix}{\rm Sp}_2{\times} {\rm Sp}_n\\  (n\ge2)\end{matrix}$&$\bH^2\otimes\bH^n$&$\fs\fp_2$&$\begin{matrix}\tr(z^{2k})\\  (k=1,2)\end{matrix}$&$\begin{matrix}\tr\big((vv^*)^k\big)\\  (k=1,2)\end{matrix}$&$\begin{matrix}\tr\big(zv(zv)^*\big)\\  \tr\big((zvv^*-vv^*z)^2\big)\end{matrix}$\\
\hline
\end{tabular}
\end{center}
\bigskip
\caption{Systems of invariants on $\fv\oplus\fz$\\ 
({\it Legenda.} Lines 3,5,6,9: $z_0\in\fz_0,\check z\in\check \fz$. 
Lines 7,8: find $n=1$ at l.2,7 resp.)
}

\label{invariants}
\end{table}

}


\quad

\end{document}